\documentclass[10pt]{amsart}
\usepackage{graphicx}
\usepackage{mathrsfs}
\usepackage{epsfig}
\usepackage{amsmath}
\usepackage{amsfonts}
\usepackage{amssymb}
\usepackage{amssymb,amscd}
\usepackage{tikz}
\usepackage{verbatim}
\usepackage{booktabs}
\usepackage[all]{xy}

\newtheorem{thm}{Theorem}[section]
\newtheorem{cor}[thm]{Corollary}
\newtheorem{lem}[thm]{Lemma}
\newtheorem{prop}[thm]{Proposition}
\newtheorem{conj}[thm]{Conjecture}
\theoremstyle{definition}
\newtheorem{defn}[thm]{Definition}
\theoremstyle{remark}
\newtheorem{rem}[thm]{Remark}
\numberwithin{equation}{section}

\newcommand{\hc}{\mathbf{H}^2_{\mathbb{C}}}
\newcommand{\pu}{\rm{PU}(2,1)}
\begin{document}

\date{\today}

\title[spherical CR structure on a two-cusped hyperbolic 3-manifold]{A uniformizable spherical CR structure on a two-cusped hyperbolic 3-manifold}
\author{Yueping Jiang, Jieyan Wang and Baohua Xie}
\address{School of Mathematics \\ Hunan University \\ Changsha \\ China}
\email{ypjiang@hnu.edu.cn, jywang@hnu.edu.cn, xiexbh@hnu.edu.cn}

\keywords{Complex hyperbolic space, spherical CR uniformization, triangle
groups, Ford domain,  hyperbolic 3-manifolds}

\subjclass[2010]{20H10, 57M50, 22E40, 51M10.}

\thanks{Y. Jiang was supported by NSFC (No.11631010). J. Wang was supported by NSFC (No.11701165). B. Xie was supported by NSFC (No.11871202) and  Hunan Provincial Natural Science Foundation of China (No.2018JJ3024).}

\maketitle

\begin{abstract}
Let $\langle I_{1}, I_{2}, I_{3}\rangle$ be the complex hyperbolic $(4,4,\infty)$ triangle group. In this paper we give a proof of a conjecture of Schwartz for $\langle I_{1}, I_{2}, I_{3}\rangle$. That is $\langle I_{1}, I_{2}, I_{3}\rangle$ is discrete and faithful if and only if $I_1I_3I_2I_3$ is nonelliptic. When $I_1I_3I_2I_3$ is parabolic, we show that the even subgroup $\langle I_2 I_3, I_2I_1 \rangle$ is the holonomy representation of a uniformizable spherical CR structure on  the two-cusped hyperbolic 3-manifold $s782$ in SnapPy notation.

\end{abstract}

\section{Introduction}

Let $\hc$ be the complex hyperbolic plane, $\pu$ be its holomorphic isometry group. See Section \ref{sec:background} for more details. It is well known that $\hc$ is one of the rank one symmetric spaces and $\pu$ is a semisimple Lie group.
$\hc$ can be viewed as the unit ball in $\mathbb{C}^2$ equipped with the Bergman metric. Its ideal boundary $\partial \hc$ is the 3-sphere $S^3$.
The purpose of this paper is to study the geometry of discrete subgroups of $\pu$.

Let $M$ be a 3-manifold. A \emph{spherical CR structure} on $M$ is a system of coordinate charts into $S^3$, such that the transition functions are restrictions of elements of $\pu$. Any spherical CR structure on $M$ determines a pair $(\rho,d)$, where $\rho:\pi_1(M)\longrightarrow \pu$ is the holonomy and $d:\widetilde{M}\longrightarrow S^3$ is the developing map. There is a special spherical CR structure.
A \emph{uniformizable spherical CR structure} on $M$ is a homeomorphism between $M$ and a quotient space $\Omega/\Gamma$, where $\Gamma$ is a discrete subgroup of $\pu$ and $\Omega\subset\partial \hc$ is the discontinuity region of $\Gamma$.  An interesting problem in complex hyperbolic geometry is to find  (uniformizable) spherical CR structures on hyperbolic 3-manifolds.

Geometric structures modelled on the boundary of complex hyperbolic space are rather difficult to construct. The first example  of a spherical CR  structure existing on a cusped hyperbolic 3-manifold  was discovered by Schwartz. In the work \cite{schwartz-litg}, Schwartz constructed a uniformizable spherical CR structure on the Whitehead link complement. He also constructed a closed hyperbolic 3-manifold that admits a uniformizable spherical CR structure in \cite{schwartz-real} at almost the same time.

Let $M_8$ be the complement of the figure eight knot. There are several works showed that $M_8$ admits (uniformizable) spherical CR structures.
In \cite{fal}, Falbel constructed two different representations $\rho_1, \rho_2$ of $\pi_1(M_8)$ in $\pu$, and proved that $\rho_1$ is the holonomy of a spherical CR structure on $M_8$. In \cite{fal-wang}, Falbel and Wang proved that $\rho_2$ is also the holonomy of a spherical CR structure on $M_8$.
In \cite{der-fal}, Deraux and Falbel constructed a uniformizable spherical CR structure on $M_8$ whose holonomy is $\rho_2$. In \cite{deraux-family}, Deraux proved that there is a 1-parameter family of spherical CR uniformizaitons of the figure eight knot complement. This family is in fact a deformation of the uniformization constructed in \cite{der-fal}.

Let us back to the Whitehead link complement. It admits a uniformizable spherical CR structure which is different from Schwartz's one.
In the recent work \cite{par-will2}, Parker and Will also constructed a spherical CR uniformization of the Whitehead link complement. By applying spherical CR Dehn surgery theorems to the uniformizations of the Whitehead link complement, one can get infinity many manifolds which admit uniformizable spherical CR structures.
In \cite{schwartz-cr}, Schwartz proved a spherical CR Dehn surgery theorem, and using it to the spherical CR uniformization of the Whitehead link complement constructed in \cite{schwartz-litg} to obtain infinity many closed hyperbolic 3-manifolds which admit uniformizable spherical CR structures.
In \cite{acosta}, Acosta used the spherical CR Dehn surgery theorem proved in \cite{acosta-scr} to the spherical CR uniformization of the Whitehead link complement constructed by Parker and Will in \cite{par-will2} to obtain infinity many one-cusped hyperbolic 3-manifolds which admit uniformizable spherical CR structures. In particular, the spherical CR uniformization of the complement of the figure eight knot constructed by Deraux and Falbel \cite{der-fal} is contained in this family.

There are some hyperbolic 3-manifolds described in the SnapPy census (see \cite{cdw}) which admit spherical CR structures.
In \cite{deraux-scr}, Deraux proved that the cusped hyperbolic 3-manifold $m009$ admits a uniformizable spherical CR structure whose holonomy representation was constructed by Falbel, Koseleff and Rouillier in \cite{fkr}. In \cite{mx}, Ma and Xie proved that the one-cusped hyperbolic 3-manifolds $m038$ and $s090$ admit spherical CR uniformizations.

In this paper, we show that the two-cusped hyperbolic 3-manifold $s782$ admits a uniformizable spherical CR structure. By studying the action of the even subgroup of a discrete complex hyperbolic triangle group on $\hc$, we will prove that the quotient space of its discontinuity region is homeomorphic to $s782$. That means the holonomy representation of the spherical CR uniformization of $s782$ is a triangle group.


Now let us talk about the complex hyperbolic triangle groups. Let $\Delta_{p,q,r}$ be the abstract $(p,q,r)$ reflection triangle group with the presentation
$$\langle \sigma_1, \sigma_2, \sigma_3 | \sigma^2_1=\sigma^2_2=\sigma^2_3=(\sigma_2 \sigma_3)^p=(\sigma_3 \sigma_1)^q=(\sigma_1 \sigma_2)^r=id \rangle,$$
where $p,q,r$ are positive integers or $\infty$ in which case the corresponding relation disappears.
A \emph{complex hyperbolic $(p,q,r)$ triangle group} is a representation of $\Delta_{p,q,r}$ in $\pu$, which maps the generators to complex involutions fixing complex lines in $\hc$. The study of complex hyperbolic triangle groups was begun by Goldman and Parker. In \cite{Go-P}, Goldman and Parker studied the complex $(\infty,\infty,\infty)$ triangle groups. They conjectured that a representation of $\Delta_{\infty,\infty,\infty}$ into $\pu$ is discrete and faithful if and only if the image of $\sigma_1 \sigma_2 \sigma_3$ is nonelliptic. The Goldman-Parker conjecture was proved by Schwartz in \cite{schwartz-go-p1} (or a better proof in \cite{schwartz-go-p2}). In particular, the representation with the image of $\sigma_1 \sigma_2 \sigma_3$ being parabolic is closely related with the holonomy of the spherical CR uniformizaiton of the Whitehead link complement constructed in \cite{schwartz-litg}. In the survey \cite{schwartz-icm}, a series of conjectures on complex hyperbolic triangle groups are put forward. Schwartz conjectured that:
\begin{conj}[Schwartz \cite{schwartz-icm}]\label{conj:schwartz}
Suppose that $p\leq q \leq r$. Let $\langle I_1, I_2, I_3 \rangle$ be a complex hyperbolic $(p,q,r)$ triangle group. Then $\langle I_1, I_2, I_3 \rangle$ is a discrete and faithful representation of $\Delta_{p,q,r}$ if and only if $I_1I_3I_2I_3$ and $I_1I_2I_3$ are nonelliptic. Moreover,
\begin{itemize}
  \item If $3 \leq p <10$, then $\langle I_1, I_2, I_3 \rangle$ is discrete and faithful if and only if $I_1I_3I_2I_3$ is nonelliptic.
  \item If $p>13$, then $\langle I_1, I_2, I_3 \rangle$ is discrete and faithful if and only if $I_1I_2I_3$ is nonelliptic.
\end{itemize}
\end{conj}

In a recent work \cite{par-will2}, Parker and Will proved Conjecture \ref{conj:schwartz} for complex hyperbolic $(3,3,\infty)$ groups. They also showed that when $I_1I_3I_2I_3$ is parabolic the quotient of $\hc$ by the group $\langle I_2 I_3, I_2I_1 \rangle$ is a complex hyperbolic orbifold whose boundary is a spherical CR uniformization of the Whitehead link complement.
In \cite{pwx}, Parker, Wang and Xie proved Conjecture \ref{conj:schwartz} for complex hyperbolic $(3,3,n)$ groups with $n\geq 4$. Furthermore, Acosta  \cite{acosta} showed that when $I_1I_3I_2I_3$ is parabolic the group $\langle I_2 I_3, I_2I_1 \rangle$ is the holonomy representation of a uniformizable spherical CR structure on the Dehn surgery of the Whitehead link complement on one cusp of type $(1,n-3)$.

In this paper, we give a proof of Conjecture \ref{conj:schwartz} for the complex hyperbolic $(4,4,\infty)$ triangle groups and further analyze the group when $I_1I_3I_2I_3$ is parabolic. Our result is as follows:
\begin{thm}
Let $\langle I_1, I_2, I_3 \rangle$ be a complex hyperbolic $(4,4,\infty)$ triangle group. Then $\langle I_1, I_2, I_3 \rangle$ is a discrete and faithful representation of $\Delta_{4,4,\infty}$ if and only if $I_1I_3I_2I_3$ is nonelliptic. Moreover, when $I_1I_3I_2I_3$ is parabolic the quotient of $\hc$ by the group $\langle I_2 I_3, I_2I_1 \rangle$ is a complex hyperbolic orbifold whose boundary is a spherical CR uniformization of the two-cusped hyperbolic 3-manifold $s782$ in the SnapPy census.
\end{thm}

In \cite{wyss}, Wyss-Gallifent studied the complex hyperbolic $(4,4,\infty)$ triangle groups. He discovered several discrete groups with $I_1I_3I_2I_3$ being regular elliptic of finite order and conjectured that there should be countable infinity many. Thus, it should be very interesting to know what is the manifold at infinity for the group with $I_1I_3I_2I_3$ being regular elliptic of finite order. Motivated by the work of Acosta \cite{acosta}, we guess that the manifold is the Dehn surgery of the hyperbolic 3-manifold $s782$ on one cusp. We will treat this problem in another paper.

Our method is to construct Ford domains for the triangle groups acting on $\hc$. The space of complex hyperbolic $(4,4,\infty)$ triangle groups $\langle I_1, I_2, I_3 \rangle$ is parameterized by the angular $\theta\in[0,\pi/2)$ (See Section \ref{sec:parameter}).
Let $S=I_2I_3$, $T=I_2I_1$ and $\Gamma=\langle S, T \rangle$. Here $S$ is regular elliptic of order 4, and $T$ is parabolic fixing the point at infinity.
For each group in the parameter space, the Ford domain $D$ is the intersection of the closures of the exteriors of the isometric spheres for the elements $S$, $S^{-1}$, $S^2$,  $(S^{-1}T)^2$ and their conjugations by the powers of $T$. The combination of $D$ is the same except when $I_1I_3I_2I_3$ is parabolic, in which case there are additional parabolic fixed points. $D$ is preserved by the subgroup $\langle T \rangle$ and is a fundamental domain for the cosets of $\langle T \rangle$ in $\Gamma$. Its ideal boundary $\partial_{\infty}D$ is the complement of a tubular neighborhood of the $T$-invariant $\mathbb{R}$-circle (or horotube defined in \cite{schwartz-cr}). By intersecting $\partial_{\infty}D$ with a fundamental domain for $\langle T \rangle$ acting on $\partial\hc$, we obtain a fundamental domain for $\Gamma$  acting on its discontinuity region. See section \ref{sec:ford}.

When $I_1I_3I_2I_3$ is parabolic, that is $\theta=\pi/3$, by studying the combinatorial properties of the fundamental domain for $\Gamma$ acting on its discontinuity region $\Omega(\Gamma)$, we prove that the quotient $\Omega(\Gamma)/\Gamma$ is homeomorphic to the two-cusped hyperbolic 3-manifold $s782$. In this case, there are four additional parabolic fixed points fixed by $T^{-1}S^2$, $S^2T^{-1}$, $ST^{-1}S$ and $T^{-1}ST^{-1}ST$, except the point at infinity which is the fixed point of $T$. See Section \ref{sec:manifold}.

\subsection*{Acknowledgements\label{ackowledgements}} We  thank Jiming Ma for his help in the proof of Theorem \ref{thm:s782}. The third author also would like to thank Jiming Ma  for numerous helpful discussions on  complex hyperbolic geometry  during his visiting at Fudan University.

\section{Background}\label{sec:background}
The purpose of this section is to introduce briefly complex hyperbolic geometry. One can refer to Goldman's book \cite{Go} for more details.

\subsection{Complex hyperbolic plane}

Let $\langle {\bf{z}}, {\bf{w}} \rangle={\bf{w}^{*}}H{\bf{z}}$ be the Hermitian form on ${\mathbb{C}}^3$ associated to $H$, where $H$ is the Hermitian matrix
$$
H=\left[
  \begin{array}{ccc}
    0 & 0 & 1 \\
    0 & 1 & 0 \\
    1 & 0 & 0 \\
  \end{array}
\right].
$$
Then ${\mathbb{C}}^3$ is the union of negative cone $V_{-}$, null cone $V_{0}$ and positive cone $V_{+}$, where
\begin{eqnarray*}
  V_{-} &=& \left\{ {\bf{z}}\in {\mathbb{C}}^3-\{0\} : \langle {\bf{z}}, {\bf{z}} \rangle <0 \right\}, \\
  V_{0} &=& \left\{ {\bf{z}}\in {\mathbb{C}}^3-\{0\} : \langle {\bf{z}}, {\bf{z}} \rangle =0 \right\}, \\
  V_{+} &=& \left\{ {\bf{z}}\in {\mathbb{C}}^3-\{0\} : \langle {\bf{z}}, {\bf{z}} \rangle >0 \right\}.
\end{eqnarray*}

\begin{defn}
Let $P: {\mathbb{C}}^3-\{0\} \rightarrow {\mathbb{C}P^2}$ be the projectivization map.
Then the \emph{complex hyperbolic plane} $\hc$ is defined to be $P(V_{-})$ and its {boundary} $\partial \hc$ is defined to be $P(V_{0})$.
This is the \emph{Siegel domain model} of $\hc$.
Let $d(u,v)$ be the distance between two points $u,v \in \hc$.
Then the \emph{Bergman metric} on complex hyperbolic plane is given by the distance formula
$$
\cosh^2\left(\frac{d(u,v)}{2}\right)=\frac{\langle {\bf{u}}, {\bf{v}} \rangle\langle {\bf{v}}, {\bf{u}} \rangle}{\langle {\bf{u}}, {\bf{u}} \rangle \langle {\bf{v}}, {\bf{v}} \rangle},
$$
where ${\bf{u}}, {\bf{v}} \in {\mathbb{C}}^3$ are lifts of $u,v$.
\end{defn}

There is another model of $\hc$.
\begin{defn}\label{def:cayley}
The \emph{ball model} of $\hc$ is the unit ball in $\mathbb{C}^2$, which is given by the Hermition matrix $J=diag(1,1,-1)$. In this model, $\partial\hc$ is then the 3-dimensional sphere $S^3\subset \mathbb{C}^2$.
The \emph{Cayley transform} $C$ is given by
$$
C=\frac{1}{\sqrt{2}}\left(
    \begin{array}{ccc}
      1 & 0 & 1 \\
      0 & \sqrt{2} & 0 \\
      1 & 0 & -1 \\
    \end{array}
  \right).
$$
It satisfies $C^{*}HC=J$  and interchanges the Siegel domain model and the ball model of $\hc$.
\end{defn}

Let $\mathcal{N}=\mathbb{C}\times \mathbb{R}$ be the Heisenberg group with product
$$
[z,t]\cdot [\zeta,\nu]=[z+\zeta,t+\nu-2{\rm{Im}}(\bar{z}\zeta)].
$$
Then, in the Siegel domain model of $\hc$, the boundary of complex hyperbolic plane $\partial \hc$ can be identified to the union $\mathcal{N}\cup \{q_{\infty}\}$, where $q_{\infty}$ is the point at infinity.
The \emph{standard lift} of $q_{\infty}$ and $q=[z,t]\in\mathcal{N}$ in $\mathbb{C}^3$ are
\begin{equation}\label{eq:lift}
  {\bf{q}_{\infty}}=\left[
                    \begin{array}{c}
                      1 \\
                      0 \\
                      0 \\
                    \end{array}
                  \right],\quad
{\bf{q}}=\left[
           \begin{array}{c}
             \frac{-|z|^2+it}{2} \\
             z \\
             1 \\
           \end{array}
         \right].
\end{equation}
The closure of complex hyperbolic plane $\hc \cup \partial \hc$ can be identified to the union of ${\mathcal{N}}\times{\mathbb{R}_{\geq 0}}$ and $\{q_{\infty}\}$.
Any point $q=(z,t,u)\in{\mathcal{N}}\times{\mathbb{R}_{\geq0}}$ has the standard lift
$$
{\bf{q}}=\left[
           \begin{array}{c}
             \frac{-|z|^2-u+it}{2} \\
             z \\
             1 \\
           \end{array}
         \right].
$$
Here $(z,t,u)$ is called the\emph{ horospherical coordinates} of $\hc \cup \partial \hc$.

\begin{defn}
The \emph{Cygan metric} $d_{\textrm{Cyg}}$ on $\partial \hc -\{q_{\infty}\}$ is defined to be
\begin{equation}\label{eq:cygan-metric}
  d_{\textrm{Cyg}}(p,q)=|2\langle {\bf{p}}, {\bf{q}} \rangle|^{1/2}=\left| |z-w|^2-i(t-s+2{\rm{Im}}(z\bar{w})) \right|^{1/2},
\end{equation}
where $p=[z,t]$ and $q=[w,s]$.

The Cygan metric satisfies the properties of a distance.
The \emph{extended Cygan metric} on $\hc$ is given by the formula
\begin{equation}\label{eq:cygan-metric-extend}
  d_{\textrm{Cyg}}(p,q)=\left| |z-w|^2+|u-v|-i(t-s+2{\rm{Im}}(z\bar{w})) \right|^{1/2},
\end{equation}
where $p=(z,t,u)$ and $q=(w,s,v)$.
\end{defn}
 The formula $d_{\textrm{Cyg}}(p,q)=|2\langle {\bf{p}}, {\bf{q}} \rangle|^{1/2}$ remains valid even if one of $p$ and $q$ lies on $\partial \hc$.
A\emph{ Cygan sphere} is  a sphere for the extended Cygan distance.

There are two kinds of 2-dimensional totally real totally geodesic subspaces of $\hc$: complex lines and Lagrangian planes.

\begin{defn}
Let $\textbf{v}^{\perp}$ be the orthogonal space of $\textbf{v}\in V_{+}$ with respect to the Hermitian form. The intersection of the projective line $P(\textbf{v}^{\perp})$ with $\hc$ is called a \emph{complex line}. The vector $\textbf{v}$ is its \emph{polar vector}.
\end{defn}

The ideal boundary of a complex line on $\partial\hc$ is called a \emph{$\mathbb{C}$-circle}. In the Heisenberg group, $\mathbb{C}$-circles are either vertical lines or ellipses whose projections on the $z$-plane are circles.

Let $\mathbf{H}^2_{\mathbb{R}}=\{ (x_1, x_2) \in \hc : x_1, x_2 \in \mathbb{R} \}$ be the set of real points. $\mathbf{H}^2_{\mathbb{R}}$ is a Lagrangian plane.
All the Lagrangian planes are the images of $\mathbf{H}^2_{\mathbb{R}}$ by isometries of $\hc$. The ideal boundary of a Lagrangian plane is called a \emph{$\mathbb{R}$-circle}. In the Heisenberg group, $\mathbb{R}$-circles are either straight lines or lemniscate curves whose projections on the $z$-plane are figure eight.

\subsection{Isometries}

Let ${\rm{SU}}(2,1)$ be the special unitary matrix preserving the Hermitian form. Then the projective unitary group $\pu={\rm{SU}}(2,1)/\{I, \omega I, \omega^2 I\}$ is the holomorphic isometry group of $\hc$, where $\omega=(-1+i\sqrt{3})/2$ is a primite cubic root of unit.
Observe that complex conjugation also preserves the Hermitian form. Thus the full isometry group of $\hc$ is generated by $\pu$ and complex conjugation.

\begin{defn}
Any isometry $g\in\pu$ is \emph{loxodromic} if it has exactly two fixed points on $\partial \hc$. $g$ is \emph{parabolic} if it has exactly one fixed point on $\partial \hc$.
$g$ is \emph{elliptic} if it has at least one fixed point in $\hc$.
\end{defn}

The types of isometries can be determined by the traces of their matrix realizations, see Theorem 6.2.4 of Goldman \cite{Go}. Now suppose that $A\in{\rm{SU}}(2,1)$ has real trace. Then $A$ is elliptic if $-1\leq{\rm{tr}(A)}<3$.
Moreover, $A$ is unipotent if ${\rm{tr}(A)}=3$. In particular, if ${\rm{tr}(A)}=-1,0,1$, $A$ is elliptic of order 2, 3, 4 respectively.

There is a special class of elliptic elements of order two.
\begin{defn}
The \emph{complex involution} on complex line $C$ with polar vector ${\bf{n}}$ is given by the following formula:
\begin{equation}\label{eq:involution}
  I_{C}({\bf{z}})=-{\bf{z}}+\frac{2\langle {\bf{z}}, {\bf{n}} \rangle}{\langle {\bf{n}}, {\bf{n}} \rangle} {\bf{n}}.
\end{equation}
It is obvious that $I_{C}$ is a holomorphic isometry fixing the complex line $C$.
\end{defn}

There is a special class of unipotent elements in $\pu$.
\begin{defn}
A left \emph{Heisenberg translation} associated to $[z,t]\in\mathcal{N}$ is given by the matrix
\begin{equation}\label{eq:translation}
  T_{[z,t]}=\left[
            \begin{array}{ccc}
              1 & -\bar{z} & \frac{-|z|^2+it}{2} \\
              0 & 1 & z \\
              0 & 0 & 1 \\
            \end{array}
          \right].
\end{equation}
It is obvious that $T_{[z,t]}$ fixes $q_{\infty}$ and maps $[0,0]\in\mathcal{N}$ to $[z,t]$.
\end{defn}

\subsection{Isometric spheres and Ford polyhedron}

Suppose that $g=(g_{ij})^3_{i,j=1}\in\pu$ does not fix $q_{\infty}$. Then it is obvious that $g_{31}\neq 0$.
We first recall the definition of isometric spheres and relevant properties; see   for instance\cite{par}.
\begin{defn}
The \emph{isometric sphere} of $g$, denoted by $\mathcal{I}(g)$, is the set
\begin{equation}\label{eq:isom-sphere}
  \mathcal{I}(g)=\{ p \in {\hc \cup \partial\hc} : |\langle {\bf{p}}, {\bf{q}}_{\infty} \rangle | = |\langle {\bf{p}}, g^{-1}({\bf{q}}_{\infty}) \rangle| \}.
\end{equation}
\end{defn}
The isometric sphere $\mathcal{I}(g)$ is the Cygan sphere with center
$$g^{-1}({\bf{q}}_{\infty})=\left[{\overline{g_{32}}}/{\overline{g_{31}}},2{\rm{Im}}({\overline{g_{33}}}/{\overline{g_{31}}})\right]$$
and radius $r_g=\sqrt{{2}/{|g_{31}|}}$.

The \emph{exterior} of $\mathcal{I}(g)$ is the set
\begin{equation}\label{eq:exterior}
  \{ p \in {\hc \cup \partial\hc} : |\langle {\bf{p}}, {\bf{q}}_{\infty} \rangle | > |\langle {\bf{p}}, g^{-1}({\bf{q}}_{\infty}) \rangle| \}.
\end{equation}
The \emph{interior} of $\mathcal{I}(g)$ is the set
\begin{equation}\label{eq:interior}
  \{ p \in {\hc \cup \partial\hc} : |\langle {\bf{p}}, {\bf{q}}_{\infty} \rangle | < |\langle {\bf{p}}, g^{-1}({\bf{q}}_{\infty}) \rangle| \}.
\end{equation}

The isometric spheres are paired as the following.
\begin{lem}[ \cite{Go}, Section 5.4.5]\label{lem:goldman}
Let $g$ be an element in $\pu$ which does not fix $q_{\infty}$. Then $g$ maps $\mathcal{I}(g)$ to $\mathcal{I}(g^{-1})$,
and the exterior of $\mathcal{I}(g)$ to the interior of $\mathcal{I}(g^{-1})$.
Besides, for any unipotent transformation $h\in\pu$ fixing $q_{\infty}$, we have $\mathcal{I}(g)=\mathcal{I}(hg)$.
\end{lem}

Since isometric spheres are Cygan spheres, we now recall some facts about Cygan spheres.
Let $S_{[0,0]}(r)$ be the Cygan sphere with center $[0,0]$ and radius $r>0$. Then
\begin{equation}\label{eq:cygan-sphere}
  S_{[0,0]}(r)=\left\{ (z,t,u)\in\hc \cup \partial\hc: (|z|^2+u)^2+t^2=r^4 \right\}.
\end{equation}

The geographic coordinates for Cygan sphere will play an important role in our calculation;  see Section 2.5 of \cite{par-will2}.
\begin{defn}
The \emph{geographic coordinates} $(\alpha,\beta,w)$ of $q=q(\alpha,\beta,w)\in S_{[0,0]}(r)$ is given by the lift
\begin{equation}\label{eq:geog-coor}
  {\bf{q}}={\bf q}(\alpha,\beta,w)=\left[
           \begin{array}{c}
             -r^2e^{-i\alpha}/2 \\
             rwe^{i(-\alpha/2+\beta)} \\
             1 \\
           \end{array}
         \right],
\end{equation}
where $\alpha\in [-\pi/2,\pi/2]$, $\beta\in [0, \pi)$ and $w\in [-\sqrt{\cos(\alpha)},\sqrt{\cos(\alpha)}]$.
The ideal boundary of $S_{[0,0]}(r)$ on $\partial\hc$ are the points with $w=\pm\sqrt{\cos(\alpha)}$.
\end{defn}

We are interested in the intersection of Cygan spheres.
\begin{prop}[  \cite{par-will2}, Proposition 2.10]\label{prop:connect}
The intersection of two Cygan spheres is connected.
\end{prop}

\begin{rem} This intersection is often called \emph{Giraud disk}.
\end{rem}
The following property should be useful to describe the intersection of Cygan spheres.
\begin{prop}[\cite{par-will2}, Proposition 2.12]
Let $S_{[0,0]}(r)$ be a Cygan sphere with {geographic coordinates} $(\alpha,\beta,w)$.
\begin{enumerate}
  \item The level sets of $\alpha$ are complex lines, called slices of $S_{[0,0]}(r)$.
  \item The level sets of $\beta$ are Lagrangian planes, called meridians of $S_{[0,0]}(r)$.
  \item The set of points with $w=0$ is the spine of $S_{[0,0]}(r)$. It is a geodesic contained in every meridian.
\end{enumerate}
\end{prop}

A central work in this paper is to construct a polyhedron for a finitely generated subgroup of $\pu$.
\begin{defn}
Let $G$ be a discrete subgroup of $\pu$. The \emph{Ford polyhedron} $D_{G}$ for $G$ is the set
$$
D_{G}=\left\{ p\in{\hc\cup\partial\hc} : |\langle {\bf{p}}, {\bf{q}}_{\infty} \rangle | \geq |\langle {\bf{p}}, g^{-1}{\bf{q}}_{\infty} \rangle| ~\textrm{for all}~g\in G~\textrm{with}~g(q_{\infty})\neq q_{\infty} \right\}.
$$
\end{defn}
That is to say $D_{G}$ is the intersection of closures of the exteriors of all the isometric spheres for elements of $G$ which do not fix $q_{\infty}$.
In fact, Ford polyhedron is the limit of Dirichlet polyhedra as the center point goes to $q_{\infty}$.

\section{The parameter space of complex hyperbolic $(4,4,\infty)$ triangle groups}\label{sec:parameter}
In this section, we give a parameter space of the complex hyperbolic $(4,4,\infty)$ triangle groups.

Let $\theta\in[0,\pi/2)$. Let $C_1, C_2,C_3$ be three complex lines in complex hyperbolic space $\hc$ with polar vectors $\bf{n_1},\bf{n_2},\bf{n_3}$, respectively. By conjugating elements in $\pu$, one can normalize that $\partial C_3=\{ [z,0] \in\mathcal{N} : |z|=\sqrt{2}\}$. That is the circle in the $z$-plane of the Heisenberg group with center the origin and radius $\sqrt{2}$. Then, up to the complex conjugation and rotations about the $t$-axis of the Heisenberg group, the $\mathbb{C}$-circles $\partial C_1$ and $\partial C_2$ can be normalized to be the sets $\partial C_1=\{[-e^{-i\theta},t]\in\mathcal{N}: t\in\mathbb{R} \}$ and $\partial C_2=\{[e^{i\theta},t]\in\mathcal{N}: t\in\mathbb{R} \}$. That is, $\partial C_1$ (respectively $\partial C_2$) is the vertical line whose projection on the $z$-plane of the Heisenberg group is the point $-e^{-i\theta}$ (respectively $e^{i\theta}$). Thus the polar vectors of the complex lines can be written as follows
$$
\bf{n_1}=\left[
  \begin{array}{c}
    e^{i\theta} \\
    1 \\
    0 \\
  \end{array}
\right],\quad
\bf{n_2}=\left[
  \begin{array}{c}
    -e^{-i\theta} \\
    1 \\
    0 \\
  \end{array}
\right],\quad
\bf{n_3}=\left[
  \begin{array}{c}
    1 \\
    0 \\
    1 \\
  \end{array}
\right].
$$

 Note that these two $\mathbb{C}$-circles $\partial C_1$  and $\partial C_2$ coincide with each other   if $\theta=\pi/2$.
According to (\ref{eq:involution}), the complex involutions $I_1$, $I_2$ and $I_3$ on the complex lines are given as
$$
I_1=\left[
  \begin{array}{ccc}
    -1 & 2e^{i\theta} & 2 \\
    0 & 1 & 2e^{-i\theta} \\
    0 & 0 & -1 \\
  \end{array}
\right], \quad
I_2=\left[
  \begin{array}{ccc}
    -1 & -2e^{-i\theta} & 2 \\
    0 & 1 & -2e^{i\theta} \\
    0 & 0 & -1 \\
  \end{array}
\right],
$$
$$
I_3=\left[
  \begin{array}{ccc}
    0 & 0 & 1 \\
    0 & -1 & 0 \\
    1 & 0 & 0 \\
  \end{array}
\right].
$$

\begin{prop}\label{prop:trianle}
Let $\theta\in[0,\pi/2)$ and $I_1$, $I_2$, $I_3$ be defined as above. Then $\langle I_1, I_2, I_3 \rangle$ is a complex hyperbolic $(4,4,\infty)$ triangle group. Furthermore, the element $I_1I_3I_2I_3$ is nonelliptic if and only if $0\leq \theta \leq \pi/3$.
\end{prop}
\begin{proof}
By computing the products of two involutions, we have
$$
I_2 I_3=\left[
  \begin{array}{ccc}
    2 & 2e^{-i\theta} & -1 \\
    -2e^{i\theta} & -1 & 0 \\
    -1 & 0 & 0 \\
  \end{array}
\right],\quad
I_3 I_1=\left[
  \begin{array}{ccc}
    0 & 0 & -1 \\
    0 & -1 & -2e^{-i\theta} \\
    -1 & e^{i\theta} & 2 \\
  \end{array}
\right],
$$
$$
I_2 I_1=\left[
  \begin{array}{ccc}
    1 & -4\cos(\theta) & -4(1+e^{-2i\theta}) \\
    0 & 1 & 4\cos(\theta) \\
    0 & 0 & 1 \\
  \end{array}
\right].
$$
It is easy to verify that $I_2I_3$ and $I_3I_1$ are elliptic of order $4$ and $I_2I_1$ is unipotent.
Thus $\langle I_1, I_2, I_3 \rangle$ is a complex hyperbolic $(4,4,\infty)$ triangle group.

Since the trace of $I_1I_3I_2I_3$ is ${\rm{tr}}(I_1I_3I_2I_3)=7+8\cos(2\theta)$,
the element $I_1I_3I_2I_3$ is elliptic if and only if
$$
-1\leq {\rm{tr}}(I_1I_3I_2I_3)=7+8\cos(2\theta) <3,
$$
that is $\pi/3 <\theta \leq \pi/2$. Thus $I_1I_3I_2I_3$ is nonelliptic if and only if $0\leq \theta \leq \pi/3$.
Moreover, when $\theta=\pi/3$, the element $I_1I_3I_2I_3$ is parabolic.
\end{proof}

If $\theta=0$, then $\langle I_1, I_2, I_3 \rangle$ preserves the Lagrangian plane whose ideal boundary is the $x$-axis of the Heisenberg group.
Thus $\langle I_1, I_2, I_3 \rangle$ is obviously a discrete subgroup of $\pu$.
If $\theta=\pi/3$, we have the following corollary.
\begin{cor}
Let $\theta=\pi/3$. Then $\langle I_1, I_2, I_3 \rangle$ is a discrete subgroup of $\pu$.
\end{cor}
\begin{proof}
Let $\theta=\pi/3$. Then
$$
I_2 I_3=\left[
  \begin{array}{ccc}
    2 & 1-i\sqrt{3} & -1 \\
    -1-i\sqrt{3} & -1 & 0 \\
    -1 & 0 & 0 \\
  \end{array}
\right],\quad
I_3 I_1=\left[
  \begin{array}{ccc}
    0 & 0 & -1 \\
    0 & -1 & -1+i\sqrt{3} \\
    -1 & 1+i\sqrt{3} & 2 \\
  \end{array}
\right],
$$
$$
I_2 I_1=\left[
  \begin{array}{ccc}
    1 & -2 & -2+2i\sqrt{3} \\
    0 & 1 & 2 \\
    0 & 0 & 1 \\
  \end{array}
\right].
$$
Observe that $I_2I_3$ and $I_3I_1$ are contained in the Eisenstein-Picard modular group $\rm{PU}(2,1;\mathbb{Z}[\omega])$, where $\omega=(-1+i\sqrt{3})/2$ is a primite cubic root of the unit. (See for example \cite{fal-par} for more details about the Eisenstein-Picard modular group).
Since $\rm{PU}(2,1;\mathbb{Z}[\omega])$ is discrete, $\langle I_2I_3, I_3I_1 \rangle$ is also discrete. Therefore, $\langle I_1, I_2, I_3 \rangle$ is discrete.
\end{proof}

\section{The Ford domain}\label{sec:ford}

For $\theta\in[0,\pi/3]$, let $S=I_2I_3$ and $T=I_2I_1$. Then $\Gamma=\langle S, T \rangle$ is a subgroup of $\langle I_1, I_2, I_3 \rangle$ of index two.
In this section, we will mainly prove that $\Gamma$ is discrete.
Our method is as follows: firstly, construct a candidate Ford domain $D$ (see Definition \ref{domain:D}),
then applying the Poincar\'{e} polyhedron theorem to show that $D$ is a fundamental domain for the cosets of $\langle T \rangle$ in $\Gamma$,
and as well $\Gamma$ is discrete.

\begin{defn}\label{def:isometric}
For $k\in\mathbb{Z}$, let
\begin{itemize}
  \item  $\mathcal{I}_k^{+}$  be the isometric sphere $\mathcal{I}(T^kST^{-k})=T^k\mathcal{I}(S)$,
  \item  $\mathcal{I}_k^{-}$  be the isometric sphere $\mathcal{I}(T^kS^{-1}T^{-k})=T^k\mathcal{I}(S^{-1})$,
  \item $\mathcal{I}_k^{\star}$  be the isometric sphere $\mathcal{I}(T^kS^2T^{-k})=T^k\mathcal{I}(S^2)$,
  \item $\mathcal{I}_k^{\diamond}$  be the isometric sphere $\mathcal{I}(T^k(S^{-1}T)^2T^{-k})=T^k\mathcal{I}((S^{-1}T)^2)$.
\end{itemize}
\end{defn}
Note that   $S$ and $S^{-1}T$  both have order 4,  so $S^2=S^{-2}$, $(S^{-1}T)^2=(S^{-1}T)^{-2}$.
The centers and radii of the isometric spheres $\mathcal{I}_{k}^{+}$, $\mathcal{I}_{k}^{-}$, $\mathcal{I}_{k}^{\star}$ and $\mathcal{I}_{k}^{\diamond}$ are listed in the following table.

\begin{table}[!htbp]
  \centering
\begin{tabular}{ccc}
\toprule
\textbf{Isometric sphere} & \textbf{Center  } & \textbf{radius}\\
\midrule
$\mathcal{I}_{k}^{+}$&$[4k\cos(\theta),8k\sin(2\theta)]$&$\sqrt{2}$\\
$\mathcal{I}_{k}^{-}$&$[4k\cos(\theta)+2e^{i\theta},0]$&$\sqrt{2}$\\
$\mathcal{I}_{k}^{\star}$&$[4k\cos(\theta)+e^{i\theta},4k\sin(2\theta)]$&$1$\\
$\mathcal{I}_{k}^{\diamond}$&$[4k\cos(\theta)-e^{-i\theta},4k\sin(2\theta)]$&$1$\\
\bottomrule
\end{tabular}

  \label{tab:isometric spheres}
\end{table}

\begin{prop}\label{prop:symmetry}
Let $k\in\mathbb{Z}$.
\begin{enumerate}
  \item There is an antiholomophic involution $\tau$ such that $\tau(\mathcal{I}_{k}^{+})=\mathcal{I}_{-k}^{+}$,
$\tau(\mathcal{I}_{k}^{-})=\mathcal{I}_{-k-1}^{-}$ and $\tau(\mathcal{I}_{k}^{\star})=\mathcal{I}_{-k}^{\diamond}$.
  \item The complex involution $I_2$ interchanges $\mathcal{I}_{k}^{\star}$ and $\mathcal{I}_{-k}^{\star}$,
        interchanges $\mathcal{I}_k^{+}$ and $\mathcal{I}_{-k}^{-}$, and interchanges $\mathcal{I}_k^{\diamond}$ and $\mathcal{I}_{-k+1}^{\diamond}$.
\end{enumerate}
\end{prop}
\begin{proof}
(1) Let $\tau : {\mathbb{C}}^3 \longrightarrow {\mathbb{C}}^3$ be given as follows:
$$
\tau : \left[
         \begin{array}{c}
           z_1 \\
           z_2 \\
           z_3 \\
         \end{array}
       \right] \longmapsto
       \left[
         \begin{array}{c}
           \bar{z}_1 \\
           -\bar{z}_2 \\
           \bar{z}_3 \\
         \end{array}
       \right].
$$
Then ${\tau}^2$ is the identity. It is easy to see that $\tau$ fixes the polar vector ${\bf{n}}_3$, and interchanges the polar vectors ${\bf{n}}_1$ and ${\bf{n}}_2$.
Thus $\tau$ conjugates $I_3$ to itself, $I_1$ to $I_2$ and vice versa.
Therefore $\tau$ conjugates $T$ to $T^{-1}$, $S$ to $T^{-1}S$, $S^{-1}$ to $S^{-1}T$, and $S^2$ to $(T^{-1}S)^2=(S^{-1}T)^2$.
This implies that $\tau(\mathcal{I}_{k}^{+})=\mathcal{I}_{-k}^{+}$,
$\tau(\mathcal{I}_{k}^{-})=\mathcal{I}_{-k-1}^{-}$ and $\tau(\mathcal{I}_{k}^{\star})=\mathcal{I}_{-k}^{\diamond}$.

(2) The statement is easily obtained by the facts $I_2SI_2=S^{-1}$ and $I_2TI_2=T^{-1}$.
\end{proof}

Before we consider the intersections of two isometric spheres, we would like to give a useful technical lemma.
Suppose that $q \in \mathcal{I}_{0}^{+}$. Then by (\ref{eq:geog-coor}) the geographic coordinates of $q=q(\alpha,\beta,w)$ is given by the lift
\begin{equation}\label{eq:geog-coor-plus-0}
{\bf{q}}={\bf q}(\alpha,\beta,w)=\left[
           \begin{array}{c}
             -e^{-i\alpha} \\
             \sqrt{2}we^{i(-\alpha/2+\beta)} \\
             1 \\
           \end{array}
         \right]
\end{equation}
where $\alpha\in [-\pi/2,\pi/2]$, $\beta\in [0, \pi)$ and $w\in [-\sqrt{\cos(\alpha)},\sqrt{\cos(\alpha)}]$.

\begin{defn}\label{def:functions}
Let $(\alpha,\beta,w)$ be the geographic coordinates of $\mathcal{I}_{0}^{+}$. We define the following functions.
\begin{equation*}
  f_{0}^{\star}(\alpha,\beta,w)= 2w^2+1+\cos(\alpha)-\sqrt{2}w\cos(-\alpha/2+\beta-\theta)-2\sqrt{2}w\cos(\alpha/2+\beta-\theta),
\end{equation*}
\begin{equation*}
  f_{0}^{-}(\alpha,\beta,w)= 2w^2+1+\cos(\alpha)-\sqrt{2}w\cos(\alpha/2+\beta-\theta)-2\sqrt{2}w\cos(-\alpha/2+\beta-\theta),
\end{equation*}
\begin{equation*}
  f_{-1}^{-}(\alpha,\beta,w)= 2w^2+1+\cos(\alpha)+\sqrt{2}w\cos(\alpha/2+\beta+\theta)+2\sqrt{2}w\cos(-\alpha/2+\beta+\theta).
\end{equation*}
\end{defn}

\begin{lem}\label{lem:functions}
Suppose that $\theta\in[0,\pi/3]$. Let $f_{0}^{\star}(\alpha,\beta,w)$, $f_{0}^{-}(\alpha,\beta,w)$ and $f_{-1}^{-}(\alpha,\beta,w)$ be the functions defined in Definition \ref{def:functions}. Suppose that $q \in \mathcal{I}_{0}^{+}$. Then we have the following properties.
\begin{enumerate}
  \item $q$ lies on $\mathcal{I}_{0}^{\star}$ (resp. in its interior or exterior) if and only if $f_{0}^{\star}(\alpha,\beta,w)=0$ (resp. negative or positive);
  \item $q$ lies on $\mathcal{I}_{0}^{-}$ (resp. in its interior or exterior) if and only if $f_{0}^{-}(\alpha,\beta,w)=0$ (resp. negative or positive);
  \item $q$ lies on $\mathcal{I}_{-1}^{-}$ (resp. in its interior or exterior) if and only if $f_{-1}^{-}(\alpha,\beta,w)=0$ (resp. negative or positive).
\end{enumerate}
\end{lem}

\begin{proof}
(1) Any point $q \in \mathcal{I}_{0}^{+}$ lies on  $\mathcal{I}_{0}^{\star}$ (resp. in its interior or exterior) if and only if the Cygan distance between $q$ and the center of $ \mathcal{I}_{0}^{\star}$ is 1 (resp. less than 1 or greater than 1).
Using \ref{eq:geog-coor-plus-0}, the difference between the Cygan distance from $q$ to the center of $ \mathcal{I}_{0}^{\star}$ and 1 is
\begin{eqnarray*}
  \lefteqn{\left| 2 \left\langle {\bf{q}}, \left[
                                        \begin{array}{c}
                                          -1/2 \\
                                          e^{i\theta} \\
                                          1 \\
                                        \end{array}
                                      \right] \right\rangle
   \right|^2-1} \\
   &=& 4\left| -e^{-i\alpha}+ \sqrt{2}we^{i(-\alpha/2+\beta-\theta)}-1/2 \right|^2-1 \\
   &=& 4 \left( 2w^2+1+\cos(\alpha)-\sqrt{2}w\cos(-\alpha/2+\beta-\theta)-2\sqrt{2}w\cos(\alpha/2+\beta-\theta) \right) \\
   &=& 4f_{0}^{\star}(\alpha,\beta,w).
\end{eqnarray*}
Hence, $q$ lies on $\mathcal{I}_{0}^{\star}$ (resp. in its interior or exterior) if and only if $f_{0}^{\star}(\alpha,\beta,w)=0$ (resp. negative or positive).

(2) Similarly, the difference between the Cygan distance from $q$ to the center of $ \mathcal{I}_{0}^{-}$ and its radius is
\begin{eqnarray*}
\lefteqn{\left| 2 \left\langle {\bf{q}}, \left[
                                        \begin{array}{c}
                                          -2 \\
                                          2e^{i\theta} \\
                                          1 \\
                                        \end{array}
                                      \right] \right\rangle
   \right|^2-4}\\
 &=& 4\left| -e^{-i\alpha}+2 \sqrt{2}we^{i(-\alpha/2+\beta-\theta)}-2 \right|^2-4 \\
  &=& 16 \left( 2w^2+1+\cos(\alpha)-\sqrt{2}w\cos(\alpha/2+\beta-\theta)-2\sqrt{2}w\cos(-\alpha/2+\beta-\theta) \right) \\
  &=& 16 f_{0}^{-}(\alpha,\beta,w).
\end{eqnarray*}

(3) Similarly, the difference between the Cygan distance from $q$ to the center of $ \mathcal{I}_{-1}^{-}$ and its radius is
\begin{eqnarray*}
  \lefteqn{\left| 2 \left\langle {\bf{q}}, \left[
                                        \begin{array}{c}
                                          -2 \\
                                          -2e^{-i\theta} \\
                                          1 \\
                                        \end{array}
                                      \right] \right\rangle
   \right|^2-4} \\
   &=& 4\left| -e^{-i\alpha}-2 \sqrt{2}we^{i(-\alpha/2+\beta+\theta)}-2 \right|^2-4 \\
   &=& 16 \left( 2w^2+1+\cos(\alpha)+\sqrt{2}w\cos(\alpha/2+\beta+\theta)+2\sqrt{2}w\cos(-\alpha/2+\beta+\theta) \right) \\
   &=& 16 f_{-1}^{-}(\alpha,\beta,w).
\end{eqnarray*}
\end{proof}

Now, we begin to study the intersections of isometric spheres.

\begin{prop}
Suppose that $\theta\in[0,\pi/3]$, then each pair of the isometric spheres in \{ $\mathcal{I}_k^{+} : k\in\mathbb{Z}$ \} are disjoint in $\hc \cup \partial\hc$.
\end{prop}
\begin{proof}
It suffices to show that $\mathcal{I}_{0}^{+}$ and $\mathcal{I}_k^{+}$ are disjoint for $|k|\geq1$.
Observe that $T$ is a Heisenberg translation associated with $[4\cos(\theta),8\sin(2\theta)]$.
Since the isometric sphere $\mathcal{I}_{0}^{+}$ has center $[0,0]$ and radius $\sqrt{2}$, the isometric sphere
$\mathcal{I}_k^{+}$ has center $[4k\cos(\theta), 8k\sin(2\theta)]$ and radius $\sqrt{2}$.
According to the Cygan metric given in (\ref{eq:cygan-metric}), the Cygan distance  between the centers of $\mathcal{I}_{0}^{+}$ and $\mathcal{I}_k^{+}$ is
$$
4\sqrt{|k| \cos(\theta)}|k\cos(\theta)-i\sin(\theta)|^{1/2}\geq 4\sqrt{\cos(\theta)} \geq 2\sqrt{2}.
$$
Thus the Cygan distance between the centers of $\mathcal{I}_{0}^{+}$ and $\mathcal{I}_k^{+}$ is bigger than the sum of the radii, except when $k=\pm 1$ and $\theta=\pi/3$.
This implies that $\mathcal{I}_{0}^{+}$ and $\mathcal{I}_k^{+}$ are disjoint for all $|k|\geq 2$.

When $k=\pm 1$ and $\theta=\pi/3$, although the Cygan distance between the centers of $\mathcal{I}_{0}^{+}$ and $\mathcal{I}_{\pm 1}^{+}$ is the sum of the radii, we claim that
they are still disjoint.
Using the symmetry $\tau$ in Proposition \ref{prop:symmetry}, we only need to show that $\mathcal{I}_{0}^{+} \cap \mathcal{I}_{1}^{+}=\emptyset$.
Suppose that $q \in \mathcal{I}_{0}^{+}$. Using the geographic coordinates of $q=q(\alpha,\beta,w)$ given in (\ref{eq:geog-coor-plus-0}), we can compute the difference between the Cygan distance of $q$ and the center of $\mathcal{I}_{1}^{+}$ with its radius. That is
\begin{eqnarray*}
  \lefteqn{\left| 2 \left\langle {\bf{q}}, \left[a
                                        \begin{array}{c}
                                          -4e^{-i\pi/3} \\
                                          2 \\
                                          1 \\
                                        \end{array}
                                      \right] \right\rangle
   \right|^2-4} \\
   &=& 4\left| -e^{-i\alpha}+ \sqrt{2}we^{i(-\alpha/2+\beta-\theta)}-4e^{i\pi/3} \right|^2-4 \\
   &=& 32 \left( w^2+\sqrt{2}w\left(\cos(\alpha/2+\beta)/2+2\cos(\alpha/2-\beta+\pi/3)\right)+\cos(\alpha+\pi/3)+2 \right) \\
   &=& 32f(\alpha,\beta,w).
\end{eqnarray*}
Here $f(\alpha,\beta,w)$ can be seen as a quadratic function of $w$. Let
$$B=\sqrt{2}\left(\cos(\alpha/2+\beta)/2+2\cos(\alpha/2-\beta+\pi/3)\right)$$ and $C=\cos(\alpha+\pi/3)+2$.
If $B^2-4C<0$, then it is obvious that  $f(\alpha,\beta,w)>0$. If $B^2-4C\geq 0$, then $B\leq -2\sqrt{C}$ (it is impossible by numerically computation) or $B\geq 2\sqrt{C}$. In this case we have $B-2\sqrt{\cos(\alpha)}\geq B-2\sqrt{C}\geq 0$ since $\cos(\alpha)\leq C$. This means that the symmetry axes of the $f$ lie on the left side of $w=-\sqrt{\cos(\alpha)}$. Besides,
one can compute numerically that $f(\alpha,\beta,-\sqrt{\cos(\alpha)})>0$ on the range of $\alpha$ and $\beta$.
So, we have $f(\alpha,\beta,w)>0$. This means that every point on $\mathcal{I}_{0}^{+}$ lies outside of $\mathcal{I}_{1}^{+}$.
Hence $\mathcal{I}_{0}^{+} \cap \mathcal{I}_{1}^{+}=\emptyset$.
\end{proof}

By a similar argument, we have the following proposition.
\begin{prop}\label{prop:s0plus}
Suppose that $\theta\in[0,\pi/3]$, then
\begin{enumerate}
  \item $\mathcal{I}_{0}^{+}$ and $\mathcal{I}_k^{-}$ are disjoint in $\hc \cup \partial\hc$, except possibly when $k=-1,0$,
  \item $\mathcal{I}_{0}^{+}$ and $\mathcal{I}_k^{\star}$ are disjoint in $\hc \cup \partial\hc$, except possibly when $k=-1,0$,
  \item $\mathcal{I}_{0}^{+}$ and $\mathcal{I}_k^{\diamond}$ are disjoint in $\hc \cup \partial\hc$, except possibly when $k=0,1$.
\end{enumerate}
\end{prop}
\begin{proof}
(1) Since $\mathcal{I}_{0}^{-}$ has center $[2e^{i\theta},0]$ and radius $\sqrt{2}$, the isometric sphere $\mathcal{I}_k^{-}$ has center $[4k\cos(\theta)+2e^{i\theta},0]$
and radius $\sqrt{2}$.
The Cygan distance between the centers of $\mathcal{I}_{0}^{+}$ and $\mathcal{I}_k^{-}$ is
$$\left|4k\cos(\theta)+2e^{i\theta}\right|=2\sqrt{\sin^2(\theta)+(2k+1)^2\cos^2(\theta)},$$
which is bigger than $2\sqrt{2}$ when $k\neq -1, 0$.

(2) Since $\mathcal{I}_{0}^{\star}$ has center $[e^{i\theta},0]$ and radius $1$, the isometric sphere $\mathcal{I}_k^{\star}$ has center $[4k\cos(\theta)+e^{i\theta},4k\sin(2\theta)]$
and radius $1$. The Cygan distance between the centers of $\mathcal{I}_{0}^{+}$ and $\mathcal{I}_k^{\star}$ is
$$\left||(4k+1)\cos(\theta)+i\sin(\theta)|^2-i(4k\sin(2\theta))\right|^{1/2}\geq |4k+1|\cos(\theta),$$
which is bigger than $1+\sqrt{2}$ when $k\neq -1, 0$.

(3) Since $\mathcal{I}_{0}^{\diamond}$ has center $[-e^{-i\theta},0]$ and radius $1$, the isometric sphere $\mathcal{I}_k^{\diamond}$ has center $[4k\cos(\theta)-e^{-i\theta},4k\sin(2\theta)]$ and radius $1$. The Cygan distance between the centers of $\mathcal{I}_{0}^{+}$ and $\mathcal{I}_k^{\diamond}$ is $$\left||(4k-1)\cos(\theta)+i\sin(\theta)|^2-i(4k\sin(2\theta))\right|^{1/2}\geq |4k-1|\cos(\theta),$$
which is bigger than $1+\sqrt{2}$ when $k\neq 0, 1$.
\end{proof}

Similarly, we have
\begin{prop}\label{prop:s0star}
Suppose that $\theta\in[0,\pi/3]$, then
\begin{enumerate}
  \item $\mathcal{I}_{0}^{\star}$ and $\mathcal{I}_k^{\star}$ are disjoint in $\hc$. Furthermore, when $\theta=\pi/3$ the closure of $\mathcal{I}_{0}^{\star}$ and $\mathcal{I}_{-1}^{\star}$
  (respectively, $\mathcal{I}_{1}^{\star}$) is tangent at the parabolic fixed point of $T^{-1}S^2$ (respectively, $T(T^{-1}S^2)T^{-1}$) on $\partial \hc$.
  \item $\mathcal{I}_{0}^{\star}$ and $\mathcal{I}_k^{\diamond}$ are disjoint in $\hc \cup \partial\hc$, except possibly when $k=0,1$.
\end{enumerate}
\end{prop}
\begin{proof}
(1) $\mathcal{I}_k^{\star}$ is a Cygan sphere with center $[4k\cos(\theta)+e^{i\theta},4k\sin(2\theta)]$ and radius $1$, thus the distance between the centers of
$\mathcal{I}_{0}^{\star}$ and $\mathcal{I}_{k}^{\star}$ is
$$
d_{\rm{Cyg}}([4k\cos(\theta)+e^{i\theta},4k\sin(2\theta)],[e^{i\theta},0])=|4k\cos(\theta)|\geq 2.
$$
The equality holds when $k=\pm 1$ and $\theta=\pi/3$.

When $\theta=\pi/3$, we have known that $I_1I_3I_2I_3=T^{-1}S^2$ is unipotent. Since
$$\mathcal{I}_{0}^{\star}=\mathcal{I}(S^2)=\mathcal{I}(T^{-1}S^{2}),$$
and
$$\mathcal{I}_{-1}^{\star}=\mathcal{I}(T^{-1}S^2T)=\mathcal{I}(S^2T)=\mathcal{I}(S^{-2}T),$$
using Phillips's theorem (Theorem 6.1 of \cite{phillips}), the closure of $\mathcal{I}_{0}^{\star}$ and $\mathcal{I}_{-1}^{\star}$ will be tangent at the fixed point of $T^{-1}S^2$ on $\partial \hc$.
Since
$$\mathcal{I}_{0}^{\star}=T(\mathcal{I}_{-1}^{\star}),\quad \mathcal{I}_{1}^{\star}=T(\mathcal{I}_{0}^{\star}),$$
the closure of $\mathcal{I}_{0}^{\star}$ and $\mathcal{I}_{1}^{\star}$ is tangent at the fixed point of $T(T^{-1}S^2)T^{-1}$ on $\partial \hc$.

(2) The distance between the centers of $\mathcal{I}_{0}^{\star}$ and $\mathcal{I}_k^{\diamond}$ is
\begin{eqnarray*}
  \lefteqn{d_{\rm{Cyg}}([4k\cos(\theta)-e^{-i\theta},4k\sin(2\theta)],[e^{i\theta},0])}\\
   &=& 2\sqrt{\cos(\theta)}\cdot|(2k-1)^2\cos(\theta)-i\sin(\theta)|^{1/2} \\
  &\geq& 2|2k-1|\cos(\theta),
\end{eqnarray*}
which is bigger than $2$, except when $k=0,1$.
\end{proof}

\begin{lem}\label{lem:triple}
Suppose that $\theta\in[0,\pi/3]$, then $\mathcal{I}_{0}^{+} \cap \mathcal{I}_{0}^{\star} \cap \mathcal{I}_{-1}^{-}=\emptyset$ except when $\theta=\pi/3$,
in which case the triple intersection is the point $[e^{i2\pi/3},-\sqrt{3}]\in\partial \hc$. Moreover, this point is the parabolic fixed point of $T^{-1}S^2$.
\end{lem}
\begin{proof}
Suppose that $q \in \mathcal{I}_{0}^{+}$. Using Lemma \ref{lem:functions},
the geographic coordinates $(\alpha,\beta,w)$ of $q \in \mathcal{I}_{0}^{+} \cap \mathcal{I}_{0}^{\star} \cap \mathcal{I}_{-1}^{-}$ should satisfy the following equation
\begin{equation}\label{eq:lemma0}
  2w^2+1+\cos(\alpha)-\sqrt{2}w\cos\left(-\frac{\alpha}{2}+\beta-\theta\right)-2\sqrt{2}w\cos\left(\frac{\alpha}{2}+\beta-\theta\right) = 0,
\end{equation}
\begin{equation}\label{eq:lemma1}
2w^2+1+\cos(\alpha)+\sqrt{2}w\cos\left(\frac{\alpha}{2}+\beta+\theta\right)+2\sqrt{2}w\cos\left(-\frac{\alpha}{2}+\beta+\theta\right) = 0.
\end{equation}
Subtracting the two equations (\ref{eq:lemma0}) and (\ref{eq:lemma1}), we have
$$
2\sqrt{2}w\cos(\beta)\left(\cos(\alpha/2+\theta)+2\cos(\alpha/2-\theta) \right)=0.
$$
This implies that either $w=0$ or $\beta=\pi/2$, since $\left(\cos(\alpha/2+\theta)+2\cos(\alpha/2-\theta) \right)\neq 0$ for $\theta\in [0,\pi/3]$.
We know that the points with $w=0$ lie in the meridian with $\beta=\pi/2$.
Therefore, a necessary condition for $q \in \mathcal{I}_{0}^{+} \cap \mathcal{I}_{0}^{\star} \cap \mathcal{I}_{-1}^{-}$ is that $\beta=\pi/2$.

Substituting $\beta=\pi/2$ into the equation (\ref{eq:lemma0}) and simplifying, we have
\begin{equation}\label{eq3}
  2w^2+2\cos^2(\alpha/2)+\sqrt{2}w\left( \sin(\alpha/2)\cos(\theta)-3\cos(\alpha/2)\sin(\theta) \right)=0.
\end{equation}
Let $b(\alpha,\theta)=\sin(\alpha/2)\cos(\theta)-3\cos(\alpha/2)\sin(\theta)$. It is easy to see that for every $\alpha$, the function $\theta\longmapsto b(\alpha,\theta)$ is decreasing on $[0,\pi/3]$.

The left hand side of the equation (\ref{eq3}) can be seen as a quadratic function of $w$ with positive leading coefficient.
Thus the equation (\ref{eq3}) has at least one solution only if $b^2-8\cos^2(\alpha/2)\geq 0$, that is $b\geq 2\sqrt{2}\cos(\alpha/2)$ (it is impossible since $b\leq b(\alpha,0)=\sin(\alpha/2)$) or $b\leq -2\sqrt{2}\cos(\alpha/2)$. Since $\sqrt{\cos(\alpha)}\leq\cos(\alpha/2)$, we have $b+2\sqrt{2}\sqrt{\cos(\alpha)}\leq b+2\sqrt{2}\cos(\alpha/2)\leq 0$. This means that the symmetry axes of the quadratic function lie on the right hand side of $w=\sqrt{\cos(\alpha)}$.

Besides, one can compute that
\begin{eqnarray*}
  \lefteqn{b\sqrt{\cos(\alpha)}+\sqrt{2}\left( \cos(\alpha)+\cos^2(\alpha/2) \right)} \\
   &\geq& \left( \sin(\alpha/2)\cos(\pi/3)-3\cos(\alpha/2)\sin(\pi/3) \right)\sqrt{\cos(\alpha)}+\sqrt{2}\left( \cos(\alpha)+\cos^2(\alpha/2) \right) \\
   &=& \frac{\sqrt{2}}{2}\left( \frac{\sqrt{\cos(\alpha)}}{2}+\frac{\sqrt{2}}{2}\sin(\alpha/2) \right)^2+\frac{3\sqrt{2}}{2}\left( \frac{\sqrt{3}}{2}\sqrt{\cos(\alpha)}-\frac{\sqrt{2}}{2}\cos(\alpha/2) \right)^2.
\end{eqnarray*}
Then $b\sqrt{\cos(\alpha)}+\sqrt{2}\left( \cos(\alpha)+\cos^2(\alpha/2) \right)\geq 0$. If it is $0$, then $\alpha=-\pi/3$ and $\theta=\pi/3$.
It means that for $w\in [-\sqrt{\cos(\alpha)},\sqrt{\cos(\alpha)}]$ the equation (\ref{eq3}) has no solution except when $\theta=\pi/3$ and $\alpha=-\pi/3$, in which case $w=\sqrt{\cos(\alpha)}=\sqrt{2}/2$.

Hence $q \in \mathcal{I}_{0}^{+} \cap \mathcal{I}_{0}^{\star} \cap \mathcal{I}_{-1}^{-}$ if and only if $\theta=\pi/3$, $\alpha=-\pi/3$ and $w=\sqrt{\cos(\alpha)}=\sqrt{2}/2$.
When $\theta=\pi/3$, $T^{-1}S^2$ is unipotent and its fixed point is the eigenvector with eigenvalue $1$. One can compute that its fixed point is $[e^{i2\pi/3},-\sqrt{3}] \in \partial\hc$, which equals to the point $q(-\pi/3,\pi/2,\sqrt{2}/2)$.
\end{proof}

\begin{figure}
\begin{center}
\begin{tikzpicture}
\node at (-2.5,0){\includegraphics[width=5cm,height=5cm]{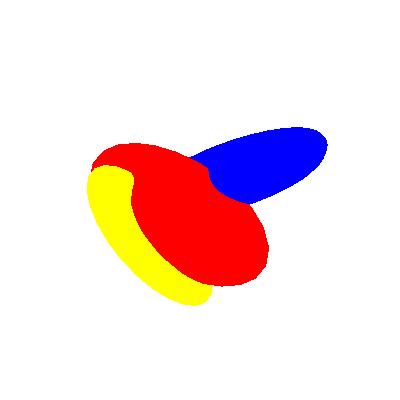}};
\node at (2.5,0){\includegraphics[width=5cm,height=5cm]{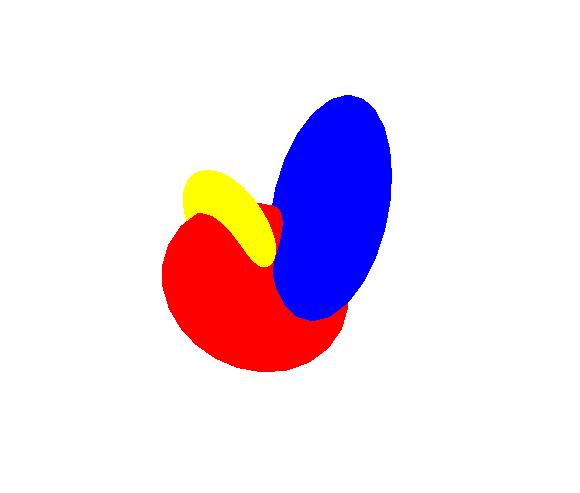}};
\coordinate [label=right:$\mathcal{I}_{0}^{-}$] (q3) at (3.4,0.2);
\coordinate [label=right:$\mathcal{I}_{0}^{+}$] (q3) at (-1.7,-0.9);
\coordinate [label=right:$\mathcal{I}_{0}^{\star}$] (q3) at (1,1);
\coordinate [label=right:$\mathcal{I}_{0}^{-}$] (q3) at (-1.7,1.3);
\coordinate [label=right:$\mathcal{I}_{0}^{+}$] (q3) at (1,-1);
\coordinate [label=right:$\mathcal{I}_{0}^{\star}$] (q3) at (-4,-1.3);
\end{tikzpicture}
\end{center}
  \caption{The ideal boundaries of the three spheres $\mathcal{I}_0^{+}$ , $\mathcal{I}_{-1}^{-}$  and $ \mathcal{I}_0^{\star}$ on $\partial\hc$ (On the left is the case when $\theta=0$ and on the right is $\theta=\pi/3$).}
  \label{fig:double}
\end{figure}

\begin{prop}\label{prop:s0star1}
Suppose that $\theta\in[0,\pi/3]$, then
\begin{enumerate}
\item The intersection $\mathcal{I}_{0}^{\star} \cap \mathcal{I}_{0}^{\diamond}$ lie in the interior of $\mathcal{I}_{0}^{+}$,
\item The intersection $\mathcal{I}_{0}^{\star} \cap \mathcal{I}_{-1}^{-}$ either is empty or lie in the interior of $\mathcal{I}_{0}^{+}$. Furthermore, when $\theta=\pi/3$,
      there is a unique point on the ideal boundary of $\mathcal{I}_{0}^{\star} \cap \mathcal{I}_{-1}^{-}$ on $\partial \hc$, which is fixed by $T^{-1}S^2$,
      lying on the ideal boundary of $\mathcal{I}_{0}^{+}$.
\end{enumerate}
\end{prop}
\begin{proof}
(1) Let $p=(z,t,u)\in \mathcal{I}_{0}^{\star} \cap \mathcal{I}_{0}^{\diamond}$, then $p$ satisfies the equations
\begin{eqnarray*}
  \left| |z-e^{i\theta}|^2+u-i\left(t+2{\rm{Im}}(ze^{-i\theta})\right) \right| &=& 1 \\
  \left| |z+e^{-i\theta}|^2+u-i\left(t+2{\rm{Im}}(-ze^{i\theta})\right) \right| &=& 1.
\end{eqnarray*}
Set $z=|z|e^{i\phi}$, by simplifying, we have
\begin{eqnarray*}
  \left| |z|^2+1+u-2|z|\cos(\phi-\theta)-i(t+2|z|\sin(\phi-\theta)) \right| &=& 1 \\
  \left| |z|^2+1+u+2|z|\cos(\phi+\theta)-i(t-2|z|\sin(\phi+\theta)) \right| &=& 1.
\end{eqnarray*}
Now set
\begin{eqnarray}
\label{eq1}   |z|^2+1+u-2|z|\cos(\phi-\theta)-i(t+2|z|\sin(\phi-\theta))  &=& e^{i\alpha} \\
\label{eq2}   |z|^2+1+u+2|z|\cos(\phi+\theta)-i(t-2|z|\sin(\phi+\theta))  &=& e^{i\beta}.
\end{eqnarray}
Since
$$\cos{\alpha}=|z|^2+1+u-2|z|\cos(\phi-\theta) = \left(|z|-\cos(\phi-\theta)\right)^2+\sin^2(\phi-\theta)+u \geq 0$$
and
$$\cos{\beta}=|z|^2+1+u+2|z|\cos(\phi+\theta) =\left(|z|+\cos(\phi+\theta)\right)^2+\sin^2(\phi+\theta)+u \geq 0,$$
we have $-\pi/2 \leq \alpha \leq \pi/2$ and $-\pi/2 \leq \beta \leq \pi/2$. Thus it implies that $\cos(\beta/2-\alpha/2)\geq 0$.
By computing the difference of the equations (\ref{eq1}) and (\ref{eq2}), we have
\begin{equation}\label{eq:coord-z}
  z=\frac{e^{i\beta}-e^{i\alpha}}{4\cos(\theta)}=\pm \frac{\sin(\beta/2-\alpha/2)}{2\cos(\theta)}e^{i(\pm \pi/2 +\beta/2+\alpha/2)}.
\end{equation}
Thus $\phi=\pm \pi/2 +\beta/2+\alpha/2$.
Therefore,
\begin{eqnarray*}
  \lefteqn{
  (|z|^2+u)^2+t^2  } \\
   &=& \left(\cos(\alpha)-1+2|z|\cos(\phi-\theta)\right)^2+\left(\sin(\alpha)+2|z|\sin(\phi-\theta)\right)^2 \\
   &=& 2+4|z|^2-2\cos(\alpha)-4|z|\cos(\phi-\theta)+4|z|\cos(\phi-\theta-\alpha) \\
   &\leq& 2+4|z|^2-2\left(|z|^2+1-2|z|\cos(\phi-\theta)\right)-4|z|\cos(\phi-\theta)+4|z|\cos(\phi-\theta-\alpha) \\
   &=& 2|z|^2+4|z|\cos(\phi-\theta-\alpha) \\
   &=& 2|z|^2+4|z|\cos(\pm \pi/2 -\theta +\beta/2-\alpha/2) \\
   &=& 2|z|^2+4|z|\left( \pm\sin(\theta)\cos(\beta/2-\alpha/2) \mp \cos(\theta)\sin(\beta/2-\alpha/2) \right) \\
   &=& \frac{\sin^2(\beta/2-\alpha/2)}{2\cos^2(\theta)}+\tan(\theta)\sin(\beta-\alpha)-2\sin^2(\beta/2-\alpha/2).
\end{eqnarray*}
Since $\theta\in[0,\pi/3]$, we have $\frac{\sin^2(\beta/2-\alpha/2)}{2\cos^2(\theta)} \leq 2\sin^2(\beta/2-\alpha/2)$ and $\tan(\theta)\sin(\beta-\alpha) \leq \sqrt{3}$. This implies that $(|z|^2+u)^2+t^2 <4$.
It means that the intersection $\mathcal{I}_{0}^{\star} \cap \mathcal{I}_{0}^{\diamond}$ lie in the interior of $\mathcal{I}_{0}^{+}$.

(2) Suppose that $p=(z,t,u)\in \mathcal{I}_{0}^{\star} \cap \mathcal{I}_{-1}^{-}$, then $p$ satisfies the equations
\begin{eqnarray*}
 1 &=& \left| |z-e^{i\theta}|^2+u-i\left(t+2{\rm{Im}}(ze^{-i\theta})\right) \right| = \left| |z|^2+u+1-it-2ze^{-i\theta} \right| \\
 2 &=& \left| |z+2e^{-i\theta}|^2+u-i\left(t+2{\rm{Im}}(-2ze^{i\theta})\right) \right| =  \left| |z|^2+u+4-it+4ze^{i\theta} \right|.
\end{eqnarray*}
Now set
\begin{eqnarray}
\label{eq4}   |z|^2+u+1-it-2ze^{-i\theta}  &=& e^{i\beta} \\
\label{eq5}    |z|^2+u+4-it+4ze^{i\theta}  &=& 2e^{i\alpha}.
\end{eqnarray}
By computing the difference of the equations (\ref{eq4}) and (\ref{eq5}), we have
\begin{eqnarray}\label{eq6}
  z &=& \frac{2e^{i\alpha}-e^{i\beta}-3}{4e^{i\theta}+2e^{-i\theta}}.
\end{eqnarray}
According to equation (\ref{eq4}), we have
\begin{eqnarray}
\label{eq7}  u &=& \cos(\beta)-\left| ze^{-i\theta}-1 \right|^2 \\
\label{eq8}  t &=& -\sin(\beta)-2{\rm{Im}}(ze^{-i\theta}).
\end{eqnarray}
Since
$$\cos{\beta}=u+\left| ze^{-i\theta}-1 \right|^2 \geq 0$$
and
$$2\cos{\alpha}=u+\left| ze^{i\theta}+2 \right|^2 \geq 0,$$
we have $-\pi/2 \leq \alpha \leq \pi/2$ and $-\pi/2 \leq \beta \leq \pi/2$.

Now let us consider the case when $\theta=\pi/3$. Substituting $\alpha=\pi/6$ and $\beta=0$ into the equations (\ref{eq6}), (\ref{eq7}) and (\ref{eq8}), we obtain the point
$$p_0=\left( \left(\frac{\sqrt{3}}{3}-1\right)+i\frac{3+\sqrt{3}}{6},\frac{3-7\sqrt{3}}{6}, \frac{-13+8\sqrt{3}}{6} \right) \in \mathcal{I}_{0}^{\star} \cap \mathcal{I}_{-1}^{-}.$$
One can compute that $p_0$ lies in the interior of $\mathcal{I}_{0}^{+}$,
since
$$
\left||z|^2+u-it\right|^2=\left|e^{i\beta}-1+2ze^{-i\theta}\right|^2=4|z|^2=20/3-2\sqrt{3}<4.
$$
We know that the intersection $\mathcal{I}_{0}^{\star} \cap \mathcal{I}_{-1}^{-}$ is connected from Proposition \ref{prop:connect}. Thus, according to Lemma \ref{lem:triple}, $\mathcal{I}_{0}^{\star} \cap \mathcal{I}_{-1}^{-}$ lie in the interior of $\mathcal{I}_{0}^{+}$ except the point $[e^{i2\pi/3},-\sqrt{3}]$, which lies on the ideal boundary of $\mathcal{I}_{0}^{+}$. See Figure \ref{fig:double}.

Observe that coordinates of the centers of $\mathcal{I}_{0}^{\star}$ and $\mathcal{I}_{-1}^{-}$ are continues on $\theta$. Thus the geometric positions of the spheres $\mathcal{I}_{0}^{\star}$ and $\mathcal{I}_{-1}^{-}$ are continuously moved on $\theta$. When $\theta=0$, since the Cygan distance between the centers of $\mathcal{I}_{0}^{\star}$ and $\mathcal{I}_{-1}^{-}$ is bigger than the sum of their radii, one can see that $\mathcal{I}_{0}^{\star} \cap \mathcal{I}_{-1}^{-}=\emptyset$. When $\theta=\pi/3$, we have shown that $\mathcal{I}_{0}^{\star} \cap \mathcal{I}_{-1}^{-}$ lies in the interior of $\mathcal{I}_{0}^{+}$. We also have $\mathcal{I}_{0}^{+} \cap \mathcal{I}_{0}^{\star} \cap \mathcal{I}_{-1}^{-}=\emptyset$ for $\theta\in[0,\pi/3)$ by Lemma \ref{lem:triple}.
Hence, the intersection $\mathcal{I}_{0}^{\star} \cap \mathcal{I}_{-1}^{-}$ is either empty or contained in the interior of $\mathcal{I}_{0}^{+}$.

\end{proof}

\begin{figure}
\begin{center}
\begin{tikzpicture}
\node at (-2.5,0){\includegraphics[width=5cm,height=5cm]{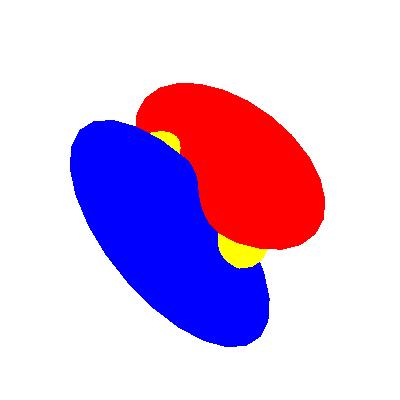}};
\node at (2.5,0){\includegraphics[width=5cm,height=5cm]{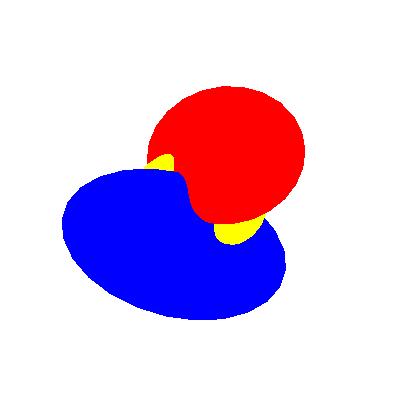}};
\coordinate [label=right:$\mathcal{I}_{0}^{\star}$] (q3) at (3.4,-0.2);
\coordinate [label=right:$\mathcal{I}_{0}^{\star}$] (q3) at (-1.7,-0.9);
\coordinate [label=right:$\mathcal{I}_{0}^{+}$] (q3) at (2,1.6);
\coordinate [label=right:$\mathcal{I}_{0}^{+}$] (q3) at (-1.7,1.3);
\coordinate [label=right:$\mathcal{I}_{0}^{-}$] (q3) at (1,-1.6);
\coordinate [label=right:$\mathcal{I}_{0}^{-}$] (q3) at (-4,-1.3);
\end{tikzpicture}
\end{center}
\caption{The ideal boundaries of the three spheres $\mathcal{I}_0^{+}$ , $\mathcal{I}_{0}^{-}$ and $ \mathcal{I}_0^{\star}$  on $\partial\hc$ (On the left is the case when $\theta=0$ and on the right is $\theta=\pi/3$).}
\label{fig:triple}
\end{figure}

\begin{prop}\label{prop:triple}
Suppose that $\theta\in [0, \pi/3]$. For $k\in\mathbb{Z}$, the three isometric spheres $\mathcal{I}_k^{+}$, $\mathcal{I}_k^{-}$, $\mathcal{I}_k^{\star}$ (respectively, $\mathcal{I}_k^{+}$, $\mathcal{I}_{k-1}^{-}$, $\mathcal{I}_k^{\diamond}$ ) have the following properties.
\begin{itemize}
  \item The intersections $\mathcal{I}_k^{+} \cap \mathcal{I}_k^{-}$, $\mathcal{I}_k^{-} \cap \mathcal{I}_k^{\star}$, and $\mathcal{I}_k^{\star} \cap \mathcal{I}_k^{+}$ (respectively, $\mathcal{I}_k^{+} \cap \mathcal{I}_{k-1}^{-}$, $\mathcal{I}_{k-1}^{-} \cap \mathcal{I}_k^{\diamond}$, $\mathcal{I}_k^{\diamond} \cap \mathcal{I}_k^{+}$) are discs.
  \item The intersection $\mathcal{I}_k^{+}\cap \mathcal{I}_{k}^{-} \cap \mathcal{I}_k^{\star}$ (respectively, $\mathcal{I}_k^{+}\cap \mathcal{I}_{k-1}^{-} \cap \mathcal{I}_k^{\diamond}$ ) is a union of two geodesics which are crossed at the fixed point of $T^{k}ST^{-k}$ (respectively $T^{k}(S^{-1}T)T^{-k}$) in $\hc$ and whose fours endpoints are on $\partial \hc$.
      Moreover, the four rays from the fixed point to the four endpoints are cyclical permuted by $T^{k}ST^{-k}$ (respectively $T^{k}(S^{-1}T)T^{-k}$).
\end{itemize}
\end{prop}

\begin{proof}
According to Definition \ref{def:isometric} and Proposition \ref{prop:symmetry}, it suffices to consider the isometric spheres $\mathcal{I}_{0}^{+}$, $\mathcal{I}_{0}^{-}$, $\mathcal{I}_{0}^{\star}$.  See Figure \ref{fig:triple}.

Let $q \in \mathcal{I}_{0}^{+}$. Consider the geographic coordinates $(\alpha,\beta,w)$ of $q$ in (\ref{eq:geog-coor-plus-0}).
By Lemma \ref{lem:functions},
if $q$ lies on $ \mathcal{I}_{0}^{+} \cap \mathcal{I}_{0}^{-}$, then the $\alpha,\beta,w$ should satisfy the equation
\begin{equation}\label{eq:prop0}
   2w^2+1+\cos(\alpha)-\sqrt{2}w\cos(\alpha/2+\beta-\theta)-2\sqrt{2}w\cos(-\alpha/2+\beta-\theta)=0.
\end{equation}
Similarly, if $q$ lies on $\mathcal{I}_{0}^{+}\cap \mathcal{I}_{0}^{\star}$, then the $\alpha,\beta,w$ should satisfy the equation
\begin{equation}\label{eq:prop1}
   2w^2+1+\cos(\alpha)-\sqrt{2}w\cos(-\alpha/2+\beta-\theta)-2\sqrt{2}w\cos(\alpha/2+\beta-\theta)=0
\end{equation}
Thus the intersection $\mathcal{I}_{0}^{+} \cap \mathcal{I}_{0}^{-}$ is the set of solutions of equation (\ref{eq:prop0})
and the intersection $\mathcal{I}_0^{+} \cap \mathcal{I}_0^{\star}$ is the set of solutions of equation (\ref{eq:prop1}).
One can easily verify that the geographic coordinate of the point $q(0,\theta,\sqrt{2}/2)\in\hc$ satisfies the equations (\ref{eq:prop0}) and (\ref{eq:prop1}),
so these intersections  $\mathcal{I}_{0}^{+} \cap \mathcal{I}_{0}^{-}$   and $\mathcal{I}_0^{+} \cap \mathcal{I}_0^{\star}$   are topological discs from Proposition \ref{prop:connect}.

The intersection of these two sets gives the triple intersection $\mathcal{I}_{0}^{+}  \cap \mathcal{I}_{0}^{-} \cap \mathcal{I}_0^{\star}$.
Now let us solve the system of the equations (\ref{eq:prop0}) and (\ref{eq:prop1}).
Let $t=\beta-\theta$. Subtracting the equations (\ref{eq:prop0}) and (\ref{eq:prop1}) and simplifying, we obtain
$$
2w\sin(\alpha/2)\sin(t)=0.
$$
Thus $w=0$ (this is impossible), or $\alpha=0$, or $t=0$. If $t=0$, then setting $\beta=\theta$ in equation (\ref{eq:prop0}), we get
$$
2w^2-3\sqrt{2}\cos\left(\frac{\alpha}{2}\right)w+1+\cos(\alpha)=2\left(w-\frac{\sqrt{2}}{2}\cos\left(\frac{\alpha}{2}\right)\right)\left(w-\sqrt{2}\cos\left(\frac{\alpha}{2}\right)\right)=0.
$$
Note that the solutions of the above equation for $w$ should satisfy $w^2\leq \cos(\alpha)$.
Thus
$$w=\frac{\sqrt{2}}{2}\cos\left(\frac{\alpha}{2}\right) \quad {\rm{with}}\quad \cos(\alpha)\geq \frac{1}{3}.$$
If $\alpha=0$, then equation (\ref{eq:prop0}) becomes to
\begin{eqnarray}\label{eq:prop2}
  2w^2-3\sqrt{2}\cos(t)w+2 &=& 0.
\end{eqnarray}
Note that the solutions of equation (\ref{eq:prop2}) for $w$ should satisfy $w^2\leq\cos(\alpha)=1$. Thus the solutions of equation (\ref{eq:prop2}) are
$$w=\frac{3\cos(t)-\sqrt{9\cos^{2}(t)-8}}{2\sqrt{2}} \quad {\rm{with}} \quad \frac{2\sqrt{2}}{3}\leq \cos(t)\leq 1,$$
and
$$w=\frac{3\cos(t)+\sqrt{9\cos^{2}(t)-8}}{2\sqrt{2}}\quad {\rm{with}} \quad -1 \leq \cos(t) \leq -\frac{2\sqrt{2}}{3}.$$
So, the triple intersection $\mathcal{I}_{0}^{+}  \cap \mathcal{I}_{0}^{-} \cap \mathcal{I}_0^{\star}$ is the union $\mathcal{L}_1 \cup \mathcal{C}_1 \cup \mathcal{C}_2$, where
\begin{equation*}
  \mathcal{L}_1=\left\{ q(\alpha,t+\theta,w)\in \mathcal{I}_{0}^{+} : \cos(\alpha)\geq \frac{1}{3}, t=0, w=\frac{\sqrt{2}}{2}\cos\left(\frac{\alpha}{2}\right) \right\},
\end{equation*}
$$
\mathcal{C}_1=\left\{ q(0,t+\theta,w)\in \mathcal{I}_{0}^{+} :  \frac{2\sqrt{2}}{3}\leq \cos(t)\leq 1, w=\frac{3\cos(t)-\sqrt{9\cos^{2}(t)-8}}{2\sqrt{2}} \right\},
$$
and
$$
\mathcal{C}_2=\left\{ q(0,t+\theta,w)\in \mathcal{I}_{0}^{+} :  -1 \leq \cos(t) \leq -\frac{2\sqrt{2}}{3}, w=\frac{3\cos(t)+\sqrt{9\cos^{2}(t)-8}}{2\sqrt{2}} \right\}.
$$

Note that $\mathcal{L}_1$ lie in a Lagrangian plane of $\mathcal{I}_{0}^{+}$, and $\mathcal{C}_1 \cup \mathcal{C}_2$ lie in a complex line of $\mathcal{I}_{0}^{+}$.
It is obvious that $\mathcal{C}_1$ is an arc. One of its endpoints is $q(0,\theta,\sqrt{2}/2)\in\hc$, which is the fixed point of $S$. The other one is $q(0,\arccos(2\sqrt{2}/3)+\theta,1)\in\partial\hc$. Similarly, $\mathcal{C}_2$ is an arc whose endpoints are $q(0,\theta,\sqrt{2}/2)$ and $q(0,\arccos(-2\sqrt{2}/3)+\theta,-1)\in\partial\hc$. Thus $\mathcal{C}_1 \cup \mathcal{C}_2$ is connected.
The endpoints of $\mathcal{L}_1$ are $q(\arccos(1/3),\theta,\sqrt{3}/3)$ and $q(-\arccos(1/3),\theta,\sqrt{3}/3)$, which are on $\partial \hc$. It is easy to see that $\mathcal{L}_1$ intersects with $\mathcal{C}_1 \cup \mathcal{C}_2$ at the point $q(0,\theta,\sqrt{2}/2)\in\hc$.

\begin{figure}
\begin{center}
\begin{tikzpicture}
\node at (0,0){\includegraphics[width=5cm,height=5cm]{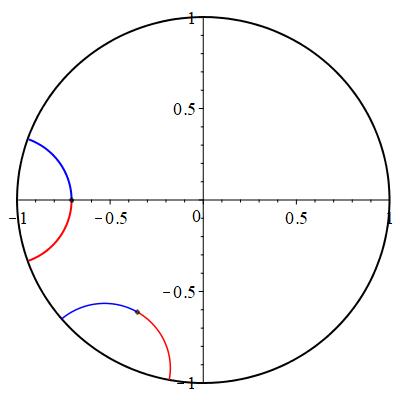}};
\coordinate [label=below:$\mathcal{C}_1$] (c1) at (-1.5,-0.3);
\coordinate [label=below:$\mathcal{C}_2$] (c) at (-1.7,1);
\coordinate [label=left:$\mathcal{L}_1$] (q0) at (-1,0.2);
\fill (-1.62,0) circle (2pt);
\fill (-0.8,-1.4) circle (2pt);
\end{tikzpicture}
\end{center}
  \caption{The triple intersection $\mathcal{I}_0^{+}\cap \mathcal{I}_{0}^{-} \cap \mathcal{I}_0^{\star}$ viewed on the vertical axis in the ball model of $\hc$. The blue curve is $\mathcal{C}_1 $ and the red one is $\mathcal{C}_2$. The black point on the curve is the projection of $\mathcal{L}_1$ on the complex line. The left curve is the case when $\theta=0$ and the one on the lower left is the case when $\theta=\pi/3$.}
  \label{fig:cross}
\end{figure}
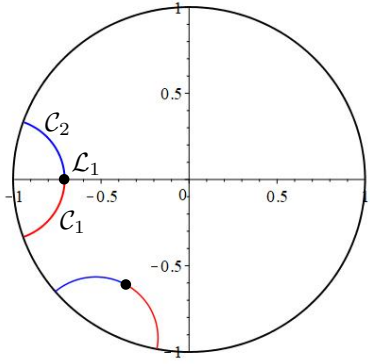

Moreover, $\mathcal{C}_1 \cup \mathcal{C}_2$ is a geodesic. In fact, the complex line containing $\mathcal{C}_1 \cup \mathcal{C}_2$ is $\mathcal{C}=\{(-1,z)\in\hc\cup\partial\hc : |z|\leq\sqrt{2}\}$. It is a disc bounded by the circle with center being the origin and radius $\sqrt{2}$.
While $\mathcal{C}_1 \cup \mathcal{C}_2$ lie in the circle with center $3e^{i\theta}/2$ and radius $1/2$ which is orthogonal to the boundary of the complex line.
By the Cayley transform given in Definition \ref{def:cayley}, $\mathcal{C}$ is mapped to the vertical axis $\{(0,z)\in\hc\cup\partial\hc : |z|\leq 1\}$ in the ball model of $\hc$.
Thus $\mathcal{C}$ is isometric to the Poincar\'{e} disc. While $\mathcal{C}_1 \cup \mathcal{C}_2$ is mapped by the  Cayley transform to an arc contained in the circle with center $-3e^{i\theta}/2\sqrt{2}$ and radius $1/2\sqrt{2}$ which is orthogonal to the unit circle. Hence $\mathcal{C}_1 \cup \mathcal{C}_2$ is a geodesic. See Figure \ref{fig:cross}.

By the Cayley transform, $\mathcal{L}_1$ is mapped to $\{(-\tan(\alpha/2)i,-e^{i\theta}/\sqrt{2})\in\hc\cup\partial\hc : \cos(\alpha)\geq 1/3\}$.
Thus $\mathcal{L}_1$ and $\mathcal{C}_1 \cup \mathcal{C}_2$ are crossed at the point $(0,-e^{i\theta}/\sqrt{2})\in\hc$ which is the image of $q(0,\theta,\sqrt{2}/2)$ under the Cayley transform.

All the five points are lifted to the vectors in $\mathbb{C}^3$:
\begin{eqnarray*}
  {\bf q}(\arccos(1/3),\theta,\sqrt{3}/3) &=& \left[
                               \begin{array}{c}
                                 -\left(\frac{1}{3}-i\frac{2\sqrt{2}}{3}\right) \\
                                 \left(\frac{2}{3}-i\frac{\sqrt{2}}{3}\right)e^{i\theta} \\
                                 1 \\
                               \end{array}
                             \right],\\
   {\bf q}(-\arccos(1/3),\theta,\sqrt{3}/3) &=& \left[
                               \begin{array}{c}
                                 -\left(\frac{1}{3}+i\frac{2\sqrt{2}}{3}\right) \\
                                 \left(\frac{2}{3}+i\frac{\sqrt{2}}{3}\right)e^{i\theta} \\
                                 1 \\
                               \end{array}
                             \right],\\
   {\bf q}(0,\arccos(2\sqrt{2}/3)+\theta,1) &=& \left[
                               \begin{array}{c}
                                 -1 \\
                                 \left(\frac{4}{3}+i\frac{\sqrt{2}}{3}\right)e^{i\theta} \\
                                 1 \\
                               \end{array}
                             \right],\\
   {\bf q}(0,\arccos(-2\sqrt{2}/3)+\theta,-1) &=& \left[
                               \begin{array}{c}
                                 -1 \\
                                 \left(\frac{4}{3}-i\frac{\sqrt{2}}{3}\right)e^{i\theta} \\
                                 1 \\
                               \end{array}
                             \right],\\
   {\bf q}(0,\theta,\sqrt{2}/2) &=& \left[
                               \begin{array}{c}
                                 -1 \\
                                 e^{i\theta} \\
                                 1 \\
                               \end{array}
                             \right].
\end{eqnarray*}

Recall that
$$
S=\left[
  \begin{array}{ccc}
    2 & 2e^{-i\theta} & -1 \\
    -2e^{i\theta} & -1 & 0 \\
    -1 & 0 & 0 \\
  \end{array}
\right].
$$
Thus it is easy to check that $S({\bf q}(0,\theta,\sqrt{2}/2))={\bf q}(0,\theta,\sqrt{2}/2)$ and the other four points are cyclical permuted by $S$ as the following
\begin{eqnarray*}
&\left[
                               \begin{array}{c}
                                 -1 \\
                                 \left(\frac{4}{3}+i\frac{\sqrt{2}}{3}\right)e^{i\theta} \\
                                 1 \\
                               \end{array}
                             \right] \underrightarrow{S}
\left[
                               \begin{array}{c}
                                 -\left(\frac{1}{3}-i\frac{2\sqrt{2}}{3}\right) \\
                                 \left(\frac{2}{3}-i\frac{\sqrt{2}}{3}\right)e^{i\theta} \\
                                 1 \\
                               \end{array}
                             \right] \underrightarrow{S}
\left[
                               \begin{array}{c}
                                 -1 \\
                                 \left(\frac{4}{3}-i\frac{\sqrt{2}}{3}\right)e^{i\theta} \\
                                 1 \\
                               \end{array}
                             \right] \underrightarrow{S}\\
&\left[
                               \begin{array}{c}
                                 -\left(\frac{1}{3}+i\frac{2\sqrt{2}}{3}\right) \\
                                 \left(\frac{2}{3}+i\frac{\sqrt{2}}{3}\right)e^{i\theta} \\
                                 1 \\
                               \end{array}
                             \right].
\end{eqnarray*}

Moreover, it is easy to verify that  $\mathcal{C}_1 \cup \mathcal{C}_2=S(\mathcal{L}_1)$, $S^2(\mathcal{L}_1)=\mathcal{L}_1$ and $S^2(\mathcal{C}_1)=\mathcal{C}_2$.
Thus, the four rays from the fixed point to the four endpoints are cyclical permuted by $S$.
\end{proof}

By applying powers of $T$ and the symmetries in Proposition \ref{prop:symmetry} to Proposition \ref{prop:s0plus}, Proposition \ref{prop:s0star} and Proposition \ref{prop:s0star1}, all pairwise intersections of the isometric spheres can be summarized in the following result.
\begin{cor}\label{cor:intersections}
Suppose that $\theta\in [0, \pi/3]$. Let $\mathcal{S}=\{\mathcal{I}_k^{\pm}, \mathcal{I}_k^{\star}, \mathcal{I}_k^{\diamond}: k \in \mathbb{Z} \}$ be the set of all the isometrical spheres. Then for all $k\in\mathbb{Z}$:
\begin{enumerate}
  \item $\mathcal{I}_k^{+}$ is contained in the exterior of all the isometric spheres in $\mathcal{S}$ except $\mathcal{I}_k^{-}$, $\mathcal{I}_{k-1}^{-}$, $\mathcal{I}_k^{\star}$, $\mathcal{I}_{k-1}^{\star}$, $\mathcal{I}_k^{\diamond}$ and $\mathcal{I}_{k+1}^{\diamond}$.
         Moreover, $\mathcal{I}_k^{+} \cap \mathcal{I}_{k-1}^{\star}$ (resp. $\mathcal{I}_k^{+} \cap \mathcal{I}_{k+1}^{\diamond}$) is either empty or contained in the interior of $\mathcal{I}_{k-1}^{-}$ (resp. $\mathcal{I}_{k}^{-}$). When $\theta=\pi/3$, $\mathcal{I}_k^{+} \cap \mathcal{I}_{k-1}^{\star}$ (resp. $\mathcal{I}_k^{+} \cap \mathcal{I}_{k+1}^{\diamond}$) will be tangent with $\mathcal{I}_{k-1}^{-}$ (resp. $\mathcal{I}_{k}^{-}$)  on $\partial \hc$ at the parabolic fixed point of $T^{k}(S^2T)T^{-k}$ (resp. $T^{k}(S^{-1}TS^{-1})T^{-k}$).
  \item $\mathcal{I}_k^{-}$ is contained in the exterior of all the isometric spheres in $\mathcal{S}$ except $\mathcal{I}_k^{+}$, $\mathcal{I}_{k+1}^{+}$,
        $\mathcal{I}_k^{\star}$, $\mathcal{I}_{k+1}^{\star}$, $\mathcal{I}_k^{\diamond}$ and $\mathcal{I}_{k+1}^{\diamond}$.
         Moreover, $\mathcal{I}_k^{-} \cap \mathcal{I}_k^{\diamond}$ (resp. $\mathcal{I}_k^{-} \cap \mathcal{I}_{k+1}^{\star}$) is either empty or contained in the interior of $\mathcal{I}_{k}^{+}$ (resp. $\mathcal{I}_{k+1}^{+}$). When $\theta=\pi/3$, $\mathcal{I}_k^{-} \cap \mathcal{I}_k^{\diamond}$ (resp. $\mathcal{I}_k^{-} \cap \mathcal{I}_{k+1}^{\star}$) will be tangent with $\mathcal{I}_k^{+}$ (resp. $\mathcal{I}_{k+1}^{+}$)  on $\partial \hc$ at the parabolic fixed point of $T^{k}(ST^{-1}S)T^{-k}$ (resp. $T^{k}(S^2T^{-1})T^{-k}$).
  \item $\mathcal{I}_k^{\star}$ is contained in the exterior of all the isometric spheres in $\mathcal{S}$ except $\mathcal{I}_k^{\pm}$,
        $\mathcal{I}_{k+1}^{+}$, $\mathcal{I}_{k-1}^{-}$, $\mathcal{I}_k^{\diamond}$ and $\mathcal{I}_{k+1}^{\diamond}$.
          Moreover, $\mathcal{I}_k^{\star} \cap \mathcal{I}_k^{\diamond}$ (resp. $\mathcal{I}_k^{\star} \cap \mathcal{I}_{k+1}^{\diamond}$) is contained in the interior of $\mathcal{I}_k^{+}$ (resp. $\mathcal{I}_k^{-}$). $\mathcal{I}_k^{\star} \cap \mathcal{I}_{k+1}^{+}$ is described in item (1), and $\mathcal{I}_k^{\star} \cap \mathcal{I}_{k-1}^{-}$ is described in item (2).
         When $\theta=\pi/3$, $\mathcal{I}_k^{\star}$ will be tangent with $\mathcal{I}_{k+1}^{\star}$ (resp. $\mathcal{I}_{k-1}^{\star}$) on $\partial \hc$ at the parabolic fixed point of $T^{k}(S^2T^{-1})T^{-k}$ (resp. $T^{k}(T^{-1}S^2)T^{-k}$).
  \item $\mathcal{I}_k^{\diamond}$ is contained in the exterior of all the isometric spheres in $\mathcal{S}$ except $\mathcal{I}_k^{\pm}$,
        $\mathcal{I}_{k-1}^{\pm}$, $\mathcal{I}_k^{\star}$ and $\mathcal{I}_{k-1}^{\star}$.
          Moreover, $\mathcal{I}_k^{\diamond} \cap \mathcal{I}_k^{\star}$ and $\mathcal{I}_{k}^{\diamond} \cap \mathcal{I}_{k-1}^{\star}$ are described in item (3). $\mathcal{I}_k^{\diamond} \cap \mathcal{I}_k^{-}$ is described in item (2) and $\mathcal{I}_{k}^{\diamond} \cap \mathcal{I}_{k-1}^{+}$ is described in item (1).
         When $\theta=\pi/3$, $\mathcal{I}_k^{\diamond}$ will be tangent with $\mathcal{I}_{k+1}^{\diamond}$ (resp. $\mathcal{I}_{k-1}^{\diamond}$) on $\partial \hc$ at the parabolic fixed point of $T^{k}(T^{-1}S^2)T^{-k}$ (resp. $T^{k}(S^2T^{-1})T^{-k}$).
\end{enumerate}
\end{cor}

\begin{defn}\label{domain:D}
Let $D$ be the intersection of the closures of the exteriors of all the isometric spheres $\mathcal{I}_k^{+}$,  $\mathcal{I}_k^{-}$, $\mathcal{I}_k^{\star}$ and $\mathcal{I}_k^{\diamond}$, for $k\in\mathbb{Z}$.
\end{defn}

\begin{defn}
For $k\in\mathbb{Z}$,
let $s_k^{+}$ (respectively, $s_k^{-}$, $s_k^{\star}$, and $s_k^{\diamond}$) be the side of $D$ contained in the isometric sphere $\mathcal{I}_k^{+}$ (respectively, $\mathcal{I}_k^{-}$, $\mathcal{I}_k^{\star}$ and $\mathcal{I}_k^{\diamond}$).
\end{defn}

\begin{defn}
A \emph{ridge} is defined to be the 2-dimensional connected intersections of two sides.
\end{defn}

By Corollary \ref{cor:intersections}, the ridges are $s_{k}^{+} \cap s_{k}^{-}$, $s_{k}^{+} \cap s_{k}^{\star}$, $s_{k}^{+} \cap s_{k-1}^{-}$, $s_{k}^{+} \cap s_{k}^{\diamond}$, $s_{k}^{-} \cap s_{k}^{\star}$ and  $s_{k-1}^{-} \cap s_{k}^{\diamond}$ for $k\in\mathbb{Z}$, and the sides and ridges are related as follows:
\begin{itemize}
\item The side $s_k^{+}$ is bounded by the ridges $s_k^{+} \cap s_k^{-}$, $s_k^{+} \cap s_k^{\star}$, $s_k^{+} \cap s_{k-1}^{-}$ and $s_k^{+} \cap s_k^{\diamond}$.
\item The side $s_k^{-}$ is bounded by the ridges $s_k^{+} \cap s_k^{-}$, $s_k^{-} \cap s_k^{\star}$, $s_k^{-} \cap s_{k+1}^{+}$ and $s_0^{-} \cap s_{k+1}^{\diamond}$.
\item The side $s_k^{\star}$ is bounded by the ridges $s_k^{+} \cap s_k^{\star}$ and $s_k^{-} \cap s_k^{\star}$.
\item The side $s_k^{\diamond}$ is bounded by the ridges $s_k^{\diamond} \cap s_{k-1}^{-}$ and $s_k^{+} \cap s_k^{\diamond}$.
\end{itemize}

\begin{prop}\label{prop:ridges}
The ridges $s_{k}^{+} \cap s_{k}^{-}$, $s_{k}^{+} \cap s_{k}^{\star}$, $s_{k}^{+} \cap s_{k-1}^{-}$, $s_{k}^{+} \cap s_{k}^{\diamond}$, $s_{k}^{-} \cap s_{k}^{\star}$ and  $s_{k-1}^{-} \cap s_{k}^{\diamond}$ for $k\in\mathbb{Z}$ are all topologically the union of two sectors.
\end{prop}

\begin{proof}
The ridge $s_{k}^{+} \cap s_{k}^{-}$ is contained in $\mathcal{I}_k^{+} \cap \mathcal{I}_k^{-}$. According to Proposition \ref{prop:triple}, $\mathcal{I}_k^{+} \cap \mathcal{I}_k^{-}$ is topologically a disc and $\mathcal{I}_k^{+} \cap \mathcal{I}_k^{-} \cap \mathcal{I}_k^{\star}$ is the union of two crossed geodesics. The two crossed geodesics divide the disc into four sectors and one opposite pair of which will lie in the interior of the isometric sphere $\mathcal{I}_k^{\star}$. Thus $s_{k}^{+} \cap s_{k}^{-}$ is the other opposite pair of the four sectors in the disc.
More precisely, up to the powers of $T$, let us consider $s_{0}^{+} \cap s_{0}^{-}$. Let $\Delta$ be the disc $\mathcal{I}_0^{+} \cap \mathcal{I}_0^{-}$ described in the equation (\ref{eq:prop0}) and the two crossed geodesics $\mathcal{L}_1\cup \mathcal{C}_1 \cup \mathcal{C}_2$ are described in Proposition \ref{prop:triple}.
By Proposition \ref{prop:symmetry}, the complex involution $I_2$ preserves $\Delta$ and $\mathcal{L}_1\cup \mathcal{C}_1 \cup \mathcal{C}_2$. Recall that $I_2$ fixing the complex line $C_2$ with polar vector $\bf{n}_2$ described in Section \ref{sec:parameter}. One can compute that the intersection $C_2 \cap \Delta$ is the curve
\begin{equation}\label{eq:ridge}
  C_2 \cap \Delta=\{q(\alpha,\alpha/2 +\theta,\sqrt{2}/2)\in\mathcal{I}_0^{+}: \cos(\alpha)\geq 1/3\}.
\end{equation}
Of course $C_2 \cap \Delta$ intersects with $\mathcal{L}_1\cup \mathcal{C}_1 \cup \mathcal{C}_2$ at the fixed point of $S$, and divides $\Delta$ into two parts. $I_2$ fixes $C_2 \cap \Delta$, and interchanges $\mathcal{L}_1$ and $\mathcal{C}_1\cup \mathcal{C}_2$. Thus $C_2 \cap \Delta$ is contained in the union of two opposite sectors. By Lemma \ref{lem:functions}, $C_2 \cap \Delta$ lie on the closure of the exterior of $\mathcal{I}_0^{\star}$, since $f_0^{\star}(\alpha,\alpha/2+\theta,\sqrt{2}/2)=1-\cos(\alpha)\geq 0$. Therefore, the union of two opposite sectors containing $C_2 \cap \Delta$ is the ridge $s_{0}^{+} \cap s_{0}^{-}$. Moreover, this ridge is preserved by $I_2$.  By using the parametrization of the Giraud disk in \cite{dpp}, it can be instructive to draw figure of the Giraud disk $\mathcal{I}_{0}^{+} \cap \mathcal{I}_{0}^{-}$ and their intersection with the isometric spheres $\mathcal{I}_{-1}^{-}$, $\mathcal{I}_0^{\star}$, $\mathcal{I}_{-1}^{\star}$, $\mathcal{I}_0^{\diamond}$ and $\mathcal{I}_{1}^{\diamond}$. See Figure \ref{fig:sector}.

The other ridges can be described by a similar argument.
\end{proof}


\begin{figure}[ht]
\centering
\includegraphics[scale=0.3]{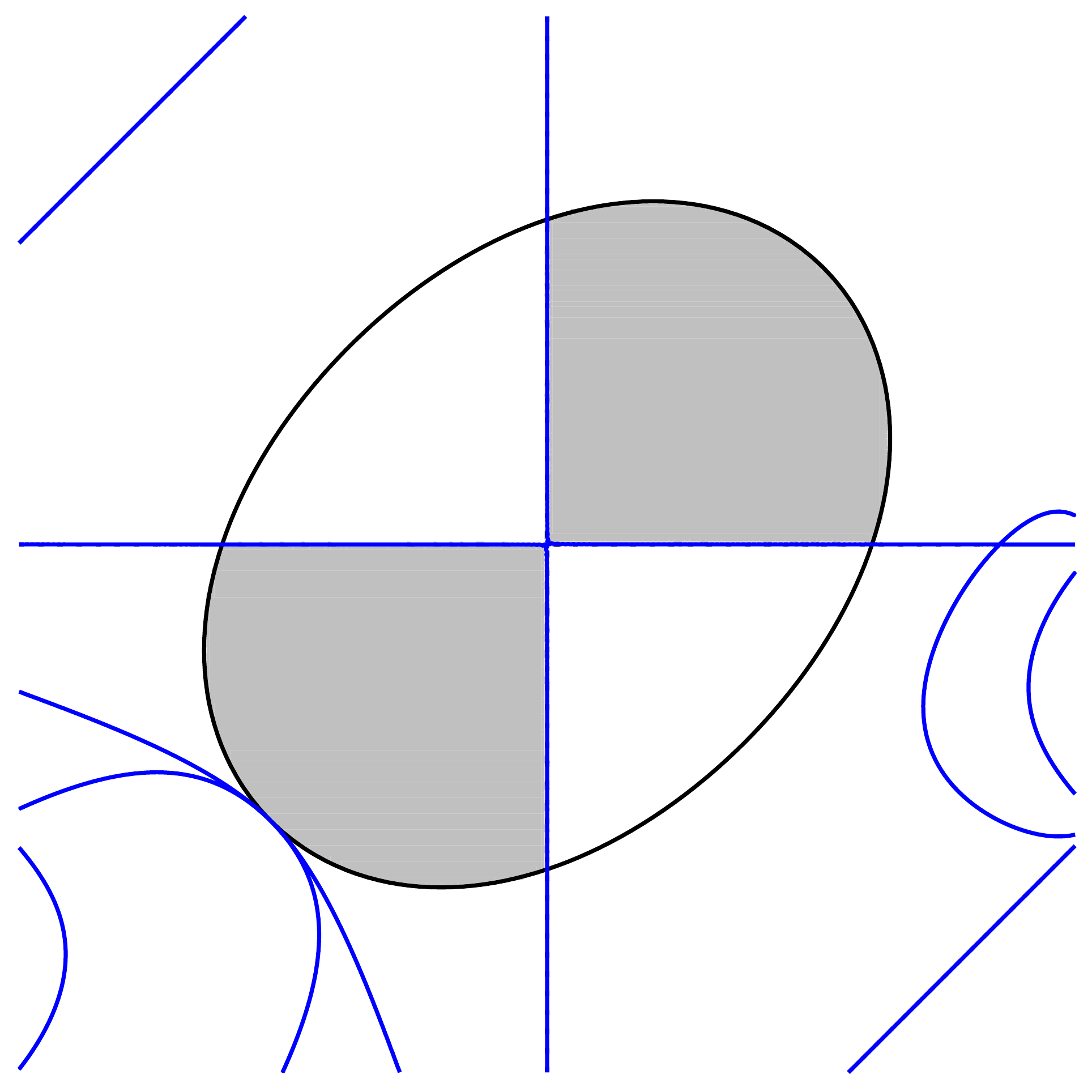}
\caption{ The figure  shows the  ridge $s_{0}^{+} \cap s_{0}^{-}$(the shaded region in the intersection of $\mathcal{I}_{0}^{+} \cap \mathcal{I}_{0}^{-}$ ) in the plane with spinal coordinates introduced in \cite{dpp}. The triple intersection $\mathcal{I}_0^{+} \cap \mathcal{I}_0^{-} \cap \mathcal{I}_0^{\star}$ is the two mutually perpendicular lines.   Compare to Fig.16 in \cite{dpp}.}
\label{fig:sector}
\end{figure}

\begin{prop}\label{prop:sides}
\begin{enumerate}
  \item The side $s_{k}^{+}$ (resp. $s_{k}^{-}$) is a topological solid cylinder in $\hc \cup \partial\hc$. The intersection of $\partial s_{k}^{+}$ (resp. $\partial s_{k}^{-}$) with $\hc$ is the disjoint union of two topological discs.
  \item The side $s_{k}^{\star}$ (resp. $s_{k}^{\diamond}$) is a topological solid light cone in $\hc \cup \partial\hc$. The intersection of $\partial s_{k}^{\star}$ (resp. $\partial s_{k}^{\diamond}$) with $\hc$ is the light cone.
\end{enumerate}
\end{prop}

\begin{proof}
(1) The side $s_{k}^{+}$ is contained in the isometric sphere $\mathcal{I}_k^{+}$. By Corollary \ref{cor:intersections}, $s_{k}^{+}$ intersects possibly with the sides contained in the isometric spheres $\mathcal{I}_k^{-}$, $\mathcal{I}_{k-1}^{-}$, $\mathcal{I}_k^{\star}$, $\mathcal{I}_{k-1}^{\star}$, $\mathcal{I}_k^{\diamond}$ and $\mathcal{I}_{k+1}^{\diamond}$.

Let $\triangle_1$ be the union of the ridges $s_{k}^{+} \cap s_{k}^{-}$ and $s_{k}^{+} \cap s_{k}^{\star}$, and $\triangle_2$ be the union of the ridges $s_{k}^{+} \cap s_{k-1}^{-}$ and $s_{k}^{+} \cap s_{k}^{\diamond}$. By Proposition \ref{prop:triple}, $\triangle_1$ contains the cross $\mathcal{I}_{k}^{+}\cap\mathcal{I}_{k}^{-}\cap\mathcal{I}_{k}^{\star}$.
By Proposition \ref{prop:ridges}, $\triangle_1$ is a union of four sectors which are patched together along the cross. Hence, $\triangle_1$ is topologically either a disc or a light cone. By a straight computation, the ideal boundary of $\triangle_1$ on $\hc$ is a simple closed curve on the ideal boundary of $\mathcal{I}_{k}^{+}$. See Figure \ref{fig:triple}. Thus $\triangle_1$ is a topological disc. By a similar argument, $\triangle_2$ is a topological disc.

Since $\mathcal{I}_k^{+} \cap \mathcal{I}_k^{\star} \cap \mathcal{I}_{k-1}^{-}=\emptyset$ except when $\theta=\pi/3$ in which case it is a point on $\partial\hc$, $\triangle_1$ and $\triangle_2$ are disjoint except when $\theta=\pi/3$ in which case they intersect at two points on $\partial\hc$. See Figure \ref{figure:cylinder} and Figure \ref{figure:fd}. Note that isometric spheres are topological balls and their pairwise intersections are connected. So, $s_{k}^{+}$ is a topological solid cylinder. See Figure \ref{fig:sideplus}.
$s_{k}^{-}$ can be described by a similar argument.

(2) The side $s_{k}^{\star}$ is contained in the isometric sphere $\mathcal{I}_k^{\star}$. According to Corollary \ref{cor:intersections}, $s_k^{\star}$ only intersects with $s_{k}^{+}$ and $s_{k}^{-}$. Let $\triangle_3$ be the union of $s_{k}^{+} \cap s_{k}^{\star}$ and $s_{k}^{-} \cap s_{k}^{\star}$. By Proposition \ref{prop:triple} and Proposition \ref{prop:ridges}, $\triangle_3$ is a union of four sectors which are patched together along the cross $\mathcal{I}_k^{+} \cap \mathcal{I}_k^{-} \cap \mathcal{I}_k^{\star}$. By a computation for the case $k=0$, one can see that the ideal boundary of $\triangle_3$ is a union of two disjoint simple closed curves in the ideal boundary of $\mathcal{I}_k^{\star}$. See Figure \ref{fig:triple}. Thus $\triangle_3$ is a light cone. Hence, $s_{k}^{\star}$ is topologically a solid light cone. See Figure \ref{fig:sidestar}.
$s_{k}^{\diamond}$ can be described by a similar argument.
\end{proof}

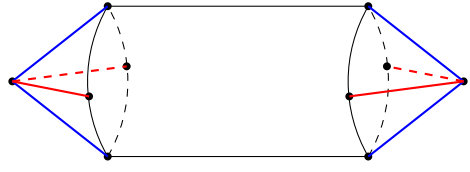
\begin{figure}
\begin{center}
\begin{tikzpicture}
\fill (30:2) circle (1.5pt) (30*5:2) circle (1.5pt) (30*7:2) circle (1.5pt) (30*11:2) circle (1.5pt);
\fill (3,0) circle (1.5pt) (-3,0) circle (1.5pt);
\fill (1.98,0.2) circle (1.5pt) (-1.98,-0.2) circle (1.5pt);
\fill (1.48,-0.2) circle (1.5pt) (-1.48,0.2) circle (1.5pt);
\draw (30:2) -- (30*5:2);
\draw (30*7:2) -- (30*11:2);
\draw (30*5:2) arc (30*5:30*7:2);
\draw[dashed] (30*11:2) arc (30*11:30*13:2);
\draw (30*1:2) arc (30*5:30*7:2);
\draw[dashed] (30*7:2) arc (30*11:30*13:2);
\draw[thick,blue] (3,0) -- (30:2);
\draw[thick,blue] (3,0) -- (30*11:2);
\draw[thick,blue] (-3,0) -- (30*5:2);
\draw[thick,blue] (-3,0) -- (30*7:2);
\draw[thick,red] (3,0) -- (1.48,-0.2);
\draw[thick,dashed,red] (3,0) -- (1.98,0.2);
\draw[thick,red] (-3,0) -- (-1.98,-0.2);
\draw[thick,dashed,red] (-3,0) -- (-1.48,0.2);
\end{tikzpicture}
\end{center}
\caption{A schematic view of the side $s_k^{+}$ ($s_k^{-}$).}
\label{fig:sideplus}
\end{figure}

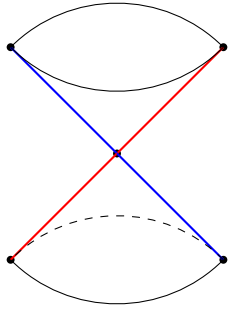
\begin{figure}
\begin{center}
\begin{tikzpicture}
\fill (45:2) circle (1.5pt) (45*3:2) circle (1.5pt) (45*5:2) circle (1.5pt) (45*7:2) circle (1.5pt);
\fill (0,0) circle (1.5pt);
\draw[thick,red] (0,0) -- (45:2);
\draw[thick,blue] (0,0) -- (45*3:2);
\draw[thick,red] (0,0) -- (45*5:2);
\draw[thick,blue] (0,0) -- (45*7:2);
\draw (45:2) arc (45*1:45*3:2);
\draw (45*3:2) arc (45*5:45*7:2);
\draw (45*5:2) arc (45*5:45*7:2);
\draw[dashed] (45*7:2) arc (45*1:45*3:2);
\end{tikzpicture}
\end{center}
\caption{A schematic view of the side $s_k^{\star}$ ($s_k^{\diamond}$).}
\label{fig:sidestar}
\end{figure}

By applying a Poincar\'{e} polyhedron theorem in $\hc$ as stated for example in \cite{par-will2}, \cite{dpp} or \cite{mostow} (see \cite{b} for a version in the hyperbolic plane), we have our main result as follows.
\begin{thm}
Suppose that $\theta\in [0, \pi/3]$. Let $D$ be as in Definition \ref{domain:D}. Then $D$ is a fundamental domain for the cosets of $\langle T \rangle$ in $\Gamma$.
Moreover, $\Gamma$ is discrete and has the presentation
$$
\Gamma=\langle S, T | S^4=(T^{-1}S)^4=id \rangle.
$$
\end{thm}
\begin{proof}
The sides of $D$ are $s_k^{+}$  $s_k^{-}$, $s_k^{\star}$ and $s_k^{\diamond}$.
The ridges of $D$ are $s_{k}^{+} \cap s_{k}^{-}$, $s_{k}^{+} \cap s_{k}^{\star}$, $s_{k}^{+} \cap s_{k-1}^{-}$, $s_{k}^{+} \cap s_{k}^{\diamond}$, $s_{k}^{-} \cap s_{k}^{\star}$ and  $s_{k-1}^{-} \cap s_{k}^{\diamond}$.
To obtain the side-pairing maps and ridge cycles, by applying powers of $T$, it suffices to consider the case where $k=0$.

{\textbf{The side-pairing maps:}}
$s_0^{+}$ is contained in the isometric sphere $\mathcal{I}(S)$ and $s_0^{-}$ in the isometric sphere $\mathcal{I}(S^{-1})$. The ridge $s_0^{+} \cap s_0^{-}$ is contained in the disc $\mathcal{I}(S) \cap \mathcal{I}(S^{-1})$, which is defined by the triple equality
$$
|\langle {\bf z}, {\bf{q}_{\infty}} \rangle|=|\langle {\bf z}, S^{-1}({\bf{q}_{\infty}}) \rangle|=|\langle {\bf z}, S({\bf{q}_{\infty}}) \rangle|.
$$
And the ridge $s_0^{-} \cap s_0^{\star}$ is contained in the disc $\mathcal{I}(S^{-1}) \cap \mathcal{I}(S^2)$, which is defined by the triple equality
$$
|\langle {\bf z}, {\bf{q}_{\infty}} \rangle|=|\langle {\bf z}, S({\bf{q}_{\infty}}) \rangle|=|\langle {\bf z}, S^{-2}({\bf{q}_{\infty}}) \rangle|.
$$

Since $S$ maps $q_{\infty}$ to $S(q_{\infty})$, $S^{-1}(q_{\infty})$ to $q_{\infty}$ and $S(q_{\infty})$ to $S^2(q_{\infty})=S^{-2}(q_{\infty})$,
$S$ maps the disc $\mathcal{I}(S) \cap \mathcal{I}(S^{-1})$ to the disc $\mathcal{I}(S^{-1}) \cap \mathcal{I}(S^2)$.
Note that $\mathcal{I}(S) \cap \mathcal{I}(S^{-1}) \cap \mathcal{I}(S^2)$ is the union of two crossed geodesics (see Proposition \ref{prop:triple}) whose four rays are cyclical permutated by $S$. Since the ridge $s_0^{+} \cap s_0^{-}$  lies in the closure of the exterior of the isometric sphere $\mathcal{I}(S^2)$, according to the equation (\ref{eq:ridge}), the point $q(\pi/3, \pi/6+\theta, \sqrt{2}/2)$ is contained in $s_0^{+} \cap s_0^{-}$. One can easily verify that $S(q(\pi/3, \pi/6+\theta, \sqrt{2}/2))$ lies in the exterior of $\mathcal{I}(S)$. Thus $S(q(\pi/3, \pi/6+\theta, \sqrt{2}/2))$ is contained in $s_0^{-} \cap s_0^{\star}$, which lie in the closure of the exterior of the isometric sphere $\mathcal{I}(S)$.

Hence $S$ maps the ridge $s_0^{+} \cap s_0^{-}$ to the ridge $s_0^{-} \cap s_0^{\star}$. Similarly, $S$ maps $s_0^{+} \cap s_0^{\star}$ to $s_0^{+} \cap s_0^{-}$. See Figure \ref{fig:ridgecycle}.
Since $S$ maps $\mathcal{I}_0^{+} \cap \mathcal{I}_{-1}^{-} \cap \mathcal{I}_0^{\diamond}$ to $\mathcal{I}_0^{-} \cap \mathcal{I}_{1}^{+} \cap \mathcal{I}_1^{\diamond}$, a similar argument shows that $S$ maps $s_0^{+} \cap s_{-1}^{-}$ to $s_0^{-} \cap s_{1}^{\diamond}$ and $s_0^{+} \cap s_0^{\diamond}$ to $s_0^{-} \cap s_{1}^{+}$.

\begin{figure}
\begin{center}
\begin{tikzpicture}
\fill (45:2) circle (1.5pt) (45*3:2) circle (1.5pt) (45*5:2) circle (1.5pt) (45*7:2) circle (1.5pt);
\fill (0,0) circle (1.5pt);
\draw[thick,red] (0,0) -- (45:2);
\draw[thick,blue] (0,0) -- (45*3:2);
\draw[thick,red] (0,0) -- (45*5:2);
\draw[thick,blue] (0,0) -- (45*7:2);
\draw (45:2) arc (45*1:45*3:2);
\draw (45*3:2) arc (45*5:45*7:2);
\draw (45*5:2) arc (45*5:45*7:2);
\draw[dashed] (45*7:2) arc (45*1:45*3:2);
\draw (45:2) -- (45*7:2);
\draw (45*3:2) -- (45*5:2);
\coordinate [label=right:$s_0^{+}\cap s_0^{-}$] (r1) at (-1.4,0);
\coordinate [label=right:$s_0^{+}\cap s_0^{-}$] (r2) at (0.2,0);
\coordinate [label=right:$s_0^{+}\cap s_0^{\star}$] (r3) at (-4,0);
\coordinate [label=right:$s_0^{-}\cap s_0^{\star}$] (r4) at (3,0);
\draw[->] (0,0.7) -- (-2.5,0);
\draw[->] (0,-1.6) -- (-2.5,-0.2);
\draw[->,dashed] (0,1.6) -- (3,0.2);
\draw[->,dashed] (0,-0.5) -- (3,-0.2);
\end{tikzpicture}
\end{center}
\caption{A schematic view of the ridges $s_0^{+}\cap s_0^{-}$, $s_0^{+}\cap s_0^{\star}$ and $s_0^{-}\cap s_0^{\star}$. The thick cross is the triple intersection $\mathcal{I}_{0}^{+}\cap\mathcal{I}_{0}^{-}\cap\mathcal{I}_{0}^{\star}$, that is the union $\mathcal{L}_1\cup \mathcal{C}_1 \cup \mathcal{C}_2$ described in Proposition \ref{prop:triple}.}
\label{fig:ridgecycle}
\end{figure}
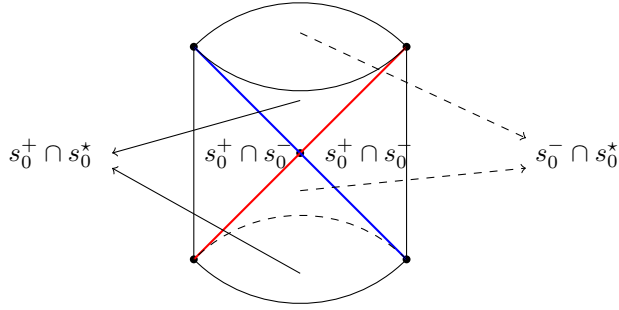

By a similar argument, $s_0^{\star}$ (respectively, $s_0^{\diamond}$) is mapped to itself by the elliptic element of order two $S^2$ (respectively, $(T^{-1}S)^2=(S^{-1}T)^2$), which sends $s_0^{+} \cap s_0^{\star}$ to $s_0^{-} \cap s_0^{\star}$ (respectively, $s_0^{\diamond} \cap s_{-1}^{-}$ to $s_0^{+} \cap s_0^{\diamond}$) and vice-versa.

Hence, the side-pairing maps are:
\begin{eqnarray*}
  T^{k}ST^{-k}: & s_k^{+}\longrightarrow s_k^{-} \\
  T^{k}S^2T^{-k}: & s_k^{\star}\longrightarrow s_k^{\star} \\
  T^{k}(T^{-1}S)^2T^{-k}: & s_k^{\diamond}\longrightarrow s_k^{\diamond}.
\end{eqnarray*}

\textbf{The cycle transformations:}
According to the side-pairing maps, the ridge cycles are:
\begin{eqnarray*}
&(s_k^{+} \cap s_k^{-}, s_k^{+} , s_k^{-}) \xrightarrow{T^{k}ST^{-k}}
(s_k^{\star} \cap s_k^{-}, s_k^{\star} , s_k^{-})  \xrightarrow{T^{k}S^2T^{-k}} \\
&(s_k^{+} \cap s_k^{\star}, s_k^{+} , s_k^{\star})  \xrightarrow{T^{k}ST^{-k}}
(s_k^{+} \cap s_k^{-}, s_k^{+} , s_k^{-}),
\end{eqnarray*}
and
\begin{eqnarray*}
  &(s_k^{+} \cap s_{k-1}^{-}, s_k^{+} , s_{k-1}^{-})  \xrightarrow{T^{k}(T^{-1}S)T^{-k}}
(s_k^{\diamond} \cap s_{k-1}^{-}, s_k^{\diamond} , s_{k-1}^{-})  \xrightarrow{T^{k}(T^{-1}S)^{2}T^{-k}} \\
  &(s_k^{+} \cap s_{k}^{\diamond}, s_k^{+} , s_{k}^{\diamond})  \xrightarrow{T^{k}(T^{-1}S)T^{-k}}
(s_k^{+} \cap s_{k-1}^{-}, s_k^{+} , s_{k-1}^{-}).
\end{eqnarray*}

Thus the cycle transformations are
$$
T^{k}ST^{-k}\cdot T^{k}S^2T^{-k}\cdot T^{k}ST^{-k} =T^{k}S^{4}T^{-k},
$$
and
$$
T^{k}(T^{-1}S)T^{-k}\cdot T^{k}(T^{-1}S)^{2}T^{-k}\cdot T^{k}(T^{-1}S)T^{-k} =T^{k}(T^{-1}S)^4T^{-k},
$$
which are equal to the identity map, since we have $S^4=id$ and $(T^{-1}S)^4=id$.

\textbf{The local tessellation}:
There are exactly two copies of $D$ along each side, since the sides are contained in isometrical spheres and the side-pairing maps send the exteriors to the interiors.
Thus there is nothing to verify for the points in the interior of every side.

According to the ridge cycles and cycle transformations, there are exactly three copies of $D$ along each ridge.
\begin{itemize}
  \item $s_k^{+} \cap s_k^{-}$, $s_k^{\star} \cap s_k^{-}$ and $s_k^{+} \cap s_k^{\star}$: These three ridges are in one cycle. Thus we only need to consider the ridge $s_k^{+} \cap s_k^{-}$. Since the cycle transformation of $s_k^{+} \cap s_k^{-}$ is
      $$T^{k}ST^{-k}\cdot T^{k}S^2T^{-k}\cdot T^{k}ST^{-k} =id,$$
      the three copies of $D$ along $s_k^{+} \cap s_k^{-}$ are $D$, $T^{k}S^{-1}T^{-k}(D)$ and $T^{k}ST^{-k}(D)$.
      We know that $s_k^{+} \cap s_k^{-}$ is contained in $\mathcal{I}(T^{k}ST^{-k}) \cap \mathcal{I}(T^{k}S^{-1}T^{-k})$, which is defined by the triple equality
$$
|\langle {\bf z}, {\bf{q}_{\infty}} \rangle|=|\langle {\bf z}, T^kS^{-1}T^{-k}({\bf{q}_{\infty}}) \rangle|=|\langle {\bf z}, T^kST^{-k}({\bf{q}_{\infty}}) \rangle|.
$$

For $z$ in the neighborhoods of $s_k^{+} \cap s_k^{-}$ in $D$, $|\langle {\bf z}, {\bf{q}_{\infty}} \rangle|$ is the smallest of the three quantities in the above triple equality.

For $z$ in the neighborhoods of $s_k^{\star} \cap s_k^{-}$ in $D$,  $|\langle {\bf z}, {\bf{q}_{\infty}} \rangle|$ is no more than $|\langle {\bf z},  T^kST^{-k}({\bf{q}_{\infty}}) \rangle|$ and $|\langle {\bf z}, T^kS^{-2}T^{-k}({\bf{q}_{\infty}}) \rangle|$. Applying $T^{k}S^{-1}T^{-k}$ gives neighborhood of $s_k^{+} \cap s_k^{-}$ in $T^{k}S^{-1}T^{-k}(D)$
where $|\langle {\bf z}, T^kS^{-1}T^{-k}({\bf{q}_{\infty}}) \rangle|$ is the smallest of the three quantities in the above triple equality.

For $z$ in the neighborhoods of $s_k^{+} \cap s_k^{\star}$ in $D$, $|\langle {\bf z}, {\bf{q}_{\infty}} \rangle|$ is no more than $|\langle {\bf z},  T^kS^{-1}T^{-k}({\bf{q}_{\infty}}) \rangle|$ and $|\langle {\bf z}, T^kS^{-2}T^{-k}({\bf{q}_{\infty}}) \rangle|$. Applying $T^{k}ST^{-k}$ gives neighborhood of $s_k^{+} \cap s_k^{-}$ in $T^{k}ST^{-k}(D)$
where $|\langle {\bf z}, T^kST^{-k}({\bf{q}_{\infty}}) \rangle|$ is the smallest of the three quantities in the above triple equality.

      Thus the union of $D$, $T^{k}S^{-1}T^{-k}(D)$ and $T^{k}ST^{-k}(D)$ forms a regular neighborhood of each point in $s_k^{+} \cap s_k^{-}$.
  \item $s_k^{+} \cap s_{k-1}^{-}$, $s_k^{\diamond} \cap s_{k-1}^{-}$ and $s_k^{+} \cap s_{k}^{\diamond}$: We only need to consider the ridge $s_k^{+} \cap s_{k-1}^{-}$. Since the cycle transformation of $s_k^{+} \cap s_{k-1}^{-}$ is
      $$T^{k}(T^{-1}S)T^{-k}\cdot T^{k}(T^{-1}S)^{2}T^{-k}\cdot T^{k}(T^{-1}S)T^{-k}=id,$$
      the union of $D$, $T^{k}(T^{-1}S)^{-1}T^{-k}(D)$ and $T^{k}(T^{-1}S)T^{-k}(D)$ forms a regular neighborhood of each point in $s_k^{+} \cap s_{k-1}^{-}$, by a similar argument as in the first item.
\end{itemize}

\textbf{Consistent system of horoballs}:
When $\theta=\frac{\pi}{3}$, there are accidental ideal vertices on $D$. The sides $s_k^{\star}$ and $s_{k+1}^{\star}$ will be asymptotic on $\partial \hc$ at the fixed point of the parabolic element $T^k(S^2T^{-1})T^{-k}$, and the sides $s_k^{\diamond}$ and $s_{k+1}^{\diamond}$ will be asymptotic on $\partial \hc$ at the fixed point of the parabolic element $T^k(ST^{-1}S)T^{-k}$. To show that there will be consistent system of horoballs, it suffices to show that all the cycle transformations fixing a given cusp is non-loxodromic.

Let $p_2$ be the fixed point of $T^{-1}S^2$ and $q_2$ be the fixed point of $(T^{-1}S)^2T$ (the coordinates of $p_2$ and $q_2$ are given in Definition \ref{def:points}, or see Figure \ref{figure:cylinder}). Then all the accidental ideal vertices $\{T^k(p_2)\}$ and $\{T^{k}(q_2)\}$ are related by the side-pairing maps as the following:
$$
\xymatrix{
\ar[r] & T^{-1}(q_2) \ar[d] \ar[r] & q_2 \ar[d]|-{T^{-1}ST} \ar[r]^{(S^{-1}T)^2}
                & T(q_2) \ar[d]^{S}  \ar[r] & T^2(q_2) \ar[d] \ar[r] &  \\
\ar[r] & T^{-1}(p_2)\ar[ur] \ar[r] & p_2 \ar[ur]_{S} \ar[r]_{S^2}
                & T(p_2) \ar[ur] \ar[r] & T^{2}(p_2)  \ar[r]    &     }
$$
Thus, up to powers of $T$, all the cycle transformations are $S\cdot S\cdot S^{-2} =id$ and
$$
T^{-1}ST\cdot (S^{-1}T)^{-2}\cdot S= (T^{-1}S^2)^2 = (I_1I_3I_2I_3)^2,
$$
which is parabolic.
This means that $p_2$ is fixed by the parabolic element $(T^{-1}S^2)^2$.

Therefore, $D$ is a fundamental domain for the cosets of $\langle T \rangle$ in $\Gamma$. The side-pairing maps and $T$ will generate the group $\Gamma$. The reflection relations are $(T^k S^2 T^{-k})^2=id$ and $(T^k(T^{-1}S)^2T^{-k})^2=id$. The cycle relations are $T^kS^4T^{-k}=id$ and $T^k(T^{-1}S)^4T^{-k}=id$. Thus $\Gamma$ is discrete and has the presentation
$$
\Gamma=\langle S, T | S^4=(T^{-1}S)^4=id \rangle.
$$
\end{proof}

Since $\Gamma$ is a subgroup of $\langle I_1, I_2, I_3 \rangle$ of index 2, as a corollary, we have
\begin{cor}
Let $\langle I_1, I_2, I_3 \rangle$ be a complex hyperbolic $(4,4,\infty)$ triangle group as in Proposition \ref{prop:trianle}.
Then $\langle I_1, I_2, I_3 \rangle$ is discrete and faithful if and only if $I_1I_3I_2I_3$ is non-elliptic.
\end{cor}

This answers Conjecture \ref{conj:schwartz} on the complex hyperbolic $(4,4,\infty)$ triangle group.

\section{The manifold at infinity}\label{sec:manifold}

In this section, we study the group $\Gamma$ in the case when $\theta=\pi/3$. That is the group $\Gamma=\langle S, T \rangle=\langle I_2I_3, I_2I_1 \rangle$ with $T^{-1}S^2=I_1I_3I_2I_3$ being parabolic.

In this case, the Ford domain $D$ has additional ideal vertices on $\partial\hc$, which are parabolic fixed points corresponding to the conjugators of $T^{-1}S^2$. By intersecting a fundamental domain for $\langle T \rangle$ acting on $\partial\hc$ with the ideal boundary of $D$, we obtain a fundamental domain for $\Gamma$ acting on its discontinuity region $\Omega(\Gamma)$.

Topologically, this fundamental domain is the  unknotted cylinder cross a ray (see Proposition \ref{prop:cylinder}). By cutting and gluing we obtain two polyhedra $\mathcal{P}_{+}$ and $\mathcal{P}_{-}$ (see Proposition \ref{prop:polyhedra}). Gluing $\mathcal{P}_{-}$ to $\mathcal{P}_{+}$ by $S^{-1}$, we obtain a polyhedron $\mathcal{P}$. By studying the combinatorial properties of $\mathcal{P}$, we show that the quotient $\Omega(\Gamma)/\Gamma$ is homeomorphic to the two-cusped hyperbolic 3-manifold $s782$.

\begin{rem}
We will use $\tilde{s}_k^{+}$ (respectively, $\tilde{s}_k^{-}$, $\tilde{s}_k^{\star}$, and $\tilde{s}_k^{\diamond}$) to denote the ideal boundary of the side of $D$ contained in the ideal boundary of the isometric sphere $\mathcal{I}_k^{+}$ (respectively, $\mathcal{I}_k^{-}$, $\mathcal{I}_k^{\star}$ and $\mathcal{I}_k^{\diamond}$).
\end{rem}

\begin{defn}\label{def:points}
In the Heisenberg coordinate, we define the  points
 \begin{eqnarray*}
 q_2&=& \left[(-3+i\sqrt{3})/2, -\sqrt{3}\right],\\
 q_3&=&\left[(1+i\sqrt{3})/2, \sqrt{3}\right],\\
 p_2&=&\left[(-1+i\sqrt{3})/2, -\sqrt{3}\right],\\
 p_3&=&\left[(3+i\sqrt{3})/2, \sqrt{3}\right],\\
 p_4&=&\left[\frac{1}{6}\left(4+\sqrt{6}+i(4\sqrt{3}-\sqrt{2})\right), 0 \right],\\
 p_5&=&\left[\frac{1}{6}\left(4-\sqrt{6}+i(4\sqrt{3}+\sqrt{2})\right), 0 \right], \\
 p_6&=&\left[\frac{1}{6}\left(2-\sqrt{6}+i(2\sqrt{3}+\sqrt{2})\right), -\frac{4\sqrt{2}}{3} \right],\\
 p_7&=&\left[\frac{1}{6}\left(2+\sqrt{6}+i(2\sqrt{3}-\sqrt{2})\right), \frac{4\sqrt{2}}{3} \right],\\
 p_8 &=& \left[\frac{1}{6}\left(-2+\sqrt{6}+i(2\sqrt{3}+\sqrt{2})\right), \frac{4\sqrt{2}}{3} \right],\\
 p_9&=& \left[\frac{1}{6}\left(-2-\sqrt{6}+i(2\sqrt{3}-\sqrt{2})\right), -\frac{4\sqrt{2}}{3} \right],\\
p_{10}&=& \left[\frac{1}{6}\left(-4+\sqrt{6}+i(4\sqrt{3}+\sqrt{2})\right), 0 \right],\\
p_{11}&=& \left[\frac{1}{6}\left(-4-\sqrt{6}+i(4\sqrt{3}-\sqrt{2})\right), 0 \right],
 \end{eqnarray*}
and  the other points
$$p_{12}=T(p_9), \quad p_{13}=T(p_8), \quad p_{14}=T(p_{11}), \quad p_{15}=T(p_{10}).$$

\end{defn}

By Proposition \ref{prop:triple} and Corollary \ref{cor:intersections}, we have the following.
\begin{prop}\label{prop:points}
The points defined in Definition \ref{def:points} have the following properties.
\begin{itemize}
  \item $p_4,p_5,p_6,p_7$ are the four points on the ideal boundary of $\mathcal{I}_{0}^{+} \cap \mathcal{I}_{0}^{-} \cap \mathcal{I}_{0}^{\star}$, which are described in Proposition \ref{prop:triple};
  \item $p_8,p_9,p_{10},p_{11}$ are the four points on the ideal boundary of $\mathcal{I}_{0}^{+} \cap \mathcal{I}_{-1}^{-} \cap \mathcal{I}_{0}^{\diamond}$;
  \item $p_{12}, p_{13},p_{14},p_{15}$ are the four points on the ideal boundary of $\mathcal{I}_{1}^{+} \cap \mathcal{I}_{0}^{-} \cap \mathcal{I}_{1}^{\diamond}$;
  \item $p_2$ (resp. $p_3$) is the parabolic fixed point of $T^{-1}S^2$ (resp. $S^2T^{-1}$), which is the intersection of four isometric spheres $\mathcal{I}_{0}^{+} \cap \mathcal{I}_{-1}^{-} \cap \mathcal{I}_{0}^{\star} \cap \mathcal{I}_{-1}^{\star}$ (resp. $\mathcal{I}_{1}^{+} \cap \mathcal{I}_{0}^{-} \cap \mathcal{I}_{0}^{\star} \cap \mathcal{I}_{1}^{\star}$);
  \item $q_3$ (resp. $q_2$) is the parabolic fixed point of $ST^{-1}S$ (resp. $T^{-1}ST^{-1}ST$), which is the intersection of the four isometric spheres $\mathcal{I}_{0}^{+} \cap \mathcal{I}_{0}^{-} \cap \mathcal{I}_{0}^{\diamond} \cap \mathcal{I}_{1}^{\diamond}$ (resp. $\mathcal{I}_{-1}^{+} \cap \mathcal{I}_{-1}^{-} \cap \mathcal{I}_{0}^{\diamond} \cap \mathcal{I}_{-1}^{\diamond}$).
\end{itemize}
\end{prop}

\begin{proof}
As described in Proposition \ref{prop:triple}, all of the triple intersections $\mathcal{I}_{0}^{+} \cap \mathcal{I}_{0}^{-} \cap \mathcal{I}_{0}^{\star}$, $\mathcal{I}_{0}^{+} \cap \mathcal{I}_{-1}^{-} \cap \mathcal{I}_{0}^{\diamond}$ and $\mathcal{I}_{1}^{+} \cap \mathcal{I}_{0}^{-} \cap \mathcal{I}_{1}^{\diamond}$ have exactly four points lying on $\partial\hc$. When write the standard lifts of $p_4,p_5,p_6,p_7$, one can see that they are the four points in the proof of Proposition \ref{prop:triple}. Thus the first item is proved.

By Proposition \ref{prop:symmetry}, the four points of $\mathcal{I}_{0}^{+} \cap \mathcal{I}_{-1}^{-} \cap \mathcal{I}_{0}^{\diamond}$ are the images of $p_4,p_5,p_6,p_7$ under the antiholomorphic involution $\tau$, which are $p_{11},p_{10},p_8,p_9$.

The second item and the fact that $\mathcal{I}_{1}^{+} \cap \mathcal{I}_{0}^{-} \cap \mathcal{I}_{1}^{\diamond}=T(\mathcal{I}_{0}^{+} \cap \mathcal{I}_{-1}^{-} \cap \mathcal{I}_{0}^{\diamond})$ imply the third item.

We will only prove the statement for $p_2$ in the last two items. The others can be described by similar arguments. By Lemma \ref{lem:triple}, $p_2$ is the parabolic fixed point of $T^{-1}S^2$ and is the triple intersection $\mathcal{I}_{0}^{+} \cap \mathcal{I}_{-1}^{-} \cap \mathcal{I}_{0}^{\star}$. By Corollary \ref{cor:intersections} (3), $\mathcal{I}_{0}^{\star}$ is tangent with $\mathcal{I}_{-1}^{\star}$ at $p_2$. This completes the proof.
\end{proof}

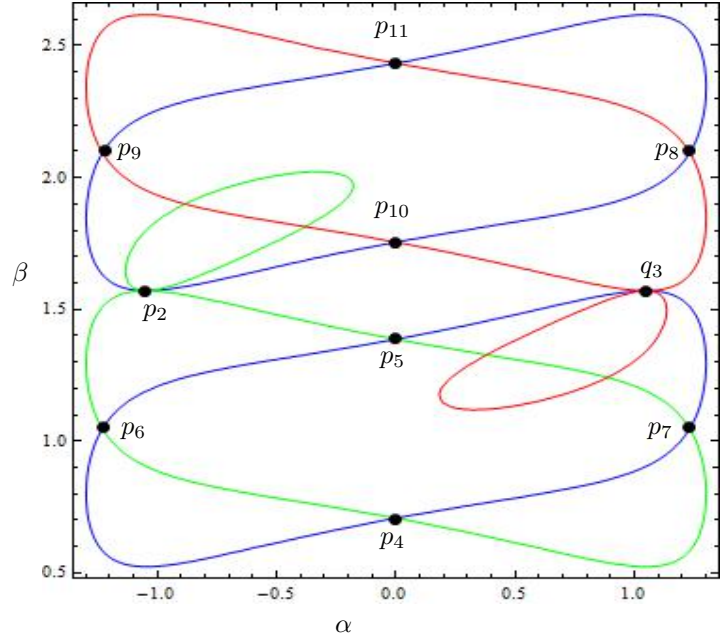
\begin{figure}
\begin{center}
\begin{tikzpicture}
\node at (0,0) {\includegraphics[width=9cm,height=8cm]{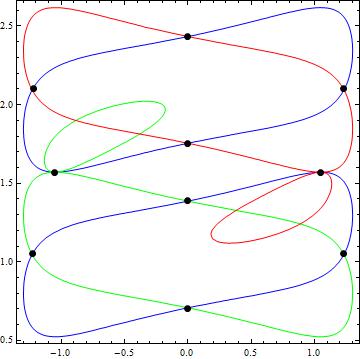}};
\coordinate [label=left:$p_9$] (S) at (-3.05,2);
\coordinate [label=above:$p_{11}$] (S) at (0.15,3.4);
\coordinate [label=above:$p_{10}$] (S) at (0.15,1);
\coordinate [label=above:$p_{5}$] (S) at (0.15,-1);
\coordinate [label=above:$p_{4}$] (S) at (0.15,-3.4);
\coordinate [label=below:$p_{2}$] (S) at (-3,0.1);
\coordinate [label=left:$p_{6}$] (S) at (-3,-1.7);
\coordinate [label=left:$p_{7}$] (S) at (4,-1.7);
\coordinate [label=left:$p_{8}$] (S) at (4.1,2);
\coordinate [label=above:$q_{3}$] (S) at (3.6,0.2);
\coordinate [label=above:$\alpha$] (S) at (-0.5,-4.5);
\coordinate [label=above:$\beta$] (S) at (-4.8,0.1);;
\end{tikzpicture}
\end{center}
  \caption{ Intersections of the isometric spheres $\mathcal{I}_0^{-}$, $\mathcal{I}_{-1}^{-}$, $\mathcal{I}_0^{\ast}$, $\mathcal{I}_{-1}^{\ast}$, $\mathcal{I}_0^{\diamond}$, $\mathcal{I}_1^{\diamond}$ with $\mathcal{I}_0^{+}$ in $\partial\hc$, viewed in geographical coordinates. Here $\alpha \in [-\pi/2,\pi/2]$ in the vertical coordinate and  $\beta \in [0,\pi]$ in the horizontal one. The region exterior to  the six  Jordan closed curves has two connect components. One of them is the topological octagon  with vertices $p_2$, $p_6$, $p_4$, $p_7$, $q_3$, $p_8$, $p_{11}$, $p_9$.  The other one is a topological quadrilateral  with vertices $p_2$, $p_5$, $q_3$, $p_{10}$. }
  \label{figure:ideal-side-plus}
\end{figure}

Now we study the combinatorial properties of the sides. See Figure \ref{figure:ideal-side-plus} and Figure \ref{figure:cylinder}.
\begin{prop}\label{prop:side-plus}
The interior of the side $\tilde{s}_0^{+}$ has two connected components.
\begin{itemize}
  \item One of them is an octagon, denoted by $\mathcal{O}_0^{+}$, whose vertices are $p_2$, $p_6$, $p_4$, $p_7$, $q_3$, $p_8$, $p_{11}$ and $p_9$.
  \item The other one is a quadrilateral, denoted by $\mathcal{Q}_0^{+}$, whose vertices are $p_2$, $p_5$, $q_3$ and $p_{10}$.
\end{itemize}
\end{prop}
\begin{proof}
By Proposition \ref{prop:sides}, when $\theta< \pi/3$, the side $\tilde{s}_0^{+}$ is topologically an annulus bounded by two disjoint simple closed curves which are the union of
the ideal boundaries of the ridges $s_{0}^{+} \cap s_{-1}^{-}$ and $s_{0}^{+}  \cap s_{0}^{\diamond}$, respectively $s_{0}^{+} \cap s_{0}^{-}$ and $s_{0}^{+}  \cap s_{0}^{\star}$.
When $\theta=\pi/3$, these two curves intersect at two points, which divide $\tilde{s}_0^{+}$ into two parts. That is to say the interior of the side $\tilde{s}_0^{+}$ has two connected components.

By Proposition \ref{prop:points}, the ideal boundary of the ridge $s_{0}^{+} \cap s_{-1}^{-}$ (resp. $s_{0}^{+}  \cap s_{0}^{\diamond}$) is a union of two disjoint Jordan arcs $[p_9,p_{10}]$ and $[p_8, p_{11}]$ (resp. $[p_{10},p_8]$ and $[p_{11},p_9]$), the ideal boundary of the ridge $s_{0}^{+} \cap s_{0}^{-}$ (resp. $s_{0}^{+}  \cap s_{0}^{\star}$) is a union of two disjoint Jordan arcs $[p_5,p_7]$ and $[p_4,p_6]$ (resp. $[p_7,p_4]$ and $[p_6,p_5]$).
Since $p_2$ is the intersection of four isometric spheres $\mathcal{I}_{0}^{+} \cap \mathcal{I}_{-1}^{-} \cap \mathcal{I}_{0}^{\star} \cap \mathcal{I}_{-1}^{\star}$, it lies on $[p_9,p_{10}]$ and $[p_6,p_5]$. Similarly, $q_3$ lies on $[p_5,p_7]$ and $[p_{10},p_8]$.

By Proposition \ref{prop:symmetry}, the antiholomorphic involution $\tau$ preserves $\tilde{s}_0^{+}$ and interchanges its boundaries.
It is easy to check that $\tau$ interchanges $p_2$ and $q_3$, $p_5$ and $p_{10}$, $p_4$ and $ p_{11}$, $p_6$ and $p_{8}$, $p_7$ and $p_9$.
Thus one part of $\tilde{s}_0^{+}$ is a quadrilateral with vertices $p_2$, $p_5$, $q_3$ and $p_{10}$, denoted by $\mathcal{O}_0^{+}$.
The other one is an octagon with vertices $p_2$, $p_6$, $p_4$, $p_7$, $q_3$, $p_8$, $p_{11}$ and $p_9$, denoted by $\mathcal{Q}_0^{+}$.
Both of them are preserved by $\tau$.
\end{proof}

According to the symmetry $I_2$ in Proposition \ref{prop:symmetry}, which interchanges $\mathcal{I}_{0}^{+} $ and $ \mathcal{I}_{0}^{-}$, we have the following.
\begin{prop}\label{prop:side-minus}
The interior of the side $\tilde{s}_0^{-}$ has two connected components.
\begin{itemize}
  \item One of them is an octagon, denoted by $\mathcal{O}_0^{-}$, whose vertices are $p_5$, $p_6$, $p_4$, $p_3$, $p_{15}$, $p_{13}$, $p_{14}$ and $q_3$.
  \item The other is a quadrilateral, denoted by $\mathcal{Q}_0^{-}$, whose vertices are $p_3$, $p_7$, $q_3$ and $p_{12}$.
\end{itemize}
\end{prop}
\begin{proof}
Note that side $s_0^{-}$ is bounded by the ridges $s_{0}^{-} \cap s_{1}^{+}$, $s_{0}^{-}  \cap s_{1}^{\diamond}$, $s_{0}^{-} \cap s_{0}^{+}$ and $s_{0}^{-}  \cap s_{0}^{\star}$.
By Proposition \ref{prop:symmetry}, the side $s_0^{-}$ is isometric to $s_0^{+}$ under the complex involution $I_2$. Thus its ideal boundary $\tilde{s}_0^{-}$ will be also isometric to $\tilde{s}_0^{+}$. This implies that $\tilde{s}_0^{-}$ has the same combinatorial properties as $\tilde{s}_0^{+}$.
One can check that
$$
I_2 : (q_3,p_5,p_2,p_{10},p_8, p_{11}, p_9, p_6) \leftrightarrow (q_3, p_7, p_3,p_{12}, p_{14}, p_{13}, p_{15}, p_4).
$$
Thus one part of $\tilde{s}_0^{-}$ is an octagon, denoted by $\mathcal{O}_0^{-}$, whose vertices are $p_5$, $p_6$, $p_4$, $p_3$, $p_{15}$, $p_{13}$, $p_{14}$ and $q_3$.
The other one is a quadrilateral, denoted by $\mathcal{Q}_0^{-}$, whose vertices are $p_3$, $p_7$, $q_3$ and $p_{12}$. See Figure \ref{figure:cylinder}.
\end{proof}

\begin{rem}
$q_3$ lies on the $\mathbb{C}$-circle associated to $I_2$, that is the ideal boundary of the complex line fixed by $I_2$.
One can also observe that $p_2$ is fixed by $I_1$.
\end{rem}

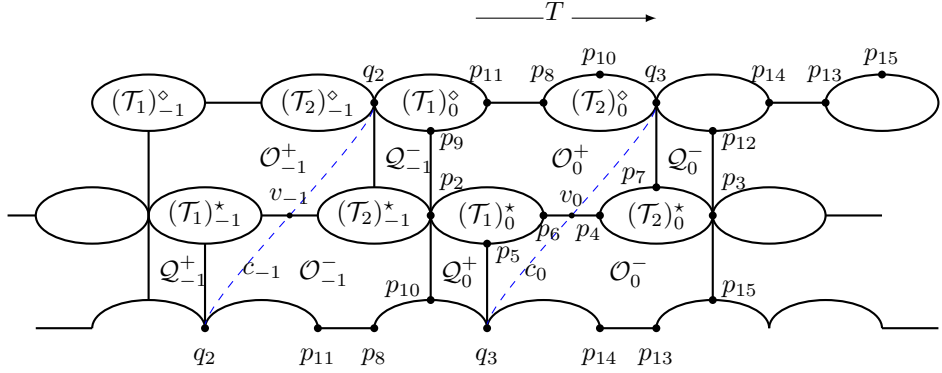
\begin{figure}
\begin{center}
\begin{tikzpicture}[scale=0.75]
\draw (-2.2,3.5) -- (-1,3.5);
\coordinate [label=above:$T$] (T) at (-0.8,3.3);
\draw[-latex] (-0.5,3.5) -- (1,3.5);

\draw[thick] (-9,0) ellipse (1 cm and 0.5 cm);
\draw[thick] (-7,0) ellipse (1 cm and 0.5 cm);
\draw[thick] (-4,0) ellipse (1 cm and 0.5 cm);
\draw[thick] (-2,0) ellipse (1 cm and 0.5 cm);
\draw[thick] (1,0) ellipse (1 cm and 0.5 cm);
\draw[thick] (3,0) ellipse (1 cm and 0.5 cm);
\draw[thick] (-3,2) ellipse (1 cm and 0.5 cm);
\draw[thick] (-5,2) ellipse (1 cm and 0.5 cm);
\draw[thick] (-8,2) ellipse (1 cm and 0.5 cm);
\draw[thick] (0,2) ellipse (1 cm and 0.5 cm);
\draw[thick] (2,2) ellipse (1 cm and 0.5 cm);
\draw[thick] (5,2) ellipse (1 cm and 0.5 cm);
\draw[thick] (-2,-2) arc (0:180:1 cm and 0.5 cm);
\draw[thick] (0,-2) arc (0:180:1 cm and 0.5 cm);
\draw[thick] (3,-2) arc (0:180:1 cm and 0.5 cm);
\draw[thick] (5,-2) arc (0:180:1 cm and 0.5 cm);
\draw[thick] (-5,-2) arc (0:180:1 cm and 0.5 cm);
\draw[thick] (-7,-2) arc (0:180:1 cm and 0.5 cm);
\draw[thick] (-3,1.5) -- (-3,0) -- (-3,-1.5);
\draw[thick] (-8,1.5) -- (-8,0) -- (-8,-1.5);
\draw[thick] (2,1.5) -- (2,0) -- (2,-1.5);
\draw[thick] (-2,-0.5) -- (-2,-2);
\draw[thick] (-4,0.5) -- (-4,2);
\draw[thick] (-7,-0.5) -- (-7,-2);
\draw[thick] (1,0.5) -- (1,2);
\draw[thick] (-1,0) -- (0,0);
\draw[thick] (-6,0) -- (-5,0);
\draw[thick] (4,0) -- (5,0);
\draw[thick] (-10.5,0) -- (-10,0);
\draw[thick] (-2,2) -- (-1,2);
\draw[thick] (-7,2) -- (-6,2);
\draw[thick] (0,-2) -- (1,-2);
\draw[thick] (3,2) -- (4,2);
\draw[thick] (-5,-2) -- (-4,-2);
\draw[thick] (-10,-2) -- (-9,-2);
\draw[thick] (5,-2) -- (6,-2);

\draw[dashed,blue] (-7,-2) .. controls (-7.1,-1.5) and (-3.9,1.5) .. (-4,2);
\draw[dashed,blue] (-2,-2) .. controls (-2.1,-1.5) and (1.1,1.5) .. (1,2);

\fill (-4,2) circle (2pt);
\fill (-3,1.5) circle (2pt);
\fill (-2,2) circle (2pt);
\fill (-1,2) circle (2pt);
\fill (0,2.5) circle (2pt);
\fill (1,2) circle (2pt);
\fill (2,1.5) circle (2pt);
\fill (3,2) circle (2pt);
\fill (4,2) circle (2pt);
\fill (5,2.5) circle (2pt);
\fill (-3,0) circle (2pt);
\fill (-2,-0.5) circle (2pt);
\fill (-1,0) circle (2pt);
\fill (0,0) circle (2pt);
\fill (1,0.5) circle (2pt);
\fill (2,0) circle (2pt);
\fill (-3,-1.5) circle (2pt);
\fill (-2,-2) circle (2pt);
\fill (0,-2) circle (2pt);
\fill (1,-2) circle (2pt);
\fill (2,-1.5) circle (2pt);
\fill (-7,-2) circle (2pt);
\fill (-5,-2) circle (2pt);
\fill (-4,-2) circle (2pt);
\coordinate [label=above:$q_2$] (q2) at (-4,2.2);
\coordinate [label=right:$p_9$] (p9) at (-3,1.3);
\coordinate [label=above:$p_{11}$] (p11) at (-2,2.2);
\coordinate [label=above:$p_8$] (p8) at (-1,2.2);
\coordinate [label=above:$p_{10}$] (p10) at (0,2.5);
\coordinate [label=above:$q_3$] (q3) at (1,2.2);
\coordinate [label=right:$p_{12}$] (p12) at (2,1.3);
\coordinate [label=above:$p_{14}$] (p14) at (3,2.2);
\coordinate [label=above:$p_{13}$] (p13) at (4,2.2);
\coordinate [label=above:$p_{15}$] (p15) at (5,2.5);
\coordinate [label=right:$p_2$] (p2) at (-3,0.6);
\coordinate [label=right:$p_5$] (p5) at (-2,-0.7);
\coordinate [label=below:$p_6$] (p6) at (-0.9,0);
\coordinate [label=above:$v_0$] (v0) at (-0.5,0);
\fill (-0.5,0) circle (1.5pt);
\coordinate [label=below:$p_4$] (p4) at (-0.2,0);
\coordinate [label=left:$p_7$] (p7) at (1,0.7);
\coordinate [label=right:$p_3$] (p3) at (2,0.6);
\coordinate [label=below:$q_2$] (q2) at (-7,-2.2);
\coordinate [label=below:$p_{11}$] (p11) at (-5,-2.2);
\coordinate [label=below:$p_8$] (p8) at (-4,-2.2);
\coordinate [label=left:$p_{10}$] (p10) at (-3,-1.3);
\coordinate [label=below:$q_3$] (q3) at (-2,-2.2);
\coordinate [label=below:$p_{14}$] (p14) at (0,-2.2);
\coordinate [label=below:$p_{13}$] (p13) at (1,-2.2);
\coordinate [label=right:$p_{15}$] (p15) at (2,-1.3);
\coordinate [label=left:$\mathcal{O}^{+}_{0}$] (O^{+}_0) at (0,1);
\coordinate [label=right:$\mathcal{O}^{-}_{0}$] (O^{-}_0) at (0,-1);
\coordinate [label=right:$\mathcal{Q}^{+}_{0}$] (Q^{+}_0) at (-3,-1);
\coordinate [label=right:$\mathcal{Q}^{-}_{0}$] (Q^{-}_0) at (1,1);
\coordinate [label=above:$(\mathcal{T}_1)^{\star}_{0}$] (T_1^{1}) at (-2,-0.4);
\coordinate [label=above:$(\mathcal{T}_2)^{\star}_{0}$] (T_2^{1}) at (1,-0.4);
\coordinate [label=below:$(\mathcal{T}_1)^{\diamond}_{0}$] (T_1^{2}) at (-3,2.4);
\coordinate [label=below:$(\mathcal{T}_2)^{\diamond}_{0}$] (T_2^{2}) at (0,2.4);

\coordinate [label=left:$\mathcal{O}^{+}_{-1}$] (O^{+}_{-1}) at (-5,1);
\coordinate [label=right:$\mathcal{O}^{-}_{-1}$] (O^{-}_{-1}) at (-5.5,-1);
\coordinate [label=right:$\mathcal{Q}^{+}_{-1}$] (Q^{+}_{-1}) at (-8,-1);
\coordinate [label=above:$v_{-1}$] (v_{-1}) at (-5.5,0);
\fill (-5.5,0) circle (1.5pt);
\coordinate [label=right:$\mathcal{Q}^{-}_{-1}$] (Q^{-}_{-1}) at (-4,1);
\coordinate [label=above:$(\mathcal{T}_1)^{\star}_{-1}$] (T_1^{-1}) at (-7,-0.4);
\coordinate [label=above:$(\mathcal{T}_2)^{\star}_{-1}$] (T_2^{-1}) at (-4,-0.4);
\coordinate [label=below:$(\mathcal{T}_1)^{\diamond}_{-1}$] (T_1^{-2}) at (-8,2.4);
\coordinate [label=below:$(\mathcal{T}_2)^{\diamond}_{-1}$] (T_2^{-2}) at (-5,2.4);

\coordinate [label=right:$c_{-1}$] (c_{-1}) at (-6.5,-1);
\coordinate [label=right:$c_{0}$] (c_{0}) at (-1.5,-1);
\end{tikzpicture}
\end{center}
\caption{A combinatorial picture of $\partial U$. The top and bottom curves are identified. $\mathcal{O}_0^{\pm}$ (resp. $\mathcal{O}_{-1}^{\pm}$) is divided by $c_0$ (resp. $c_{-1}$) into a quadrilateral $\mathcal{Q'}_0^{\pm}$ (resp. $\mathcal{Q'}_{-1}^{\pm}$) and a heptagon $\mathcal{H}_{0}^{\pm}$ (resp. $\mathcal{H}_{-1}^{\pm}$). $v_0$ is the intersection of $c_0$ with the arc $[p_4,p_6]$. $v_{-1}$ is the intersection of $c_{-1}$ with the arc $[T^{-1}(p_4),T^{-1}(p_6)]$.}
\label{figure:cylinder}
\end{figure}

\begin{prop}\label{prop:side-star}
The interior of side $\tilde{s}_0^{\star}$ is a union of two disjoint triangles, denoted by $(\mathcal{T}_{1})_0^{\star}$ and $(\mathcal{T}_{2})_0^{\star}$,
whose vertices are $p_2$, $p_5$, $p_6$ and respectively, $p_3$, $p_4$, $p_7$.
\end{prop}
\begin{proof}
By Proposition \ref{prop:sides}, the side $\tilde{s}_0^{\star}$ is the union of two disjoint discs, which are bounded by the ideal boundary of the ridges $s_{0}^{+}  \cap s_{0}^{\star}$ and $s_{0}^{-}  \cap s_{0}^{\star}$.

As stated in Proposition \ref{prop:points}, the ideal boundary of $\mathcal{I}_{0}^{+} \cap \mathcal{I}_{0}^{-} \cap \mathcal{I}_{0}^{\star}$ contains the four points $p_4,p_5,p_6,p_7$. Thus $\tilde{s}_0^{\star}$ is the union of two disjoint bigons whose vertices are $p_5, p_6$, and respectively $p_4,p_7$.
Proposition \ref{prop:points} also tells us that $p_2$ and $p_3$ lie on different component of the boundaries of the two bigons.

Therefore, both of the components are triangles, denoted by $(\mathcal{T}_{1})_0^{\star}$ and $(\mathcal{T}_{2})_0^{\star}$, whose vertices are $p_2$, $p_5$, $p_6$ and respectively, $p_3$, $p_4$, $p_7$.
\end{proof}

According to the symmetry $\tau$ in Proposition \ref{prop:symmetry}, the side $\tilde{s}_0^{\diamond}$ has the same topological properties as the side $\tilde{s}_0^{\star}$.
Thus by a similar argument, we have the following.
\begin{prop}\label{prop:side-diamond}
The interior of side $\tilde{s}_0^{\diamond}$ is a union of two disjoint triangles, denoted by $(\mathcal{T}_{1})_0^{\diamond}$ and $(\mathcal{T}_{2})_0^{\diamond}$,
whose vectors are $q_2$, $p_9$, $p_{11}$ and respectively, $q_3$, $p_8$, $p_{10}$.
\end{prop}

Let $U$ be the ideal boundary of $D$ on $\partial \hc$. Then the union of all the sides $\{\tilde{s}_{k}^{+}\}$, $\{\tilde{s}_{k}^{-}\}$, $\{\tilde{s}_{k}^{\star}\}$ and $\{\tilde{s}_{k}^{\diamond}\}$ for $k\in\mathbb{Z}$ form the boundary of $U$.

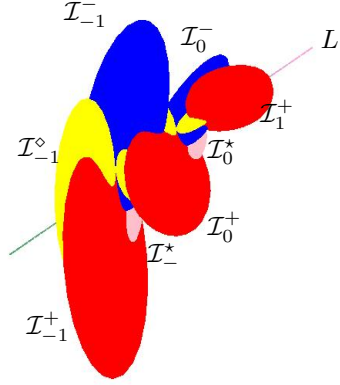
\begin{figure}
\begin{center}
\begin{tikzpicture}
\node at (0,0) {\includegraphics[width=8cm,height=8cm]{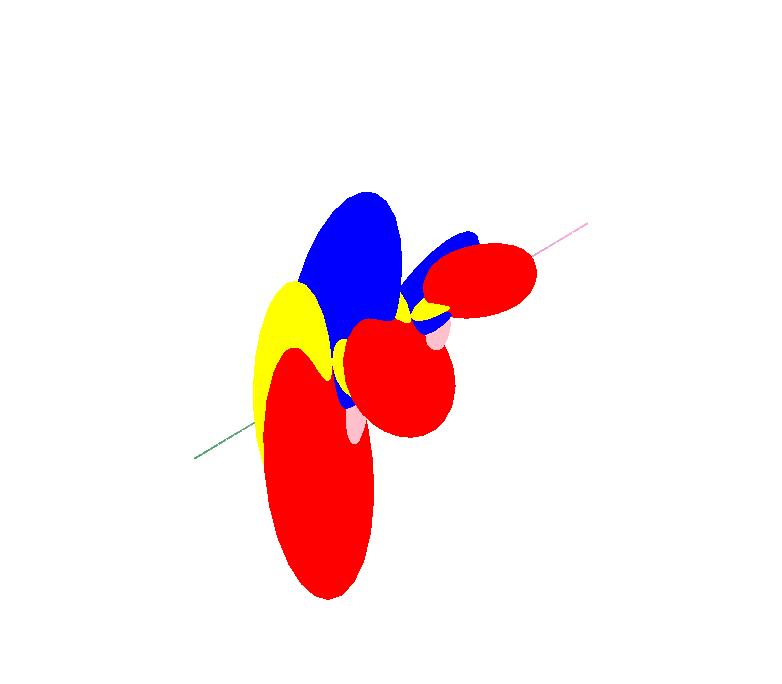}};
\coordinate [label=left:$\mathcal{I}_{-1}^{\diamond}$] (S) at (-1.2,0);

\coordinate [label=right:$L$] (S) at (2,1.5);

\coordinate [label=above:$\mathcal{I}_{-1}^{-}$] (q3) at (-1,1.5);
\coordinate [label=above:$\mathcal{I}_{0}^{-}$] (q3) at (0.5,1.2);

\coordinate [label=right:$\mathcal{I}_{0}^{\star}$] (q3) at (0.5,0);
\coordinate [label=right:$\mathcal{I}_{-}^{\star}$] (q3) at (-0.3,-1.4);

\coordinate [label=right:$\mathcal{I}_0^{+}$] (v1) at (0.5,-1);
\coordinate [label=below:$\mathcal{I}_{-1}^{+}$] (v1) at (-1.5,-2);
\coordinate [label=right:$\mathcal{I}_1^{+}$] (v1) at (1.2,0.5);
\end{tikzpicture}
\end{center}
  \caption{ A realistic picture of the ideal boundaries of the isometric spheres: $\mathcal{I}_{0}^{+}$, $\mathcal{I}_{1}^{+}$ and $\mathcal{I}_{-1}^{+}$ (red); $\mathcal{I}_{0}^{-}$ and $\mathcal{I}_{-1}^{-}$ (blue); $\mathcal{I}_{0}^{\star}$ and $\mathcal{I}_{-1}^{\star}$ (pink); $\mathcal{I}_{0}^{\diamond}$, $\mathcal{I}_{-1}^{\diamond}$ and $\mathcal{I}_{1}^{\diamond}$ (yellow). The line $L$ is the $T$-invariant $\mathbb{R}$-circle.}
  \label{figure:fd}
\end{figure}

\begin{prop}\label{prop:r-circle}
Let $L=\{ [x+i\sqrt{3}/2, \sqrt{3}x]\in\mathcal{N} : x\in \mathbb{R} \}$, then $L$ is a $T$-invariant $\mathbb{R}$-circle. Furthermore, $L$ is contained in the complement of $D$.
\end{prop}
\begin{proof}
It is obvious that $L$ is a $\mathbb{R}$-circle, since it is the image of the $x$-axis of $\mathcal{N}$ by a Heisenberg translation along the $y$-axis.
For any point $[x+i\sqrt{3}/2, \sqrt{3}x]\in L$, we have
$T([x+i\sqrt{3}/2, \sqrt{3}x])=[(x+2)+i\sqrt{3}/2, \sqrt{3}(x+2)]$ which lies in $L$.
 Thus $L$ is a $T$-invariant $\mathbb{R}$-circle. See Figure \ref{figure:fd}.

Note that $T$ acts on $L$ as a translation through $2$. To show $L$ is contained in the complement of $D$, it suffices to show that a segment with length $2$ is contained in the interior of some isometric spheres.  By considering their Cygan distance between a point in $L$ and the center of a isometric sphere,
one can compute that the segment $\{ [x+i\sqrt{3}/2, \sqrt{3}x] : -1/2 \leq x \leq 1/2 \}$ lie in the interior of $\mathcal{I}_0^{+}$ and
the segment $\{ [x+i\sqrt{3}/2, \sqrt{3}x] : 1/2 \leq x \leq 3/2 \}$ lie in the interior of $\mathcal{I}_0^{-}$.
\end{proof}

\begin{defn}
Let $\Sigma_{-1}=\{ [-3/2+iy,t]\in\mathcal{N} : y,t \in\mathbb{R} \}$ and $\Sigma_{0}=\{ [1/2+iy,t]\in\mathcal{N} : y,t \in\mathbb{R} \}$ be two planes in the Heisenberg group.
\end{defn}
In fact, the vertical planes $\Sigma_{-1}$ and $\Sigma_{0}$ are boundaries of {\it fans} in the sense of \cite{Go-P2}. Let $D_T$ be the region between $\Sigma_{-1}$ and $\Sigma_0$, that is
$$
D_T=\{ [x+iy,t]\in\mathcal{N}: -3/2\leq x\leq 1/2 \}.
$$
It is obvious that $\Sigma_0=T(\Sigma_{-1})$. Thus $D_T$ is a fundamental domain for $\langle T \rangle$ acting on $\partial\hc$.

\begin{prop}\label{prop:plane-sigma}
The intersections of $\Sigma_0$ and $\Sigma_{-1}$ with the isometric spheres $\mathcal{I}_k^{\pm}$, $\mathcal{I}_k^{\star}$ and $\mathcal{I}_k^{\diamond}$ are empty, except the following:
\begin{itemize}
  \item Each one of $\Sigma_0 \cap \mathcal{I}_0^{\pm}$ and $\Sigma_0 \cap \mathcal{I}_0^{\star}$ is a circle and $\Sigma_0 \cap \mathcal{I}_0^{\diamond}=\Sigma_0 \cap \mathcal{I}_1^{\diamond}=\{q_3\}$.
  \item Each one of $\Sigma_{-1}\cap \mathcal{I}_{-1}^{\pm}$ and $\Sigma_{-1} \cap \mathcal{I}_{-1}^{\star}$ is a circle and $\Sigma_{-1} \cap \mathcal{I}_{-1}^{\diamond}=\Sigma_{-1} \cap \mathcal{I}_0^{\diamond}=\{q_2\}$.
\end{itemize}
\end{prop}
\begin{proof}
Since the isometric spheres are strictly convex, their intersections with a plane is either a topological circle, or a point or empty. Note that $\Sigma_0=T(\Sigma_{-1})$. Thus it suffices to consider the intersections of $\Sigma_0$ with the isometric spheres. By a strait computation, each one of $\Sigma_0 \cap \mathcal{I}_0^{\pm}$ and $\Sigma_0 \cap \mathcal{I}_0^{\star}$ is a circle. (See Figure \ref{figure:curves}).
\end{proof}

\begin{lem}\label{lem:curve-sigma}
The plane $\Sigma_0$ (respectively $\Sigma_{-1}$) is preserved by $I_2$ (respectively $T^{-1}I_2T$).
The intersection $\Sigma_{0} \cap \partial U$ (respectively $\Sigma_{-1} \cap \partial U$) is a simple closed curve $c_0$ (respectively $c_{-1}$)
in the union $ \tilde{s}_0^{+} \cup \tilde{s}_0^{-}$ (respectively $\tilde{s}_{-1}^{+} \cup \tilde{s}_{-1}^{-}$), which contains the points $q_3$ and $v_0=[1/2+i\sqrt{3}/2,-\sqrt{3}]$ (respectively $q_2$ and $v_{-1}=T^{-1}(v_0)$).
\end{lem}
\begin{proof}
It suffices to consider $\Sigma_{0}$. The $\mathbb{C}$-circle associated to $I_2$, that is the ideal boundary of the complex line fixed by $I_2$, is $\{[1/2+i\sqrt{3}/2,t]\in \mathcal{N} : t\in\mathbb{R} \}$, which
is contained in $\Sigma_{0}$. Thus $\Sigma_0$ is preserved by $I_2$.

It is obvious that $ \Sigma_{0}$ contains $q_3$, which is the tangent point of $\mathcal{I}_{0}^{\diamond}$ and $\mathcal{I}_{1}^{\diamond}$.
The intersections $\Sigma_{0} \cap \mathcal{I}_{0}^{+}$, $\Sigma_{0} \cap \mathcal{I}_{0}^{-}$ and $\Sigma_{0} \cap \mathcal{I}_{0}^{\star}$ are circles by Proposition \ref{prop:plane-sigma}. One can compute that the intersection $\Sigma_{0} \cap \mathcal{I}_{0}^{+} \cap \mathcal{I}_{0}^{-}$ contain two points $q_3$ and $v_0=[1/2+i\sqrt{3}/2,-\sqrt{3}]$. See Figure \ref{figure:curves}. These two points divide the circles on $\mathcal{I}_{0}^{+}$ and $\mathcal{I}_{0}^{-}$ into two arcs.
Let $c_0^{+}$ be the arc with endpoints $q_3$ and $v_0$ on $\mathcal{I}_{0}^{+}$ lying in the exterior of $\mathcal{I}_{0}^{-}$ and $c_0^{-}$ be the one on $\mathcal{I}_{0}^{-}$ lying in the exterior of $\mathcal{I}_{0}^{+}$. Then $c_0=c_0^{+}\cup c_0^{-}$ is a simple closed curve.
Observe that $\Sigma_{0} \cap \mathcal{I}_{0}^{\star}$ lie in the union of the interiors of $\mathcal{I}_{0}^{+}$ and $\mathcal{I}_{0}^{-}$.
Thus $\Sigma_0$ does not intersect $\tilde{s}_0^{\star}$.
By Proposition \ref{prop:plane-sigma}, $c_0^{+}$ lie on $\tilde{s}_0^{+}$ and $c_0^{-}$ lie on $\tilde{s}_0^{-}$.
Therefore, the intersection $\Sigma_{0} \cap \partial U$  is $c_0$, which is a simple closed curve containing $q_3$ and $v_0$.
\end{proof}

\begin{figure}
\begin{center}
\begin{tikzpicture}
\node at (0,0) {\includegraphics[width=5cm,height=5cm]{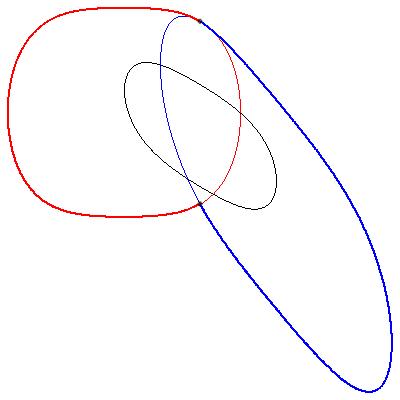}};
\coordinate [label=left:$\Sigma_{0} \cap \mathcal{I}_{0}^{+}$] (S) at (-2.5,1);
\coordinate [label=right:$\Sigma_{0} \cap \mathcal{I}_{0}^{-}$] (S) at (2.4,-1.8);
\coordinate [label=right:$\Sigma_{0} \cap \mathcal{I}_{0}^{\star}$] (S) at (0.2,-0.3);
\coordinate [label=left:$c_{0}^{+}$] (S) at (-2,0);
\coordinate [label=right:$c_{0}^{-}$] (S) at (2,0);
\coordinate [label=above:$q_3$] (q3) at (0,2.2);
\coordinate [label=below:$v_0$] (v1) at (0,-0.2);
\end{tikzpicture}
\end{center}
\caption{The intersection of $\Sigma_{0}$ with $\mathcal{I}_{0}^{+}$, $\mathcal{I}_{0}^{-}$ and $\mathcal{I}_{0}^{\star}$. Viewed in $\Sigma_{0}$. $c_0=c_0^{+} \cup c_0^{-}$ is a simple closed curve, where $c_0^{+}$ and $c_0^{-}$ are the thicker parts of $\Sigma_{0} \cap \mathcal{I}_{0}^{+}$ and $\Sigma_{0} \cap \mathcal{I}_{0}^{-}$, respectively. }
  \label{figure:curves}
\end{figure}
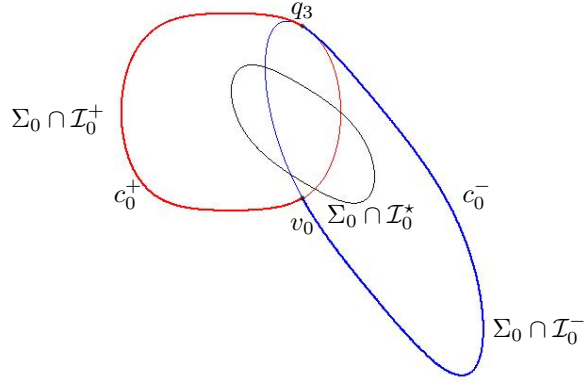

Let $U^{c}$ be closure of the complement of $U$ in $\mathcal{N}$.

\begin{prop}\label{prop:solid-tube}
The closure of the intersection $U^{c}\cap D_T$ is a solid tube homeomorphic to a 3-ball.
\end{prop}
\begin{proof}
It suffices to show that the boundary of $U^{c}\cap D_T$ is a 2-sphere. Now let us consider the cell structure of $U^{c}\cap D_T$. See Figure \ref{figure:cylinder}.
According to Lemma \ref{lem:curve-sigma}, the intersection of $U^{c}$ with $\Sigma_0$ (resp. $\Sigma_{-1}$) is a topological disc with two vertices $q_3$ and $v_0$ (resp. $q_2$ and $v_{-1}$)and two edges $c_0^{\pm}$ (resp. $c_{-1}^{\pm}$). Moreover, $c_0$ (resp. $c_{-1}$) divides $\mathcal{O}_0^{\pm}$ (resp. $\mathcal{O}_{-1}^{\pm}$)into a quadrilateral $\mathcal{Q'}_0^{\pm}$ (resp. $\mathcal{Q'}_{-1}^{\pm}$) and a heptagon $\mathcal{H}_{0}^{\pm}$ (resp. $\mathcal{H}_{-1}^{\pm}$).

Since $p_2$, $p_5$ and $T^{-1}(p_4)$ are contained in $D_T$, one can see that $D_T$ contains $\mathcal{Q'}_0^{-}$, $\mathcal{Q'}_{-1}^{+}$, $\mathcal{H}_{0}^{+}$ and $\mathcal{H}_{-1}^{-}$.
Besides, $D_T$ contains $\mathcal{Q}_0^{+}$, $\mathcal{Q}_{-1}^{-}$, $(\mathcal{T}_1)_0^{\diamond}$, $(\mathcal{T}_2)_0^{\diamond}$, $(\mathcal{T}_1)_0^{\star}$ and $(\mathcal{T}_2)_{-1}^{\star}$.
Thus the boundary of $U^{c}\cap D_T$ consists of 12 faces, 23 edges and 13 vertices. See the region between $c_0$ and $c_{-1}$ in Figure \ref{figure:cylinder}.
Therefore the Euler characteristic of the boundary of $U^{c}\cap D_T$ is 2. So the boundary of $U^{c}\cap D_T$ is a 2-sphere.
\end{proof}

Proposition \ref{prop:r-circle} and Proposition \ref{prop:solid-tube} imply the following result.
\begin{prop}\label{prop:cylinder}
$U \cap D_T$ is an unknotted cylinder cross a ray homeomorphic to $S^{1}\times [0,1]\times \mathbb{R}_{\geq0}$.
\end{prop}
\begin{proof}
As stated in Proposition \ref{prop:r-circle}, $U^{c}$ contains the line $L$. Thus $U^{c}\cap D_T$ is a tubular neighborhood of $L \cap D_T$. It cannot be knotted.
Therefore $\partial U \cap D_T$ is un unknotted cylinder homeomorphic to $S^{1}\times [0,1]$.
One can see that $U\cap\Sigma_0$ is $c_0$ cross a ray and $U\cap\Sigma_{-1}$ is $c_{-1}$ cross a ray. Both of them are homeomorphic to $S^{1}\times\mathbb{R}_{\geq 0}$. Hence $U \cap D_T$ is an unknotted cylinder cross a ray homeomorphic to $S^{1}\times [0,1]\times \mathbb{R}_{\geq0}$.
\end{proof}

Applying the powers of $T$, Proposition \ref{prop:cylinder} immediately implies the following corollary.
\begin{cor}
$U$ is an unknotted cylinder cross a ray homeomorphic to $S^{1}\times \mathbb{R}\times \mathbb{R}_{\geq0}$.
\end{cor}

\begin{rem}
$U$ is the complement of a tubular neighborhood of the $T$-invariant $\mathbb{R}$-circle $L$. That is a horotube for $T$ (See \cite{schwartz-cr} for the definition of horotube).
\end{rem}

\begin{defn}
Suppose that the cylinder $S^{1}\times [0,1]$ has a combinatorial cell structure with finite faces $\{F_i\}$.
A \emph{canonical subdivision} on $S^{1}\times [0,1] \times \mathbb{R}_{\geq0}$ is a finite union of 3-dimensional pieces $\{\widehat{F_i}\}$ where $\widehat{F_i}=F_{i}\times \mathbb{R}_{\geq 0}$.
\end{defn}

\begin{prop}\label{prop:subdivision}
There is a canonical subdivision on $U \cap D_T$.
\end{prop}
\begin{proof}
As described in the proof of Proposition \ref{prop:solid-tube}, the combinatorial cell structure of $\partial U \cap D_T$ has 10 faces $\mathcal{Q'}_0^{-}$, $\mathcal{Q'}_{-1}^{+}$, $\mathcal{H}_{0}^{+}$, $\mathcal{H}_{-1}^{-}$, $\mathcal{Q}_0^{+}$, $\mathcal{Q}_{-1}^{-}$, $(\mathcal{T}_1)_0^{\diamond}$, $(\mathcal{T}_2)_0^{\diamond}$, $(\mathcal{T}_1)_0^{\star}$ and $(\mathcal{T}_2)_{-1}^{\star}$. By Proposition \ref{prop:cylinder}, $U \cap D_T$ is the union of 3-dimensional pieces $\widehat{\mathcal{Q'}_0^{-}}$, $\widehat{\mathcal{Q'}_{-1}^{+}}$, $\widehat{\mathcal{H}_{0}^{+}}$, $\widehat{\mathcal{H}_{-1}^{-}}$, $\widehat{\mathcal{Q}_0^{+}}$, $\widehat{\mathcal{Q}_{-1}^{-}}$, $\widehat{(\mathcal{T}_1)_0^{\diamond}}$, $\widehat{(\mathcal{T}_2)_0^{\diamond}}$, $\widehat{(\mathcal{T}_1)_0^{\star}}$ and $\widehat{(\mathcal{T}_2)_{-1}^{\star}}$. Combinatorially, these 3-dimensional pieces are the cone from $q_{\infty}$ to the faces of $\partial U \cap D_T$.
\end{proof}

Let $\Omega(\Gamma)$ be the discontinuity region of $\Gamma$ acting on $\partial\hc$. Then $U \cap D_T$ is obvious a fundamental domain for $\Gamma$. By cutting and gluing, we can obtain the following fundamental domain for $\Gamma$ acting on $\Omega(\Gamma)$.

\begin{prop}\label{prop:polyhedra}
Let $\mathcal{P}_{+}$ be the union of $\widehat{\mathcal{H}_{0}^{+}}$, $\widehat{(\mathcal{T}_1)_0^{\diamond}}$, $\widehat{(\mathcal{T}_2)_0^{\diamond}}$, $T(\widehat{\mathcal{Q'}_{-1}^{+}})$, $T(\widehat{\mathcal{Q}_{-1}^{-}})$ and $T(\widehat{(\mathcal{T}_2)_{-1}^{\star}})$.
Let $\mathcal{P}_{-}$ be the union of $\widehat{\mathcal{Q}_0^{+}}$, $\widehat{(\mathcal{T}_1)_0^{\star}}$, $\widehat{\mathcal{Q'}_0^{-}}$ and $T(\widehat{\mathcal{H}_{-1}^{-}})$.
Then $\mathcal{P}_{+} \cup \mathcal{P}_{-}$ is a fundamental domain  for $\Gamma$ acting on $\Omega(\Gamma)$.
Moreover, $\mathcal{P}_{+}$ (resp. $\mathcal{P}_{-}$) is combinatorially an eleven pyramid (resp. nine pyramid) with cone vertex $q_{\infty}$ and base $\mathcal{O}_0^{+} \cup \mathcal{Q}_0^{-} \cup (\mathcal{T}_1)_0^{\diamond} \cup (\mathcal{T}_2)_0^{\diamond} \cup (\mathcal{T}_{2})_0^{\star}$ (resp. $\mathcal{O}_0^{-} \cup \mathcal{Q}_0^{+} \cup (\mathcal{T}_{1})_0^{\star}$).
\end{prop}
\begin{proof}
Since $\Sigma_0=T(\Sigma_{-1})$ and $c_0=T(c_{-1})$, $U\cap D_T$ and $T(U\cap D_T)$ can be glued together along $c_0\times\mathbb{R}_{\geq 0}$.
Note that $U \cap D_T$ is a fundamental domain  for $\Gamma$ acting on $\Omega(\Gamma)$ and has a subdivision as described in Proposition \ref{prop:subdivision}.
Therefore $\mathcal{P}_{+} \cup \mathcal{P}_{-}$ is also a fundamental domain.

As described in Proposition \ref{prop:solid-tube}, $c_0$ (resp. $c_{-1}$) divides $\mathcal{O}_0^{\pm}$ (resp. $\mathcal{O}_{-1}^{\pm}$)into a quadrilateral $\mathcal{Q'}_0^{\pm}$ (resp. $\mathcal{Q'}_{-1}^{\pm}$) and a heptagon $\mathcal{H}_{0}^{\pm}$ (resp. $\mathcal{H}_{-1}^{\pm}$).
Note that $\mathcal{O}_0^{\pm}=T(\mathcal{O}_{-1}^{\pm})$ and $(\mathcal{T}_2)_0^{\star}=T((\mathcal{T}_2)_{-1}^{\star})$.
Thus the base of $\mathcal{P}_{+}$ (resp. $\mathcal{P}_{-}$) is $\mathcal{O}_0^{+} \cup \mathcal{Q}_0^{-} \cup (\mathcal{T}_1)_0^{\diamond} \cup (\mathcal{T}_2)_0^{\diamond} \cup (\mathcal{T}_{2})_0^{\star}$ (resp. $\mathcal{O}_0^{-} \cup \mathcal{Q}_0^{+} \cup (\mathcal{T}_{1})_0^{\star}$) which is combinatorially a hendecagon (resp. an enneagon).
See Figure \ref{figure:cylinder} and Figure \ref{figure:poly0}.
\end{proof}

\begin{figure}
\begin{center}
\begin{tikzpicture}[scale=0.4]
\tikzstyle{every node}=[font=\large,scale=1.0]

\path (0,6) coordinate (p11);
\path (3,6) coordinate (p8);
\path (4,4) coordinate (q3);
\path (3,0) coordinate (p7);
\path (0,-3) coordinate (p4);
\path (-3,-3) coordinate (p6);
\path (-4,0) coordinate (p2);
\path (-3,3) coordinate (p9);
\path (-6.3,-6) coordinate (p5);
\path (-7,-9) coordinate (q33);
\path (-6,-12) coordinate (p14);
\path (-3,-12) coordinate (p13);
\path (0,-9) coordinate (p15);
\path (0.6,-5) coordinate (p3);

\path (-2.2,5.5) coordinate (q2);
\path (5,5) coordinate (p10);
\path (6,1) coordinate (p12);
\path (-9,-5) coordinate (p10p);
\path [fill=yellow] (p2)-- (p5)--(p6);
\path [fill=yellow] (p2)-- (p10p)--(q33);
\path [fill=yellow] (p6)--(p5)--(q33)--(p14)--(p13)--(p15)--(p3)--(p4);

\draw [dashed,ultra thick](p9)--(p11)--(p8) -- (q3) -- (p7) -- (p4) -- (p6)--(p2)--cycle;
\draw [dashed,ultra thick](p9) -- (q2) -- (p11);
\draw [dashed,ultra thick](p10) -- (p8);
\draw [dashed,ultra thick](p8) -- (q3);
\draw [ultra thick](p10) --(q3);
\draw [dashed,ultra thick](p10p) --(p2)--(p5)-- (q33);
\draw [ultra thick](p10p) -- (q33);
\draw [ultra thick](q3) -- (p12) -- (p3);
\draw [dashed,ultra thick](p3) -- (p7);
\draw [dashed,ultra thick](q33)--(p5) -- (p6);
\draw [ultra thick](q33) -- (p14) -- (p13) -- (p15)--(p3);
\draw [dashed,ultra thick](p3)--(p4);

\draw [dashed,ultra thick] (0,6) -- (0,10);
\draw [dashed, ultra thick] (3,6) -- (3,10);
\draw [ ultra thick] (4,4) -- (4,8);
\draw [dashed, ultra thick] (3,0) -- (3,4);
\draw [dashed, ultra thick] (0,-3) -- (0,1);
\draw [dashed, ultra thick] (-3,-3) -- (-3,1);
\draw [dashed, ultra thick] (-4,0) -- (-4,4);
\draw [dashed, ultra thick] (-3,3) -- (-3,7);
\draw [dashed, ultra thick] (-6.3,-6) -- (-6.3,-1);
\draw [ ultra thick] (-7,-9) -- (-7,-5);
\draw [ ultra thick] (-6,-12) -- (-6,-8);
\draw [ ultra thick] (-3,-12) -- (-3,-8);
\draw [ ultra thick] (0,-9) -- (0,-5);
\draw [ ultra thick] (0.6,-5) -- (0.6,-1);
\draw [dashed, ultra thick] (-2.2,5.5) -- (-2.5,9.5);
\draw [ ultra thick] (5,5) -- (5,9);
\draw [ ultra thick] (6,1) -- (6,5);
\draw [ ultra thick] (-9,-5) -- (-9,-1);

\node [above right] (p11) at (0,6){$ \mathbf{p_{11}}$};
\draw [fill] (0,6) circle [radius=.2];
\node [below left] (p8) at (3,6){$ \mathbf{p_{8}}$};
\draw [fill] (3,6) circle [radius=.2];
\node [ right] (q3) at (4,4) {$ \mathbf{q_{3}}$};
\draw [fill=black] (4,4)  circle [radius=.2];
\node [above right] (p7) at (3,0){$ \mathbf{p_{7}}$};
\draw [fill] (3,0) circle [radius=.2];
\node [above left] (p4) at (0,-3){$ \mathbf{p_{4}}$};
\draw [fill] (0,-3) circle [radius=.2];

\node [above left] (p6) at (-3,-3){$ \mathbf{p_{6}}$};
\draw [fill] (-3,-3) circle [radius=.2];

\node [above left] (p2) at (-4,0){$ \mathbf{p_{2}}$};
\draw [fill=red] (-4,0) circle [radius=.2];
\node [right] (p9) at (-3,3){$ \mathbf{p_{9}}$};
\draw [fill] (-3,3) circle [radius=.2];

\node [right] (p5) at (-6.3,-6){$ \mathbf{p_{5}}$};
\draw [fill] (-6.3,-6) circle [radius=.2];

\node [left] (q33) at (-7,-9){$ \mathbf{q_{3}}$};
\draw [fill=red] (-7,-9) circle [radius=.2];

\node [left] (p14) at (-6,-12){$ \mathbf{p_{14}}$};
\draw [fill] (-6,-12) circle [radius=.2];

\node [below] (p13) at (-3,-12){$ \mathbf{p_{13}}$};
\draw [fill] (-3,-12) circle [radius=.2];

\node [right] (p15) at (0,-9){$ \mathbf{p_{15}}$};
\draw [fill] (0,-9) circle [radius=.2];

\node [right] (p3) at (0.6,-5) {$ \mathbf{p_{3}}$};
\draw [fill=red] (0.6,-5)  circle [radius=.2];

\node [above right] (q2) at (-2.2,5.5) {$ \mathbf{q_{2}}$};
\draw [fill=red] (-2.2,5.5)  circle [radius=.2];

\node [above right] (p10) at (5,5) {$ \mathbf{p_{10}}$};
\draw [fill] (5,5)  circle [radius=.2];

\node [above right] (p12) at (6,1) {$ \mathbf{p_{12}}$};
\draw [fill] (6,1)  circle [radius=.2];

\node [left] (p10p) at (-9,-5) {$ \mathbf{p_{10}}$};
\draw [fill] (-9,-5)  circle [radius=.2];

\end{tikzpicture}

\end{center}
\caption{A schematic view of the fundamental domain of $\Gamma$ on $\Omega(\Gamma)$. The vertices colored with red  are the parabolic fixed points. The polygon colored with yellow is $\mathcal{O}_0^{-} \cup \mathcal{Q}_0^{+} \cup (\mathcal{T}_{1})_0^{\star}$.}
\label{figure:poly0}
\end{figure}
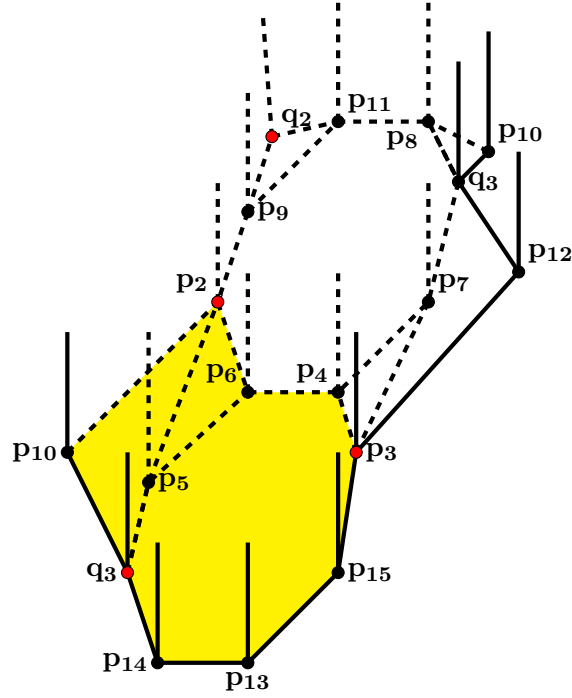

\begin{defn}
Let $p'_2=S^{-1}(p_2)$, $p_1=S^{-1}(q_{\infty})=[0,0]$ and $p'_{10}=S^{-1}(p_{10})$.
\end{defn}

\begin{lem}\label{lem:gluing}
Let $\mathcal{P}$ be the union $\mathcal{P}_{+} \cup S^{-1}(\mathcal{P}_{-1})$.
Then $\mathcal{P}$ is combinatorially a polyhedron with 8 triangular faces, 4 square faces, 2 pentagonal faces and 2 hexagonal faces.
The faces of $\mathcal{P}$ are paired as follows:
\begin{eqnarray*}
T: & (q_{\infty},p_2, p_9, q_2)\longmapsto (q_{\infty}, p_3, p_{12}, q_3), \\
S^{-1}T: & (q_{\infty}, q_2, p_{11}, p_8, p_{10}) \longmapsto (p_1, p_2, p_9, p_{11}, p_8),\\
(S^{-1}T)^2: & (q_2, p_9, p_{11}) \longmapsto (q_3, p_8, p_{10}),\\
S^{-1}: & (q_{\infty}, p_{10}, q_3) \longmapsto (p_1, p'_{10}, p_2),\\
S^{-1}T^{-1}S: & (p_1, p_8, q_3) \longmapsto (p_1, p'_{10}, p'_2),\\
S^{-2}: & (q_3, p_{12}, p_3, p_7) \longmapsto (p'_2, p'_{10}, p_2, p_6),\\
S^{-1}: & (q_{\infty}, p_2, p_6, p'_2, p_4, p_3) \longmapsto (p_1, p'_2, p_4, p_3, p_7, q_3),\\
S^{-1}: & ( p_6, p'_2, p_4) \longmapsto (p_4, p_3, p_7).
\end{eqnarray*}
\end{lem}
\begin{proof}
The bases of $\mathcal{P}_{+} $ and $\mathcal{P}_{-1}$ are paired as follows:
\begin{eqnarray*}
  S^{-1}:& \mathcal{O}_{0}^{-} \longrightarrow \mathcal{O}_{0}^{+} \\
         & (p_3, p_4, p_6, p_5, q_3, q_{14}, p_{13}, p_{15}) \longmapsto (q_3, p_7, p_4, p_6, p_2, p_9, p_{11},p_8), \\
  S : & \mathcal{Q}_0^{+} \longrightarrow \mathcal{Q}_0^{-} \\
      & (p_2, p_5, q_3, p_{10}) \longmapsto (q_3, p_7, p_3, p_{12}), \\
  S^2 : & (\mathcal{T}_{1})_0^{\star} \longrightarrow (\mathcal{T}_{2})_0^{\star} \\
        & (p_2,p_5,p_6) \longmapsto (p_3, p_4, p_7), \\
  (S^{-1}T)^2 : & (\mathcal{T}_{1})_0^{\diamond} \longrightarrow (\mathcal{T}_{2})_0^{\diamond} \\
                 & (q_2, p_9, p_{11}) \longmapsto (q_3, p_8, p_{10}).
\end{eqnarray*}
Thus $S^{-1}(\mathcal{P}_{-1})$ and $\mathcal{P}_{+} $ are glued along $\mathcal{O}_0^{+}$.
According to Lemma \ref{lem:goldman}, $S^{-1}(\mathcal{P}_{-1})$ lie in the interior of $\mathcal{I}_0^{+}$, since $\mathcal{P}_{-1}$ lie in the exterior of $\mathcal{I}_0^{+}$.
Moreover, $p_1=S^{-1}(q_{\infty})=[0,0]$ is the center of the isometric sphere $\mathcal{I}_0^{+}$.
See Figure \ref{figure:poly1}.
\end{proof}

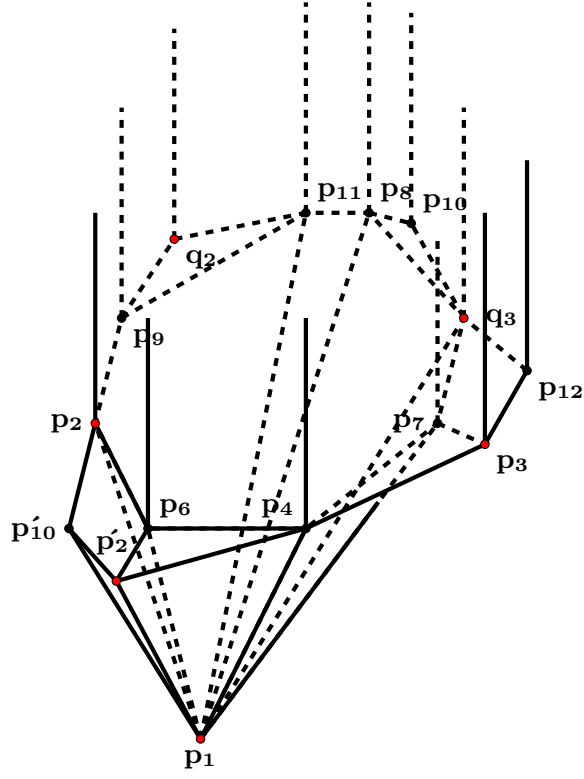
\begin{figure}
\begin{center}
\begin{tikzpicture}[scale=0.7]
\tikzstyle{every node}=[font=\small,scale=1.2]
\path (1,3) coordinate (p11);
\path (2.2,3) coordinate (p8);
\path (4,1) coordinate (q3);
\path (3.5,-1) coordinate (p7);
\path (1,-3) coordinate (p4);
\path (-2,-3) coordinate (p6);
\path (-3,-1) coordinate (p2);
\path (-2.5,1) coordinate (p9);
\path (-1.5,2.5) coordinate (q2);
\path (3,2.8) coordinate (p10);
\path (-1,-7) coordinate (p1);
\path (-2.6,-4) coordinate (p2p);
\path (-3.5,-3) coordinate (p10p);
\path (5.2,0) coordinate (p12);
\path (4.4,-1.4) coordinate (p3);

\draw [dashed,ultra thick](p9)--(q2)--(p11);
\draw [dashed,ultra thick](p8)--(p10)--(q3);
\draw [dashed,ultra thick](p9)--(p11)--(p8) -- (q3) -- (p7) -- (p4) -- (p6)--(p2)--cycle;
\draw [ultra thick](p2)--(p10p)--(p2p)--(p4);
\draw [ultra thick](p4)--(p3)--(p12);
\draw [dashed,ultra thick](p12)--(q3);
\draw [dashed,ultra thick](p7)--(p3);
\draw [ultra thick](p6)--(p2p);
\draw [ultra thick](p6)--(p2);
\draw [ultra thick](p6)--(p4);

\draw [dashed,ultra thick](1,3)--(1,7);
\draw [dashed,ultra thick](2.2,3)--(2.2,7);
\draw [dashed,ultra thick](4,1)--(4,5);
\draw [dashed,ultra thick](3.5,-1)--(3.5,2.5);
\draw [ultra thick](1,-3)--(1,1);
\draw [ultra thick](-2,-3)--(-2,1);
\draw [ultra thick](-3,-1)--(-3,3);
\draw [dashed,ultra thick](-2.5,1)--(-2.5,5);
\draw [ultra thick](5.2,0)--(5.2,4);
\draw [ultra thick](4.4,-1.4)--(4.4,3);
\draw [dashed,ultra thick](-1.5,2.5)--(-1.5,6.5);
\draw [dashed,ultra thick](3,2.8)--(3,6.8);
\draw [ultra thick](-1,-7)--(-2.6,-4);
\draw [ultra thick](-1,-7)--(-3.5,-3);
\draw [ultra thick](-1,-7)--(1,-3);
\draw [dashed,ultra thick](-1,-7)--(-2,-3);
\draw [dashed,ultra thick](-1,-7)--(-3,-1);
\draw [dashed,ultra thick](-1,-7)--(1,-3);
\draw [dashed,ultra thick](-1,-7)--(1,3) ;
\draw [dashed,ultra thick](-1,-7)--(2.2,3) ;
\draw [dashed,ultra thick](-1,-7)--(4,1) ;
\draw [dashed,ultra thick](-1,-7)--(3.5,-1)  ;
\draw [ultra thick](-1,-7)--(-1+4.5*0.75,-7+6*0.75)  ;

\node [above right] (p11) at (1,3){$ \mathbf{p_{11}}$};
\draw [fill] (1,3) circle [radius=.08];
\node [above right] (p8) at (2.2,3){$ \mathbf{p_{8}}$};
\draw [fill] (2.2,3) circle [radius=.08];
\node [right ] (q3) at (4.2,1){$ \mathbf{q_{3}}$};
\draw [fill=red] (4,1) circle [radius=.08];
\node [left ] (p7) at (3.5,-1){$ \mathbf{p_{7}}$};
\draw [fill] (3.5,-1) circle [radius=.08];
\node [above left ] (p4) at (1,-3){$ \mathbf{p_{4}}$};
\draw [fill] (1,-3) circle [radius=.08];
\node [above right ] (p6) at (-2,-3){$ \mathbf{p_{6}}$};
\draw [fill] (-2,-3) circle [radius=.08];
\node [left ] (p2) at (-3,-1){$ \mathbf{p_{2}}$};
\draw [ fill=red] (-3,-1) circle [radius=.08];
\node [below right ] (p9) at (-2.5,1){$ \mathbf{p_{9}}$};
\draw [fill] (-2.5,1) circle [radius=.08];
\node [below right ] (p12) at (5.2,0){$ \mathbf{p_{12}}$};
\draw [fill] (5.2,0) circle [radius=.08];
\node [below right ] (p3) at (4.4,-1.4){$ \mathbf{p_{3}}$};
\draw [fill=red] (4.4,-1.4) circle [radius=.08];

\node [below right ] (q2) at (-1.5,2.5){$ \mathbf{q_{2}}$};
\draw [fill=red] (-1.5,2.5) circle [radius=.08];

\node [below right] (p2p) at (-3.2,-2.8){$ \mathbf{\acute{p_{2}}}$};
\draw [fill=red] (-2.6,-4) circle [radius=.08];
\node [left] (p10p) at (-3.5,-3){$ \mathbf{\acute{p_{10}}}$};
\draw [fill] (-3.5,-3) circle [radius=.08];
\node [below] (p1) at (-1,-7){$ \mathbf{p_{1}}$};
\draw [fill=red] (-1,-7) circle [radius=.08];

\node  [above right] (p10) at (3,2.8){$ \mathbf{p_{10}}$};
\draw [fill] (3,2.8) circle [radius=.08];

\end{tikzpicture}
\end{center}
\caption{The combinatorial picture of $\mathcal{P}$. The vertices of $\mathcal{P}$ colored with red  are the parabolic fixed points.}
\label{figure:poly1}
\end{figure}

\begin{prop}\label{prop:fundamental-group}
Let $\Omega$ be the discontinuity region of $\Gamma$ acting on $\hc$. Then the fundamental group of $\Omega/\Gamma$ has a presentation
$$
\langle u, v, w | w^{-1}vu^{-1}v^{-1}wu=v^2wuw^{-3}u=id \rangle.
$$
\end{prop}
\begin{proof}
Let $x_i, i=1,2,3,4,5,6,7,8$ be the corresponding gluing maps of $\mathcal{P}$ given in Lemma \ref{lem:gluing}.
These are the generators of the fundamental group of $\Omega/\Gamma$.

By considering the edge cycles of $\mathcal{P}$ under the gluing maps, we have the relations as follows.
\begin{enumerate}
  \item $x_7^{-1}\cdot x_5 \cdot x_7 \cdot x_1 =id,$
  \item $x_2^{-1} \cdot x_4 \cdot x_1 =id,$
  \item $x_2 \cdot x_3^{-1} \cdot x_4^{-1} \cdot x_6 \cdot x_1 =id,$
  \item $x_3^{-1} \cdot x_5^{-1} \cdot x_6 \cdot x_1 =id,$
  \item $x_2 \cdot x_3 \cdot x_2 =id,$
  \item $x_4^{-1} \cdot x_5 \cdot x_2=id,$
  \item $x_7 \cdot x_8 \cdot x_6=id,$
  \item $x_8 \cdot x_7 \cdot x_6=id,$
  \item $x_8^{-1} \cdot x_7=id.$
\end{enumerate}
For example, the edge cycle of $[q_{\infty},p_2]$ is
$$
[q_{\infty},p_2] \xrightarrow{x_1} [q_{\infty},p_3] \xrightarrow{x_7} [p_1,q_3] \xrightarrow{x_5} [p_1,p'_2] \xrightarrow{x_7^{-1}} [q_{\infty},p_2].
$$
Thus
$$
x_7^{-1}\cdot x_5 \cdot x_7 \cdot x_1=id.
$$
This is the relation in (1). The others can be given by a similar argument.

Simplifying the relations and setting $u=x_1, v=x_2, w=x_7$, we obtain the presentation of the fundamental group of $\Omega/\Gamma$.

\end{proof}

Now, we are ready to show the following theorem.
\begin{thm}\label{thm:s782}
Let $\Omega$ be the discontinuity region of $\Gamma$ acting on $\hc$. Then the quotient space $\Omega/\Gamma$ is homeomorphic to the two-cusped hyperbolic 3-manifold $s782$ in the SnapPy census.
\end{thm}
\begin{proof}
Let $M=\Omega/\Gamma$. According to Proposition \ref{prop:fundamental-group},
 the fundamental group of $M$ has a presentation
$$
\pi_1(M)=\langle u, v, w | w^{-1}vu^{-1}v^{-1}wu=v^2wuw^{-3}u=id \rangle.
$$

The manifold $s782$ is a two-cusped hyperbolic 3-manifold with finite volume.
Its fundamental group provided by SnapPy has a presentation
$$
\pi_1(s_{782})=\langle a,b,c | a^2 c b^4 c=abca^{-1}b^{-1}c^{-1}=id \rangle.
$$

Using \texttt{Magma}, we get an isomorphism $\Psi: \pi_1(M)\longrightarrow \pi_1(s782)$ given by
$$
\Psi(u)=c^{-1}b^{-1}, \quad \Psi(v)=b^{-1},\quad \Psi(w)=a.
$$

Therefore $M$ will be the connect sum of  $s782$ and a closed 3-manifold with trivial fundamental group by the prime decompositions of 3-manifolds \cite{Hempel}.  The solution of the
Poincar\'e conjecture implies that the closed 3-manifold is the 3-sphere.  Thus $M$ is homeomorphic to  $s782$.  The proof of Theorem \ref{thm:s782}  is complete.


\end{proof}

\end{document}